\documentclass[letter, twoside, 11pt]{article}

\usepackage{amsmath, amscd, amsfonts, amssymb, amsthm, latexsym, url, color}

\usepackage{graphicx}

\usepackage[left=1.3in, right=1.3in, top=1.3in, bottom=1in, includefoot, headheight=13.6pt]{geometry}

\usepackage[ps2pdf=true,colorlinks,breaklinks=true]{hyperref}
\usepackage[figure,table]{hypcap}
\hypersetup{
	bookmarksnumbered,
	pdfstartview={FitH},
	citecolor={blue},
	linkcolor={red},
	urlcolor={black},
	pdfpagemode={UseOutlines}
}
\makeatletter
\newcommand\org@hypertarget{}
\let\org@hypertarget\hypertarget
\renewcommand\hypertarget[2]{%
  \Hy@raisedlink{\org@hypertarget{#1}{}}#2%
} 
\makeatother

\input{xy}
\xyoption{all}
\newtheorem{theorem}{Theorem}[section]
\newtheorem{lemma}[theorem]{Lemma}
\newtheorem{corollary}[theorem]{Corollary}
\newtheorem{proposition}[theorem]{Proposition}

\theoremstyle{definition}
\newtheorem{definition}[theorem]{Definition}
\newtheorem{remark}[theorem]{Remark}
\newtheorem{example}[theorem]{Example}

\newcommand{\al}{\alpha}
\newcommand{\bb}{\mathbb}

\newcommand{\bor}{\partial}

\newcommand{\Cl}[1]{\overline{\{#1\}}}
\newcommand{\comment}[1]{}

\newcommand{\dblecurly}[1]{\{\!\{ #1 \}\!\}}

\newcommand{\into}{\hookrightarrow}
\newcommand{\isoto}{\stackrel{\simeq}{\to}}

\newcommand{\mult}[1]{#1^{\!\times}}

\newcommand{\op}{\operatorname}
\newcommand{\pid}[1]{\langle #1 \rangle}
\newcommand{\res}{\overline}
\newcommand{\roi}{\mathcal{O}}
\newcommand{\ROI}{O}
\newcommand{\sub}[1]{{\mbox{\scriptsize #1}}}

\newcommand{\To}{\longrightarrow}
\newcommand{\ul}[1]{\underline{#1}}
\newcommand{\w}{\omega}

\renewcommand{\cal}{\mathcal}
\renewcommand{\hat}{\widehat}
\renewcommand{\frak}{\mathfrak}
\newcommand{\indlim}{\varinjlim}
\renewcommand{\tilde}{\widetilde}

\renewcommand{\ker}{\operatorname{Ker}}
\renewcommand{\projlim}{\varprojlim}

\DeclareMathOperator{\Aut}{Aut}
\DeclareMathOperator{\Char}{char}
\DeclareMathOperator{\codim}{codim}

\DeclareMathOperator{\Frac}{Frac}
\DeclareMathOperator{\Gal}{Gal}

\DeclareMathOperator*{\rprod}{\prod\nolimits^{\prime}\hspace{-1mm}}

\DeclareMathOperator{\Spec}{Spec}
\DeclareMathOperator{\supp}{supp}

\DeclareMathOperator{\height}{ht}
\DeclareMathOperator{\cdvdim}{cdvdim}

\newcommand{\rec}{\Upsilon}
\newcommand{\dash}{\textendash}
\newcommand{\comp}{{\hat{\phantom{o}}}}

\usepackage{fancyhdr}

\usepackage{sectsty}
\sectionfont{\Large\sc\centering}
\chapterfont{\large\sc\centering}
\chaptertitlefont{\LARGE\centering}
\partfont{\centering}

\begin{document}
\title{\Huge An introduction to higher dimensional local fields and ad\`eles}
\author{\Large \sc Matthew Morrow}

\date{}

\maketitle

\begin{abstract}
These notes are an introduction to higher dimensional local fields and higher dimensional ad\`eles. As well as the foundational theory, we summarise the theory of topologies on higher dimensional local fields and higher dimensional local class field theory.
\end{abstract}

\tableofcontents

\section*{Introduction}
The theory of local fields, or more generally complete discrete valuation fields, is a widely used tool in algebraic and arithmetic geometry. In particular, such fields are at the heart of the local-to-global principle, the idea that one can study a family of local problems and then deduce global information, typically using the ring of ad\`eles.

In the 1970s, A.~Parshin \cite{Parshin1975, Parshin1976, Parshin1978, Parshin1983, Parshin1984} generalised local fields by introducing higher dimensional local fields and establishing, in finite characteristic, their local class field theory using Milnor $K$-theory. At the same time he generalised ad\`eles to algebraic surfaces by constructing two-dimensional ad\`ele groups, namely restricted products of two-dimensional local fields and their rings of integers; using his two-dimensional ad\`eles, he studied Serre duality, intersection theory, and the class field theory of algebraic surfaces. Independently, K.~Kato \cite{Kato1979, Kato1980, Kato1982} developed class field theory for higher dimensional local fields of mixed characteristic. The class field theory was subsequently extended to arithmetic surfaces by K.~Kato and S.~Saito \cite{KatoSaito1983, KatoSaito1986}, who later further extended the theory to arbitrary dimensional arithmetic varieties. Meanwhile, ad\`eles in arbitrary dimensions were introduced by A.~Beilinson \cite{Beilinson1980}.

As well as class field theory, higher dimensional fields have found applications in the development of explicit approaches to Grothendieck duality and trace maps: the basic case of a curve may be found in \cite[III.7.14]{Hartshorne1977}, while the higher dimensional theory \cite{Hubl1996, Yekutieli1992, Lomadze1981, Osipov2000, Parshin2000, Morrow2009, Morrow2011} has still not reached its final form. We mention also the representation theory of algebraic groups over two-dimensional local fields \cite{GaitsgoryKazhdan2004, Lee2004}, Ind-Pro approaches to harmonic analysis on such fields \cite{Kapranov2001a, Parshin2008, Parshin2010}, and a theory of integration and zeta integrals on such fields \cite{Fesenko2003, Fesenko2008a, Morrow2008, Morrow2008a}.

The aim of these notes is to explain the fundamentals of higher dimensional local fields and higher dimensional ad\`eles, and then to summarise various additional aspects. We hope that this will provide the reader with a sufficient grounding in the subject to be able to navigate and understand the existing literature with relative ease. Another introduction to the local theory is \cite{Zhukov2000a}

Section \ref{section_cdvfields} is a terse review of complete discrete valuation fields; we expect the reader to be reasonably comfortable with this material, and many readers may wish to skip the section. The only result which the reader should particularly note is theorem \ref{theorem_uniqueness_of_valn}, from which it follows that a field can be a complete discrete valuation field in at most one way.

Higher dimensional local fields appear in section \ref{section_hdlfs}, where we start by introducing, purely for clarity of exposition, the idea of the `complete discrete valuation dimension' of a field; this is a natural notion which arises as soon as one acknowledges that the residue field of a complete discrete valuation field may again be a complete discrete valuation field. We prove the classification theorem, namely the higher dimensional analogue of the fact that a local field is either a field of Laurent series over $\bb F_q$ or a finite extension of $\bb Q_p$. The section finishes by introducing sequences of local parameters and defining morphisms between higher dimensional fields.

The next three sections of these notes are largely, though not completely, independent, and readers may choose the subjects of most interest to them.

Section \ref{section_integers} describes the family of higher rank rings of integers with which a higher dimensional field is equipped, generalising the usual ring of integers of a local field.

Section \ref{section_topology} is a summary, without proofs, of a body of work due originally to Parshin [{\em op.~cit.}] and Fesenko \cite{Fesenko_thesis, Fesenko1989} concerning the use of topologies in the theory of higher dimensional local fields. The key idea is that the usual discrete valuation topology on a higher dimensional local field is far from satisfactory and must be replaced. We explain how to do this and state the main properties of the resulting topology, before summarising several additional applications (e.g.~topological $K$-groups) and other approaches.

Section \ref{section_CFT} is a summary of the higher dimensional local class field theory originally developed by Parshin and Kato, as we discussed above. After defining and summarising the necessary properties of Milnor $K$-theory, we carefully state the main results of higher dimensional local class field theory and then sketch four existing approaches, due to Parshin, Kato, Fesenko, and Y.~Koya and M.~Spiess.

The remaining sections are of a global or semi-global nature and do not rely on sections \ref{section_integers} -- \ref{section_CFT}. Firstly, sections \ref{section_regular} and \ref{section_singular} explain, with complete proofs, how to associate a family of higher dimensional fields to a scheme: the input is a complete flag $\xi$ of irreducible closed subschemes (or, equivalently, a chain of prime ideals in a local ring), and the output is a higher dimensional field $F_\xi$ (or, more typically, a finite product of such fields). This is a reflection of the adelic philosophy that the local data in higher dimensions are not points, but flags of irreducible closed subschemes. Section \ref{section_regular} treats the case in which the subschemes occurring in the flag are suitably regular; this vastly simplifies the technicalities, while preserving the main concepts. The general case is treated in section \ref{section_singular}, where the statements are clearly separated from the technical proofs; these proofs may certainly be ignored by a reader encountering this material for the first time.

Section \ref{section_adeles} is an introduction to the theory of higher dimensional ad\`eles, due to A.~Parshin in dimension two and A.~Beilinson in general. We begin by carefully formulating the recursive definition, which associates an ad\`ele group $\bb A(T,M)$ to any quasi-coherent sheaf $M$ and a collection $T$ of flags on the scheme $X$. Each such ad\`ele group can be interpreted as a restricted product $\prod_\xi M_\xi^\comp$ of local factors; in particular, taking $M=\roi_X$ and $T$ to be the set of all complete flags, one obtains $\bb A_X$, a restricted product of higher dimensional local fields which plays the role of the ring of ad\`eles for $X$. Section \ref{subsection_adeles_and_cohomology} makes precise this notion of `restricted product'. Then the main theorem of ad\`eles is stated without proof: they provide a cosimplicial flasque resolution of $M$ and therefore may be used to compute its cohomology. Finally, we construct all the ad\`ele groups in dimensions one and two, explicitly describing the restricted products.

Each section begins with its own more detailed introduction.

\subsection*{Acknowledgements}
These notes have existed in a wide variety of forms for several years, and I have received many useful comments from O.~Br\"aunling, A.~C\'amara, and I.~Fesenko. I am very grateful to them.

Further comments are very welcome!

\section{Complete discrete valuation fields}\label{section_cdvfields}
Here we review the basic theory of discrete valuation fields, especially those which are complete, summarising material such as valuations, prime elements, morphisms, and extensions. Familiarity with this material is expected, and the reader may prefer to jump ahead to the next section and return here if necessary; more comprehensive introductions may be found in \cite{Fesenko2002} or \cite{Serre1979}.

\begin{definition}
Let $F$ be a field. A {\em discrete valuation} on $F$ is a non-zero homomorphism $\nu:\mult{F}\to\bb{Z}$, with $\nu(0):=\infty$ where $\infty>n$ for all $n\in\bb{Z}$, and which satisfies $\nu(a+b)\ge\min\{\nu(a),\nu(b)\}$ for all $a,b\in F$.

Associated to the valuation there is the {\em ring of integers} $\roi_\nu=\{x\in F:\nu(x)\ge 0\}$, which is a discrete valuation ring with maximal ideal $\frak{p}_\nu=\{x\in F:\nu(x)>0\}$ and residue field $\res{F}_\nu=\roi/\frak{p}$. A {\em prime} or {\em uniformiser} $\pi_{\nu}$ for $\nu$ is a generator of the principal ideal $\frak{p}_\nu$. When the valuation $\nu$ is clear from the context, the $\nu$ subscripts will sometimes be omitted.

The {\em discrete valuation topology} on $F$ associated to $\nu$ is the topology on $F$ in which the basic open neighbourhoods of $a\in F$ are $a+\frak{p}_\nu^n$, for $n\ge 0$. We say that $F$ is {\em complete} under $\nu$ if and only if it is complete in this topology.
\end{definition}

$\nu(\mult{F})$ is a non-zero subgroup of $\bb{Z}$, hence equals $m\bb{Z}$ for some $m\in\bb{Z}$; we may scale by $m$ and so there is therefore no loss in generality in assuming $\nu$ surjective; i.e., $\nu(\mult{F})=\bb{Z}$. In this case a uniformiser $\pi$ is an element satisfying $\nu(\pi)=1$. From now on all our discrete valuations will be surjective.

The following result, the Approximation lemma, applies to fields equipped with multiple valuations, such as $\bb{Q}$ and is closely related to the Chinese remainder theorem:

\begin{proposition}
Let $F$ be a field and $\nu_1,\dots,\nu_n:\mult{F}\to\bb{Z}$ distinct discrete valuations. Given $a_1,\dots,a_n\in F$ and $c\in\bb{Z}$, there exists $a\in F$ such that $\nu_i(a-a_i)\ge c$ for $i=1,\dots,n$.
\end{proposition}
\begin{proof}
\cite[I.3.7]{Fesenko2002} or \cite[I.\S3]{Serre1979}.
\end{proof}

The second basic result on valued fields is Hensel's lemma:

\begin{proposition}
Let $F$ be a field complete with respect to a discrete valuation $\nu$; let $f(X)$ be a monic polynomial over $\roi_F$ and suppose that the reduction $\res{f}\in\res F_\nu[X]$ has a simple root $\xi$ in $\res F_\nu$. Then there is a unique $a\in\roi_F$ which satisfies $f(a)=0$ and $\res{a}=\xi$.
\end{proposition}
\begin{proof}
\cite[II.1.2]{Fesenko2002} or \cite[II.\S4]{Serre1979}.
\end{proof}

The previous two results have a corollary which will be of enormous importance throughout. Indeed, without the following result it would be almost impossible to define higher dimensional local fields.

\begin{theorem}\label{theorem_uniqueness_of_valn}
Let $F$ be a field complete with respect to a discrete valuation $\nu$. Then $\nu$ is the only (surjective) discrete valuation on $F$.
\end{theorem}
\begin{proof}
Suppose that $w$ is a surjective discrete valuation on $F$ which is not equal to $\nu$; let $\pi$ be a prime for $\w$. By the approximation theorem there exists $a\in F$ such that $w(a-\pi)>0$ and $\nu(a-1)>0$. Then $a\in\roi_{\nu}$ and the image of $a$ in $\res{F}_{\nu}$ is $1$.

Now let $m>1$ be any integer not divisible by $\mbox{char}\,\res{F}_{\nu}$. Applying Hensel's lemma to the polynomial $X^m-a\in\roi_{\nu}[X]$ obtains $b\in\roi_{\nu}$ such that $b^m=a$. But therefore $w(a)$ is divisible by $m$, whereas the inequality $w(a-\pi)>0$ implies $w(a)=1$. This contradiction completes the proof.
\end{proof}

\begin{corollary}\label{corollary_automatic_continuity}
Let $F,L$ be fields with discrete valuations $\nu_F,\nu_L$; assume $F$ is complete. If $\sigma:F\to L$ is any field isomorphism then $\nu_F=\nu_L\circ\sigma$ and $L$ is also complete.
\end{corollary}
\begin{proof}
It is enough to note that $\nu_L\circ\sigma$ is a discrete valuation on $F$, for then the previous result implies that $\nu_F=\nu_L\circ\sigma$, and then $F$ complete implies that $L$ is complete.
\end{proof}

\begin{definition}
Let $F$ be a field. Then $F$ is said to be a {\em complete discrete valuation field} if and only if there exists a discrete valuation on $F$ under which $F$ is complete. By theorem \ref{theorem_uniqueness_of_valn} there is exactly one surjective valuation on such a field; it will be denoted $\nu_F$. We will write $\roi_F,\frak{p}_F,\res{F}$ in place of $\roi_{\nu_F},\frak{p}_{\nu_F},\res{F}_{\nu_F}$.
\end{definition}

\begin{remark}\label{remark_uniqueness_of_valn}
It is important to note the phrasing of the previous definition; being a `complete discrete valuation field' is an intrinsic property of the field $F$ and {\em not dependent on any prior choice of valuation}.
\end{remark}

\begin{definition}
A {\em local field} is a complete discrete valuation field whose residue field is finite. It is not uncommon to be more flexible and insist only that the residue field be perfect, but we will not do this.
\end{definition}

\begin{example}
For completeness we include some examples, but hope that they are familiar to most readers:
\begin{enumerate}
\item $\bb Q_p$, for $p$ a prime number, is a local field.
\item Let $k$ be an arbitrary field. The field of {\em formal Laurent series} $k((t))$ consists of formal infinite series $\sum_i a_it^i$ where $a_i$ are elements of $k$ which vanish for $i$ sufficiently small. Addition and multiplication are defined in the usual way. This makes $k((t))$ into a complete discrete valuation field (even a local field if $k$ is finite); the discrete valuation is defined by \[\nu\left(\sum_ia_it^i\right)=\min\{i:a_i\neq0\}.\] The ring of integers is $k[[t]]$, the maximal ideal $tk[[t]]$, and the residue field is $k$.

Note that each expression $\sum_ia_it^i$ is a genuinely convergent series in the discrete valuation topology, because $\nu(a_it^i)\ge i\to\infty$ as $i\to\infty$.
\end{enumerate}
\end{example}

We finish this preliminary section by reviewing extensions of discrete valuation fields. The following result, describing the equivalent formulations of morphisms between valuation fields, is well-known but I could not find a good reference:

\begin{lemma}\label{lemma_extensions_of_dvfields}
Let $L,F$ be fields with discrete valuations $\nu_F,\nu_L$ and suppose that $i:F\to L$ is a field embedding. Then the following are equivalent:
\begin{enumerate}
\item $i^{-1}(\roi_L)=\roi_F$;
\item $i^{-1}(\frak{p}_L)=\frak{p}_F$;
\item there exists an integer $e\ge 1$ such that $\nu_L\circ i=e\nu_F$;
\item $i$ is continuous with respect to the discrete valuation topologies on $L$ and $F$;
\item $i$ is a homeomorphism onto its image.
\end{enumerate}
\end{lemma}
\begin{proof}
For simplicity we may identify $F$ as a field with its image $i(F)$ and assume $i$ is an inclusion. It is straightforward to check that (iii) implies each of (i)-(v). Let $\pi\in F$ be a prime for $\nu_F$.

Assume (i). Then for any $x\in F$ such that $\nu_F(x)<0$ we have $\nu_L(x)<0$; moreover, for $x\in\mult{F}$ with $\nu_F(x)>0$ we may replace $x$ by $x^{-1}$ to deduce $\nu_L(x)>0$. Hence $\nu_L(x)=0$ for any $u\in F$ with $\nu_F(u)=0$. Define $e=\nu_L(\pi)>0$. Now for any $x\in\mult{F}$ one has $\nu_F(x\pi^{-\nu(x)})=0$ and so \[\nu_L(x)=\nu_L(x\pi^{-\nu_F(x)}\pi^{\nu_F(x)})=e\nu_F(x),\] proving (iii).

Assume (ii). Then for any $x\in F$ with $\nu_F(x)\le 0$ we have $\nu_L(x)\le 0$; replacing $x$ by $x^{-1}$ we deduce $\nu_F(x)\ge 0$ implies $\nu_L(x)\ge0$. So if $x\in F$ satisfies $\nu_L(x)\ge0$ then $\nu_L(x\pi)>0$ and so $\nu_F(x\pi)>0$; therefore $\nu_F(x)\ge 0$, proving (i).

(v)$\implies$(iv) is apparent.

Assume (iv). The crucial observation is that if $x\in F$ then $x^n$ tends to $0$ in the valuation topology as $n\to\infty$ if and only if $\nu_F(x)>0$. The assumed continuity now implies that if $\nu_F(x)>0$ then $\nu_L(x)>0$; after replacing $x$ by $x^{-1}$, we deduce that $\nu_F(x)<0$ implies $\nu_L(x)<0$, i.e.~(i).

This completes the proof.
\end{proof}

When a field embedding between two discrete valuation fields satisfies the equivalent conditions of the previous lemma, we will say that it is a {\em morphism of discrete valuation fields}, or simply that it is {\em continuous}. Such a morphism induces an embedding of the residue fields \[\res i:\res F\to\res L.\]
Typically one identifies $F$ with its image in $L$ and speaks of {\em an extension of discrete valuation fields}. Recall that, in this situation, we have the \emph{inertia degree} \[f(\nu_L/\nu_F)=|\res{L}/\res{F}|\] and the {\em ramification degree} \[e(\nu_L/\nu_F)=|\nu_L(\mult{L}):\nu_L(\mult{F})|\] (that is, since $\nu_F,\nu_L$ are assumed to be surjective, $e(\nu_L/\nu_F)$ is the unique integer $e\ge1$ which satisfies $\nu_L|_F=e\nu_F$). When there is no ambiguity, in particular when $L/F$ is an extension of complete discrete valuation fields, we will write $e(L/F), f(L/F)$ in place of $e(\nu_L/\nu_F),f(\nu_L/\nu_F)$.

\begin{example}
Consider the field extension $\bb{Q}_p((t))/\,\bb{Q}_p$. Although each field is a complete discrete valuation field, this is {\em not} an extension of discrete valuation fields (for example, $\bb Q_p[[t]]\cap \bb Q_p=\bb Q_p\neq\bb Z_p$, violating condition (i) in the previous lemma).

Similarly, $\bb{F}_p((t_1))((t_2))/\,\bb{F}_p((t_1))$ is not an extension of discrete valuation fields, but $\bb{F}_p((t_1))((t_2))/\,\bb{F}_p((t_2))$ is.
\end{example}

The valuation on a complete discrete valuation field extends uniquely to any finite extension:

\begin{proposition}\label{proposition_extension_of_local_field}
Let $F$ be a complete discrete valuation field, and let $L$ be a finite field extension of $F$. Then $L$ is a complete discrete valuation field and $L/F$ is an extension of complete discrete valuation fields.
\end{proposition}
\begin{proof}
\cite[II.\S2]{Serre1979} or \cite[II.2.5]{Fesenko2002}.
\end{proof}

We finish with the well-known $n=ef$ equality:

\begin{proposition}\label{proposition_e=fn}
Let $L/F$ be an extension of complete discrete valuation fields. Then $L/F$ is a finite extension if and only if $\res L/\res F$ is a finite extension, in which case \[|L/F|=e(L/F)f(L/F).\]
\end{proposition}
\begin{proof}
\cite[II.\S2]{Serre1979} or \cite[II.2.4]{Fesenko2002}.
\end{proof}

\begin{corollary}
A finite field extension of a local field is a local field.
\end{corollary}
\begin{proof}
Combine the previous two propositions.
\end{proof}

\section{Higher dimensional local fields}\label{section_hdlfs}
In this section we introduce higher dimensional local fields. We begin by introducing the notion of the `complete discrete valuation dimension' (cdvdim) of a field; for example, a field has cdvdim $\ge 2$ if and only if it is a complete discrete valuation field whose residue field is also a complete discrete valuation field. We explain a process $K\leadsto K\{\{t\}\}$ which, like the Laurent series construction $K\leadsto K((t))$, can be used to produce such fields. Then we define higher dimensional local fields in definition \ref{definition_hlf}) and offer examples. The section finishes with three important additional topics: the Classification theorem for higher dimensional local fields, sequences of local parameters, and extensions.

The following definition does not appear anywhere in the literature, but offers a certain clarity of exposition:

\begin{definition}
Let $F$ be a field; the {\em complete discrete valuation dimension} of $F$, denoted $\cdvdim{F}$, is defined as follows. If $F$ is not a complete discrete valuation field then $\cdvdim F:=0$; if $F$ is a complete discrete valuation field then $\cdvdim{F}:=\cdvdim{\res{F}}+1$.
\end{definition}

In other words, supposing that $F^{(0)}:=F$ is a complete discrete valuation field, we consider the residue field $F^{(1)}:=\res{F}$; if this is not a complete discrete valuation field then we are done and we set $\cdvdim{F}=1$. But perhaps $\res{F}$ is also a complete discrete valuation field, in which case we consider {\em its} residue field $F^{(2)}$; we continue in this way until we reach a field which is not a complete discrete valuation field, and we define $\cdvdim{F}$ to be the number of residue fields we passed through. One typically represents this situation by a diagram \[\xymatrix{F=F^{(0)}\ar@{-}[d]\\F^{(1)}\ar@{-}[d]\\ \vdots\ar@{-}[d]\\F^{(n)},}\] where $F^{(i)}$ is a complete discrete valuation field with residue field $F^{(i+1)}$ for $i=0,\dots,n-1$ and $F^{(n)}$ is not a complete discrete valuation field.

If the sequence of residue fields never terminates then set $\cdvdim{F}=\infty$. Such fields do exist but they are unnatural and we will not encounter them (the reader may wish to try to construct one).

The residue field $\res F=F^{(1)}$ is often referred to as the {\em first residue field} of $F$ while $F^{(\op{cdvdim}F)}$ is called the {\em last residue field}. In general, $F^{(i)}$ is called the $i^\sub{th}$ residue field of $F$, for $1\le i\le\op{cdvdim}F$.

\begin{example}
A field $F$ has cdvdim $\ge 2$ if and only if it is a complete discrete valuation field and its residue field $\res F$ is {\em also} a complete discrete valuation field.
\end{example}

\begin{remark}
In case the reader has already forgotten remark \ref{remark_uniqueness_of_valn}, we take this opportunity to stress again that the property of a field being a complete discrete valuation field is uniquely determined by its algebraic structure, without specifying a priori any valuation. Therefore it makes sense to ask whether $F, \res{F}, F^{(2)}$, etc.~are complete discrete valuation fields, {\em without saying what the valuations are}. This is a common point of confusion when first encountering the theory of higher dimensional local fields.
\end{remark}

\begin{remark}
There is not a fixed convention whether one should label the successive residue fields of $F$ using upper or lower indices, nor whether the final residue field should be indexed by $0$ or $n$. We feel that our chosen convention is most flexible; in particular, it is compatible with the `upper numbering for codimension' of algebraic geometry.
\end{remark}

In the next two examples we describe two ways of constructing complete discrete valuation fields, both of which easily allow us to build fields of arbitrarily large complete discrete valuation dimension. The first is the familiar Laurent series construction:

\begin{example}\label{example_iterated_Laurent_series}
Let $k$ be an arbitrary field (perhaps already a complete discrete valuation field) and let \[F=k((t_1))\cdots((t_n))=\big(k((t_1))\cdots((t_{n-1}))\big)((t_n))\] be a field of iterated Laurent serious over it. Then $\cdvdim{F}=n+\cdvdim{k}$ and \[F^{(i)}=\begin{cases}k((t_1))\cdots((t_{n-i}))&0\le i\le n,\\ k^{(n-i)}&n\le i\le\cdvdim F.\end{cases}\]
\end{example}

The second example is extremely important from an arithmetic point of view and perhaps not familiar to the reader; we will spend some time describing the construction and resulting properties:

\begin{example}\label{example_double_curly_field}
Let $K$ be a complete discrete valuation field. Let $K\dblecurly{t}$ be the following collection of doubly infinite formal series:
\begin{align*}
	K\{\{t\}\}
	=\bigg\{\sum_{i=-\infty}^{\infty}a_it^i:\, &a_i\in K\mbox{ for all }i,\\
	&\inf_i\nu_K(a_i)>-\infty,\\
	&\mbox{and }a_i\to 0\mbox{ as } i\to -\infty\bigg\},\\
\end{align*}
where $\sum_ia_it^i=\sum_ib_it^i$ if and only if $a_i=b_i$ for all $i$. Define addition, multiplication, and a discrete valuation by
\begin{align*}
\sum_{i=-\infty}^{\infty}a_it^i + \sum_{j=-\infty}^{\infty}b_jt^j&=\sum_{i=-\infty}^{\infty}(a_i+b_i)t^i\\
\sum_{i=-\infty}^{\infty}a_it^i\cdot\sum_{j=-\infty}^{\infty}b_jt^j&=\sum_{i=-\infty}^{\infty}\left(\sum_{r=-\infty}^{\infty}a_rb_{i-r}\right)t^i\\
\nu\left(\sum_{i=-\infty}^{\infty}a_it^i\right)&=\inf_i\nu_K(a_i)
\end{align*}
Note that there is nothing formal about the sum over $r$ in the definition of multiplication; rather it is a convergent double series in the complete discrete valuation field $K$. It is left to the reader as an important exercise to verify that these operations are well-defined, make $K\{\{t\}\}$ into a field, and that $\nu$ is a discrete valuation under which $K\{\{t\}\}$ is complete. Note also that $K\{\{t\}\}/K$ is an (infinite) extension of complete discrete valuation fields of ramification degree $e(K\{\{t\}\}/K)=1$.

The ring of integers of $F=K\{\{t\}\}$ and its maximal ideal are given by
\begin{align*}
\roi_F&=\left\{\sum_ia_it^i:a_i\in\roi_K\mbox{ for all }i\mbox{ and }a_i\to0\mbox{ as }i\to-\infty\right\},\\
\frak{p}_F&=\left\{\sum_ia_it^i:a_i\in\frak{p}_K\mbox{ for all }i\mbox{ and }a_i\to0\mbox{ as }i\to-\infty\right\}.
\end{align*}
The surjective homomorphism \begin{align*}\roi_F&\To\res{K}((t))\\\sum_ia_it^i&\mapsto\sum_i\res{a}_it^i\end{align*} identifies the residue field of $F$ with $\res{K}((t))$ ($=$ Laurent series over the residue field of $K$), which is itself a complete discrete valuation field. Therefore we make take residue fields again: $F^{(2)}=\res{K}$:
\[\begin{array}{c}
K\{\{t\}\}\\ | \\ \res K((t)) \\ | \\ \res K
\end{array}\]

Continuing to take successive residue fields, we see that $\cdvdim F=1+\cdvdim K$ and that $F^{(i)}=K^{(i-1)}$ for $i=2,\dots,\cdvdim F$.
\end{example}

\begin{remark}
One may also describe the field $F$ from the previous example as being the field of fractions of the $\frak p_K$-adic completion of $\bb \roi_K((t))$.
\end{remark}

\begin{example}
If we apply the previous construction to the field $K=\bb Q_p$ we obtain $\bb Q_p\{\{t\}\}$, a complete discrete valuation field of characteristic $0$, in which $p$ is a uniformiser, and with residue field $\bb F_p((t))$. Therefore one may think of $\bb Q_p\{\{t\}\}$ as the `field of fractions of the Witt ring of $\bb F_p((t))$', except that the Witt ring construction does not actually work for imperfect fields.
\end{example}

\begin{example}
Just as we iterated Laurent series in example \ref{example_iterated_Laurent_series}, we may iterate the construction of $K\{\{t\}\}$. Given a complete discrete valuation field $K$, consider \[F=K\{\{t_1\}\}\cdots\{\{t_n\}\}=\big(K\{\{t_1\}\}\cdots\{\{t_{n-1}\}\}\big)\{\{t_n\}\}.\] We have a tower of infinite extensions of complete discrete valuation fields \[K\le K\{\{t_1\}\}\le\cdots\le K\{\{t_1\}\}\cdots\{\{t_n\}\},\] with respective residue fields \[\res K\le \res K((t_1))\le\cdots\le\res K((t_1))\cdots ((t_n)).\] In particular, $\cdvdim F=n+\cdvdim K$, and its first $n+1$ residue fields are as follows:
\[\begin{array}{l}
F=K\{\{t_1\}\}\cdots\{\{t_n\}\} \\ | \\ \res F=\res K((t_1))\cdots ((t_n)) \\ | \\F^{(2)}=\res K((t_1))\cdots ((t_{n-1}))\\ | \\\vdots\\ | \\ F^{(n+1)}=\res K
\end{array}\]
\end{example}

The construction of the field $K\{\{t\}\}$ is most interesting when $K$ has mixed characteristic, for the following reason:

\begin{lemma}
Let $K=k((u))$ be the field of Laurent series over a field $k$. Then there is an isomorphism of fields \[K\{\{t\}\}\cong k((t))((u)),\quad \sum_i\left(\sum_j a_{i,j}u^j\right)t^i\mapsto \sum_j\left(\sum_ia_{i,j}t^i\right)u^i\] ($a_{i,j}\in k$). 
\end{lemma}
\begin{proof}
We leave this as an exercise to the reader to test his/her own understanding of the definition of $K\{\{t\}\}$.
\end{proof}

\begin{corollary}
Let $k$ be a field and let $1\le r\le n$ be integers. Then there is an isomorphism of fields \[k((t_1))\cdots((t_r))\{\{t_{r+1}\}\}\cdots\{\{t_n\}\}\cong k((t_{r+1}))\cdots((t_n))((t_1))\cdots((t_r)).\]
\end{corollary}

The study of complete discrete valuation fields is essentially the study of fields of cdvdim $=1$, because one does not take into account the possibility that the residue field has the additional structure of being a complete discrete valuation field itself. Acknowledging this extra structure leads to the theory of higher dimensional fields. Among complete discrete valuation fields an exalted position is occupied by local fields; these are the complete discrete valuation fields whose residue field is finite. The higher dimensional analogue is our main object of study:

\begin{definition}[\cite{Kato1980, Parshin1976, Parshin1984}]\label{definition_hlf}
A field $F$ is said to be an {\em $n$-dimensional local field} for some $n\ge 0$ if and only if $\cdvdim{F}=n$ and the final (i.e., the $n^\sub{th}$) residue field of $F$ is finite. When the exact dimension is not specified one simply speaks of a {\em higher dimensional local field}.
\end{definition}

\begin{remark}
Unfortunately, the phrase `$n$-dimensional local field over $k$' is sometimes used in the literature to mean a field $F$ of $\cdvdim\ge n$ whose $n^\sub{th}$ residue field is $k$. We do not like this notation because a `$2$-dimensional local field over $\bb Q_p$' is then a $3$-dimensional local field. More generally, `higher dimensional field' or `higher local field' are often used for any field of $\cdvdim\ge 1$.

In an attempt to avoid such confusing notation, we will try to be rigorous in self-imposing the following guide: `higher dimensional field' informally means any field, probably of cdvdim $\ge 1$, while `higher dimensional {\em local} field' strictly means that the final residue field is finite.
\end{remark}

\begin{example}
A zero-dimensional local field is a finite field, a one-dimensional local field is a usual local field (e.g., $\bb{Q}_p$), and a two-dimensional local field is a complete discrete valuation field whose residue field is a usual local field (e.g., $\bb F_q((t_1))((t_2))$, $\bb{Q}_p((t))$, and $\bb{Q}_p\dblecurly{t}$). Note that if $F$ is a $n$-dimensional local field then, for $i=0,\dots,n$, the field $F^{(i)}$ is an $n-i$ -dimensional local field.

Here is a table indicating some two-dimensional local fields and their respective residue fields:
\[\begin{array}{l|cccc}
F & & \bb F_q((t_1))((t_2)) & \bb Q_p((t)) & \bb Q_p\{\{t\}\}\\
|&& |&|&|\\
F^{(1)}=\res F & & \bb F_q((t_1)) & \bb Q_p & \bb F_p((t))\\
|&&|&|&|\\
F^{(2)} & & \bb F_q & \bb F_p & \bb F_p
\end{array}\]
\end{example}

So far we have only seen higher dimensional local fields which have been constructed in a rather straightforward way from lower dimensional fields; the following provides many more interesting ones:

\begin{proposition}\label{proposition_finite_extension_of_hlf}
Let $F$ be an $n$-dimensional local field, and let $L$ be a finite field extension of $F$. Then $L$ is an $n$-dimensional local field.
\end{proposition}
\begin{proof}
This follows by induction on $n$ using proposition \ref{proposition_extension_of_local_field}.
\end{proof}

\begin{example}\label{example_non_standard_field}
Let $F=\bb Q_p\{\{t\}\}$ and put $f(X)=X^p-pT^{-1}\in\roi_F[X]$; let $\al$ be a root of $f(X)$ and put $L=F(\al)$. Then $L$ is a two-dimensional local field. Moreover, since $f(X)$ is an Eisenstein polynomial, $L/F$ is totally ramified (as an extension of complete discrete valuation fields) of degree $p$. By more carefully analysing the arguments which appear in the Classification theorem below, one can show that $L$ does not have the form $K\{\{u\}\}$ for any finite extension $K$ of $\bb Q_p$.
\end{example}

\begin{example}\label{example_ferocious}
Let $F=\bb F_p((t_1))((t_2))$ and put $g(X)=X^p-X-t_1t_2^{-p}$; let $\beta$ be a root of $g(X)$ and put $M=F(\beta)$. Then $M$ is a two-dimensional local field. Moreover, since $g(X)$ is an Artin-Schreier polynomial, $M/F$ is a cyclic extension of degree $p$. Put $\gamma=t_2\beta$ and notice that $\gamma^p-t_2^{p-1}\gamma-t_1=0$; so $\gamma\in\roi_F$, and $\res\gamma^p=t_1$ in $\res M$. Therefore $\res M=\bb F_p((t_1^{1/p}))$, and so the $n=ef$ equality implies that $e(M/F)=1$.

However, notice that $M/F$ is {\em not} unramified in the usual sense since the residue field extension $\res M/\res F$ is purely inseparable. This phenomenon, which does not appear in the classical theory of local fields, is known as {\em ferocious ramification} \cite{Zhukov2003}, and has been the major source of trouble in developing ramification theory for higher dimensional local fields, or more generally for complete discrete valuation fields with imperfect residue field. Various approaches and contributions to the theory, in chronological order, are due to K.~Kato \cite{Kato1987, Kato1989, Kato1994}, O.~Hyodo \cite{Hyodo1987}, I.~Fesenko \cite{Fesenko1995}, I.~Zhukov \cite{Zhukov2000, Zhukov2003}, A.~Abbes and T.~Saito \cite{Abbes2002} \cite{Abbes2003}, and J.~Borger \cite{Borger2004, Borger2004a}.

 \end{example}

\subsection{The Classification theorem for higher dimensional local fields}
It is well-known that a (one-dimensional) local field is either a field of Laurent series over a finite field, or is a finite extension of $\bb Q_p$ for some prime number $p$. An entirely analogous result holds in higher dimensions, originally due to Parshin \cite{Parshin1978}. First notice that if $F$ is an $n$-dimensional local field, then either it has finite characteristic or at some point in its chain of residue fields $F^{(0)},F^{(1)},\dots, F^{(n)}$ the characteristic switches from $0$ to finite.

\begin{theorem}[Classification theorem]\label{theorem_classification}
Let $F$ be an $n$-dimensional local field.
\begin{enumerate}
\item If $\op{char}F\neq0$ then \[F\cong  F^{(n)}((t_1))\cdots((t_n))\] ($F^{(n)}$ is a finite field).
\item If $\op{char}F^{(n-1)}=0$ then \[F\cong F^{(n-1)}((t_1))\cdots((t_{n-1}))\] ($F^{(n-1)}$ is a (one-dimensional) local field of characteristic $0$).
\item In the remaining case, let $2\le r\le n$ be the unique integer such that $F^{(n-r)}$ has mixed characteristic (i.e., $\op{char}F^{(n-r)}=0\neq\op{char}F^{(n+1-r)}$). Then $F$ is isomorphic to a finite extension of \[\bb Q_q\{\{t_1\}\}\cdots\{\{t_{r-1}\}\}((t_{r+1}))\cdots((t_n)),\] where $\bb Q_q$ is the unramified extension of $\bb Q_p$ with residue field $F^{(n)}$.
\end{enumerate}
\end{theorem}
\begin{proof}
(i) and (ii) follow easily from the well-known fact that any equal-characteristic complete discrete valuation field $L$ is isomorphic to $\res L((t))$ (e.g., \cite[II.\S4]{Serre1979} or \cite[II.5]{Fesenko2002}). For (iii), put $K=F^{(r)}$, which is a mixed characteristic $r$-dimensional local field; the aforementioned fact implies that $F\cong K((t_{r+1}))\cdots((t_n))$, so it remains to show that $K$ is a finite extension of $\bb Q_q\{\{t_1\}\}\cdots\{\{t_{r-1}\}\}$.

Well, we may apply case (i) to $\res K$ to deduce that $\res K\cong F^{(n)}((t_1))\cdots((t_{r-1}))$. Thus $K':=\bb Q_q\{\{t_1\}\}\cdots\{\{t_{r-1}\}\}$ is a complete discrete valuation field, in which $p$ is prime, with the same residue field as the mixed characteristic complete discrete valuation field $K$; a standard structure theorem for mixed characteristic complete discrete valuation field implies that there is therefore a field embedding $K'\into K$ making $K$ into a finite extension of $K$ (e.g., \cite[II.5.6]{Fesenko2002}). This completes the proof.
\end{proof}

\begin{remark}
An $n$-dimensional local field isomorphic to one of the form \[K\{\{t_1\}\}\cdots\{\{t_{r-1}\}\}((t_{r+1}))\cdots((t_n)),\] where $K$ is a (one-dimensional) mixed characteristic local field, is called {\em standard}.

If $F$ is as in part (iii) of the Classification theorem, then a theorem of I.~Zhukov \cite{Zhukov1995} states that $F$ is not only a finite extension of a standard field, but moreover that it has a finite extension which is standard. Being able to sandwich $F$ in this way can reduce calculations to standard fields.
\end{remark}

\subsection{Sequences of local parameters}
Next we define what is meant by the sequence of local parameters in a higher dimensional field; recall that a prime, or uniformiser, of a complete discrete valuation field $F$ is an element $\pi$ satisfying $\nu(\pi)=1$, or equivalently a generator of the maximal ideal of $\roi_F$.

\begin{definition}
Let $F$ be a field of cdvdim $\ge n$. A {\em sequence of ($n$) local parameters} $t_1,\dots,t_n\in F$ is a sequence of elements with the following (recursively defined) properties:
\begin{enumerate}
\item $t_n$ is a uniformiser of $F$;
\item $t_1,\dots,t_{n-1}\in\roi_F$, and their reductions $\res t_1,\dots,\res t_{n-1}$ form a sequence of local parameters for the field $\res F$ of cdvdim $\ge n-1.$
\end{enumerate}
Informally, pick any uniformiser of $F^{(n-1)}$ and lift it $F$ to get $t_1$; then pick a uniformiser of $F^{(n-2)}$ and lift it to $F$ to get $t_2$; etc.
\end{definition}

\begin{example}
Here is a table of higher dimensional fields and some corresponding sequences of local parameters; the reader should pay particular attention to the order of $p$ and $t$ for the fields $\bb Q_p((t))$ and $\bb Q_p\{\{t\}\}$:
\[\begin{array}{c|c|l}
\mbox{field } F & \op{cdvdim}F & \mbox{a sequence of local parameters}\\\hline
\bb Q_p & 1 & p\\
\bb Q_p((t)) & 2 & p,t\\
\bb Q_p\{\{t\}\} & 2 & t,p\\
\bb F_q((t_1))\cdots((t_n)) & n & t_1,\dots,t_n\\
\bb Q_p\{\{t_1\}\}\cdots\{\{t_{r-1}\}\}((t_{r+1}))\cdots((t_n)) & n & t_1,\dots, t_{r-1},p,t_{r+1},\dots,t_n\\
L\mbox{ (example \ref{example_non_standard_field})} & 2 & t,\al\\
M\mbox{ (example \ref{example_ferocious})} & 2 &\gamma,t_2
\end{array}\]
\end{example}

\begin{remark}
To dispel any feeling that sequences of local parameters must `look nice', as in the previous table, we remark that if $F$ is a two-dimensional local field with sequence of local parameters $t_1,t_2$, then another sequences of local parameters is $t_1+a,t_2+b$ for any $a\in\frak p_F$, $b\in\frak p_F^2$. Indeed, $t_2+b$ is still a uniformiser of $F$, and the reduction of $t_1$ to $\res F$ is not changed by adding $a$.

For example, \[p+10t+27t^5,\,t+p^{-1}t^2\] is a perfectly good sequence of local parameters for $\bb Q_p((t))$, though not one we would typically choose to work with.
\end{remark}

Sequences of local parameters will appear again in section \ref{subsection_rois_and_parameters} where they will be used to decompose the multiplicative group of a higher dimensional field.

\subsection{Extensions of higher dimensional fields}
We finish this section by discussing extensions of higher dimensional fields; the results are straightforward inductions on the dimension of the fields. Again, the following definition does not exist in the literature but it is convenient:

\begin{definition}
Let $i:F\to L$ be a field embedding between fields of cdvdim $\ge n$. We recursively define what it means for $i$ to be {\em a morphism of fields of cdvdim $\ge n$}, or more simply {\em $n$-continuous} as follows: If $n=0$ then $i$ is always $0$-continuous; if $n\ge 1$ then $i$ is $n$-continuous if and only if it is a morphism of discrete valuation fields and $\res i:\res F\to\res L$ is $n-1$\,-continuous.
\end{definition}

\begin{example}
Suppose that $i:F\to L$ is an embedding between fields of cdvdim $\ge2$. Then $i$ is automatically $0$-continuous. It is $1$-continuous if and only if it is a morphism of complete discrete valuation fields, which is equivalent to saying that it is continuous with respect to the discrete valuation topologies on $F$ and $L$. It is $2$-continuous if and only if it is continuous and the induced embedding $\res i:\res F\to\res L$ is also continuous (with respect to the discrete valuation topologies on $\res F$ and $\res L$).
\end{example}

If $F\to L$ is an $n$-continuous embedding between fields of cdvdim $\ge n$, then we may pass to successive residues fields to obtain field extensions $L^{(1)}/F^{(1)},\dots,L^{(n)}/F^{(n)}$, which we could represent with a diagram like the following:
\[\begin{array}{ccc}
F=F^{(0)} &\le& L=L^{(0)}\\
|&&|\\
F^{(1)}&\le&L^{(1)}\\
|&&|\\
\vdots&&\vdots\\
|&&|\\
F^{(n)}&\le& L^{(n)}
\end{array}\]

\begin{example}
The natural inclusion $\bb Q_p\{\{t_2\}\}\to\bb Q_p\{\{t_1\}\}\{\{t_2\}\}$ is $2$-continuous, giving rise to residue field extensions as follows:
\[\begin{array}{ccc}
\bb Q_p\{\{t_2\}\} & \le & \bb Q_p\{\{t_1\}\}\{\{t_2\}\} \\
|&&|\\
\bb F_p((t_2)) & \le & \bb F_p((t_1))((t_2))\\
|&&|\\
\bb F_p & \le & \bb F_p((t_1))
\end{array}\]
\end{example}

We have already pointed out in proposition \ref{proposition_finite_extension_of_hlf} than a finite field extension $L$ of an $n$-dimensional local field is again an $n$-dimensional local field; the proof shows moreover that $F\to L$ is $n$-continuous. We now return to the two earlier examples of finite extensions:

\begin{example}
We consider the finite extension $L/F$ of two-dimensional local fields from example \ref{example_non_standard_field}; the resulting residue field extensions are
\[\begin{array}{ccc}
F & \le & L \\
|&&|\\
\bb F_p((t)) & \le & \bb F_p((t))\\
|&&|\\
\bb F_p & \le & \bb F_p
\end{array}\]
\end{example}

\begin{example}
We consider the finite extension $M/F$ of two-dimensional local fields from example \ref{example_ferocious}; the resulting residue field extensions are
\[\begin{array}{ccc}
F & \le & M \\
|&&|\\
\bb F_p((t_1)) & \le & \bb F_p((t_1^{1/p}))\\
|&&|\\
\bb F_p & \le & \bb F_p
\end{array}\]
\end{example}

We finish with the generalised $n=ef$ equality:

\begin{proposition}
Let $F\to L$ be an $n$-continuous embedding between fields of cdvdim $\ge n$. Then $L/F$ is a finite extension if and only if $L^{(n)}/F^{(n)}$ is a finite extension, in which case all the residue field extensions $L^{(i)}/F^{(i)}$ are finite ($i=1,\dots,n$) and \[|L/F|=e(L/F)e(L^{(1)}/F^{(1)})\cdots e(L^{(n-1)}/F^{(n-1)})|L^{(n)}/F^{(n)}|.\]
\end{proposition}
\begin{proof}
A straightforward induction using the dimension $1$ result, namely proposition \ref{proposition_e=fn}.
\end{proof}

We will return to morphisms of higher dimensional fields in theorem \ref{theorem_n_continuity_of_functor}, where we show that the localisation-completion process of constructing higher dimensional fields is functorial.

\section{Higher rank rings of integers}\label{section_integers}
A complete discrete valuation field $F$ contains its subring of integers $\roi_F$. Analogously, higher dimensional fields come equipped with a chain of rings of integers \[F\supset\roi_F=\ROI_F^{(1)}\supset\ROI_F^{(2)}\supset\cdots\supset\ROI_F^{(n)},\] arising by lifting the rings of integers from each successive residue field. Each of these higher rank rings of integers is a Henselian valuation ring, and we will describe their prime ideals. We use them to decompose $\mult F$, refining the well-known decomposition $\mult F\cong\mult{\roi_F}\times\bb Z$, and then describe their behaviour under field extensions.

These higher rank rings of integers do not play an important role in the geometric applications of higher dimensional fields, though their unit groups are fundamental in higher dimensional local class field theory.

\begin{remark}[Valuation rings]\label{remark_valn_rings}
Before defining the higher rank rings of integers we remind the reader of various facts concerning valuation rings; a good reference is \cite[Chap.~4]{Matsumura1989}. A {\em valuation ring} $\ROI$ is an integral domain such that for all non-zero $x$ in $\Frac{\ROI}$, either $x$ or $x^{-1}$ belongs to $\ROI$. Any valuation ring is a local ring in which the ideals are totally ordered by inclusion. If $\roi\supseteq\ROI$ is an extension of valuation rings with the same field of fractions, $\frak{m}$ is the maximal ideal of $\ROI$, and $\frak{p}$ the maximal ideal of $\roi$, then $\frak{m}\supseteq\frak{p}$ and $\roi=\ROI_{\frak{p}}$ (this will be particularly important to us).

If $\ROI$ is a valuation ring with field of fractions $F$, then $\Gamma:=\mult{F}/\mult{\ROI}$ becomes a totally ordered abelian group under the ordering \[a\mult{\ROI}\ge b\mult{\ROI}\Longleftrightarrow ab^{-1}\in\ROI,\] and the natural map \[\nu:\mult{F}\to\Gamma\] is a valuation on $F$. Moreover, every valuation on $F$ arises from a valuation subring in this way (when one speaks of a {\em valuation subring} of a field, one always means one whose field of fractions is all of $F$).

$\ROI$ said to be {\em Henselian} if and only if it satisfies Hensel's lemma with respect to its maximal ideal $\frak m$: If $f(X)$ is a monic polynomial over $\ROI$ whose reduction $f\op{mod}\frak{m}$ has a simple root $\xi$ in $\ROI/\frak m$, then there is a unique $a\in\ROI$ which satisfies $f(a)=0$ and $a\op{mod}\frak{m}=\xi$.

Valuation rings are at the heart of the Zariski school's approach to algebraic geometry, in contrast with the approach via Noetherian local rings. The two rarely coincide, for if $\ROI$ is a valuation ring then the following are equivalent:
\begin{enumerate}
\item $\ROI$ is Noetherian;
\item $\ROI$ is a principal ideal domain;
\item $\ROI$ is a field or a discrete valuation ring;
\item $\ROI$ has at most one non-zero prime ideal.
\end{enumerate}
\end{remark}

\begin{remark}[Rank of a valuation ring]\label{remark_rank_of_valn_ring}
We must also mention the notion of the {\em rank} of a valuation ring $\ROI$. Let $\Gamma=\mult{F}/\mult{\ROI}$ be the totally ordered abelian group associated to $\ROI$, as in the previous remark. Then the rank of $\ROI$ is defined to be the order-rank of $\Gamma$, a notion which may be unfamiliar:

Let $\Gamma$ be any totally ordered abelian group. A subgroup $H$ of $\Gamma$ is called {\em isolated} if and only if \[\gamma\in H\Longrightarrow [-\gamma,\gamma]\subseteq H,\] where $[-\gamma,\gamma]$ is the interval of elements in $\Gamma$ between $-\gamma$ and $\gamma$; the {\em order-rank} of $\Gamma$ is then defined to be the number of non-zero isolated subgroups it has. For example, if $\Gamma$ is non-zero and embeds (as an ordered group) into $\bb{R}$ then it has order-rank $1$.

If $\Gamma,\Delta$ are totally ordered abelian groups, then $\Gamma\times\Delta$ becomes a totally ordered abelian group under the (left) lexicographic ordering: \[(\gamma,\delta)>(\gamma',\delta')\Longleftrightarrow \gamma>\gamma'\mbox{, or }\gamma=\gamma'\mbox{ and }\delta>\delta'.\] (Of course one can swap the importance of the roles of $\Gamma$ and $\Delta$ to obtain the right lexicographic ordering, and the reader should be careful which convention is being used when consulting the literature.) Iterating the process yields the  lexicographic ordering on any finite product of totally ordered abelian groups.

The reader should convince him/herself that \[\bb{Z}^n=\underbrace{\bb{Z}\times\dots\times\bb{Z}}_\sub{$n$ times},\] with the lexicographic ordering, has order-rank $n$.
\end{remark}

\begin{definition}
Let $F$ be a field of $\cdvdim\ge n$. Then the \emph{rank $n$ ring of integers} $\ROI_F^{(n)}$ is defined inductively as follows: if $n=0$ then $\ROI_F^{(0)}:=F$, and if $n>0$ then \[\ROI_F^{(n)}:=\{x\in\roi_F:\res{x}\in\ROI_{\res{F}}^{(n-1)}\},\] where $\ROI_{\res{F}}^{(n-1)}$ is the rank $n-1$ ring of integers of $\res{F}$, a field of $\cdvdim\ge n-1$. An elementary induction shows that $\ROI_F^{(n)}$ is a valuation subring of $F$, and we let $\frak p_F^{(n)}$ denote its maximal ideal.

Note that we have just defined higher rank rings of integers $\ROI_F^{(i)}$, for $i=0,\dots,\cdvdim F$, and from the very definition it is clear that they form a descending chain of valuation subrings of $F$: \[F\supset\roi_F=\ROI_F^{(1)}\supset\ROI_F^{(2)}\supset\dots\supset\ROI_F^{(\cdvdim F)}.\]

When we are interested in a particular higher rank ring of integers, e.g.~the $n^\sub{th}$, then we will use the notation $\ROI_F$. The notation $\roi_F$ is strictly reserved for the usual discrete valuation subring of $F$.
\end{definition}

\begin{example}
Suppose $F$ is a field of $\cdvdim\ge 2$. Then
\begin{align*}
\ROI_F^{(0)}&=F\\
\ROI_F^{(1)}&=\roi_F\\
\ROI_F^{(2)}&=\{x\in\roi_F:\res{x}\in\roi_{\res{F}}\},
\end{align*}
where $\roi_F$ (resp. $\roi_{\res{F}}$) is the usual ring of integers of the complete discrete valuation field $F$ (resp. $\res{F}$). Usually one write $\ROI_F=\ROI_F^{(2)}$. Here is a table of our main examples of two-dimensional local fields together with their rings of integers:
\[\begin{array}{c|c|c}
F & \roi_F=\ROI_F^{(1)} & \ROI_F=\ROI_F^{(2)}\\\hline
\bb F_q((t_1))((t_2)) & \bb F_q((t_1))[[t_2]] & \bb F_q[[t_1]]+t_2\bb F_q((t_1))[[t_2]]\\
\bb{Q}_p((t)) & \bb{Q}_p[[t]] & \bb{Z}_p+t\bb{Q}_p[[t]]\\
\bb Q_p\{\{t\}\} & \{\sum_ia_it^i:a_i\in\bb Z_p\,\forall i\}& \{\sum_ia_it^i:a_i\in\bb Z_p\,\forall i\ge 0\mbox{ and }a_i\in p\bb Z_p\,\forall i<0\}
\end{array}\]
\end{example}

We start by describing the algebraic structure of the higher rank rings of integers:

\begin{proposition}\label{proposition_main_properties_of_rois}
Let $F$ be a field of $\cdvdim\ge n$ and rank $n$ ring of integers $\ROI_F=\ROI_F^{(n)}$. Then
\begin{enumerate}
\item $\ROI_F$ is a Henselian valuation ring with field of fractions $F$ and residue field $F^{(n)}$.
\item As defined above, let $\frak p_F^{(i)}$ be the maximal ideal of $\ROI_F^{(i)}$, for $i=0,\dots,n$. These form a strictly descending chain of prime ideals of $\ROI_F$, \[\ROI_F\supset \frak{p}_F^{(n)}\supset\dots\supset\frak{p}_F^{(1)}=\frak{p}_F\supset\frak p_F^{(0)}=0,\] and $\ROI_F$ has no other prime ideals.
\item The localisation of $\ROI_F$ away from $\frak p_F^{(i)}$ is $\ROI_F^{(i)}$, for $i=0,\dots,n$.
\item The quotient of $\ROI_F$ by $\frak p_F^i$ is $\ROI_{F^{(i)}}^{(n-i)}$.
\end{enumerate}
\end{proposition}
\begin{proof}
Since $\ROI_F\subseteq\ROI_F^{(i)}$, for $i=0,\dots,n$, remark \ref{remark_valn_rings} implies that $\frak p_F^{(i)}\subset\ROI_F$ and that $(\ROI_F)_{\frak p_F^{(i)}}=\ROI_F^{(i)}$, proving (iii). More generally, the same result implies that $\frak p_F^{(i)}\subset\frak p_F^{(j)}$ whenever $i\le j$. In particular, $\frak p_F\subset\ROI_F$ and $(\ROI_F)_{\frak p_F}=\roi_F$; if $\ROI_F$ were to contain a non-zero prime ideal $\frak q$ strictly smaller that $\frak p_F$, then $(\ROI_F)_{\frak q}$ would be a valuation ring strictly between $\roi_F$ and $F$, which is impossible. Therefore $\frak p_F$ is the minimal non-zero prime ideal of $\ROI_F$.

From the original definition of $\ROI_F$, one has $\ROI_F/\frak p_F=\ROI_{\res F}$, the rank $n-1$ ring of integers of $\res F$; parts (ii) and (iv) now follow by easy inductions, as does the fact that the residue field of $\ROI_F$ is $F^{(n)}$.

To see that $\ROI_F$ is Henselian, take a polynomial $f(X)$ with coefficients in $\ROI_F$ and a simple root in $F^{(n)}$; the root may be succesively lifted up to $\ROI_F$ using completeness of each intermediate residue field $F^{(i)}$, $i=n-1,\dots,0$.
\end{proof}

\begin{example}
Suppose $F$ is a field of cdvdim $\ge 2$. Then the chain of prime ideals of $\ROI_F=\ROI_F^{(2)}$ is \[\ROI_F\supset\frak p_F^{(2)}\supset\frak p_F\supset 0,\] where \[\frak p_F^{(2)}=\{x\in\roi_F:\res x\in\frak p_{\res F}\}.\] Here is a table of our main examples of two-dimensional local fields together with the chain of prime ideals:
\[\begin{array}{c|c|c}
F & \frak p_F^{(2)} & \frak p_F=\frak p_F^{(1)}\\\hline
\bb F_q((t_1))((t_2)) &  t_1\bb F_q[[t_1]]+t_2\bb F_q((t_1))[[t_2]] & t_2\bb F_q((t_1))[[t_2]]\\
\bb{Q}_p((t)) & p\bb Z_p+t\bb{Q}_p[[t]] & t\bb{Q}_p[[t]]\\
\bb Q_p\{\{t\}\} & \{\sum_ia_it^i:a_i\in\bb Z_p\,\forall i>0\mbox{ and }a_i\in p\bb Z_p\,\forall i\le0\}& \{\sum_ia_it^i:a_i\in\bb Z_p\,\forall i\}
\end{array}\]
\end{example}

Before continuing any further we verify that the nomenclature `rank $n$ ring of integers' is not misleading:

\begin{proposition}\label{proposition_rank_is_rank}
Let $F$ be a field of $\cdvdim\ge n$ and rank $n$ ring of integers $\ROI_F=\ROI_F^{(n)}$. Then the quotient group $\mult{F}/\mult{\ROI}_F$ is order isomorphic to $\bb{Z}^n$ equipped with the lexicographic ordering, and so $\ROI_F$ has rank $n$ by remark \ref{remark_rank_of_valn_ring}.
\end{proposition}
\begin{proof}
For $n=0$ this is trivial so suppose $n>0$. Choosing a uniformiser of $F$ induces an isomorphism $\mult{F}\cong\bb{Z}\times\mult{\roi}_F$, which descends to an isomorphism \[\mult F/\mult \ROI_F\cong \bb Z\times(\mult\roi_F/\mult\ROI_F).\tag{\dag}\] Moreover, if $\mult\roi_F/\mult\ROI_F$ is totally ordered as a subgroup of $\mult F/\mult\ROI_F$ and the right hand side of (\dag) is given the left lexicographical ordering, then (\dag) is an isomorphism of totally ordered groups, since $\frak p_F\subseteq\ROI_F$.

Moreover, the reduction map $\mult{\roi}_F\to\mult{\res{F}}$ induces an isomorphism $\mult\roi_F/\mult\ROI_F\isoto \mult{\res{F}}/\mult{\ROI}_{\res{F}}$, where $\ROI_{\res{F}}$ is the rank $n-1$ ring of integers of $\res{F}$. The inductive hypothesis completes the proof.
\end{proof}

\subsection{Relation to local parameters and decomposition of $\mult F$}\label{subsection_rois_and_parameters}
If $F$ is a complete discrete valuation field, and $t\in F$ is a uniformiser, then we have \[F=\roi_F[t^{-1}]\mbox{   and   }\mult F\cong\mult{\roi_F}\times t^{\bb Z},\] where $t^{\bb Z}$ indicates the infinite cyclic subgroup of $\mult F$ generated by $t$. Here we will describe the higher dimensional generalisations of these results, using sequences of local parameters and the higher rank rings of integers.

\begin{proposition}
Let $F$ be a field of cdvdim $\ge n$, and let $t_1,\dots,t_n$ be a sequence of $n$ local parameters. Then, for any $i=1,\dots,n$, \[\ROI_F^{(n-1-i)}[t_i^{-1}]=\ROI_F^{(n-i)}\mbox{   and   }\mult{(\ROI_F^{(n-i)})}\cong \mult{(\ROI_F^{(n-1-i)})}\times t_i^{\bb Z}.\]
\end{proposition}
\begin{proof}
If $i=n$ then this is the one-dimensional result recalled immediately above, so we may assume $i<n$. Then both $\ROI_F^{(n-i)}$ and $\ROI_F^{(n-1-i)}$ have minimal non-zero prime ideal $\frak p_F$ and it is enough to prove the result after reducing mod $\frak p_F$; by a straightforward induction this completes the proof of the first result, and the second result follows using the same argument.
\end{proof}

\begin{example}
Let $F=\bb Q_p((t))$, which has local parameters $p,t$ and rank $2$ ring of integers $\ROI_F^{(2)}=\bb Z_p+t\bb Q_p[[t]]$. Clearly \[\ROI_F^{(2)}[p^{-1}]=\bb Q_p[[t]]=\roi_F,\] exactly as the proposition states (taking $n=2$ and $i=1$).
\end{example}

The following corollary should be compared with proposition \ref{proposition_rank_is_rank}, where we saw more abstractly that $\mult F/\mult{\ROI_F}\cong\bb Z^n$:

\begin{corollary}\label{corollary_decomposition_of_units}
Let $F$ be a field of cdvdim $\ge n$, and let $t_1,\dots,t_n$ be a sequence of $n$ local parameters; let $\ROI_F=\ROI_F^{(n)}$ denote the rank $n$ ring of integers. Then \[\ROI_F[t_1^{-1},\dots,t_n^{-1}]=F\mbox{   and   }\mult F\cong\mult{\ROI_F}\times t_1^{\bb Z}\times\cdots\times t_n^{\bb Z}.\]
\end{corollary}
\begin{proof}
Immediate from the previous proposition.
\end{proof}

\begin{remark}[Higher groups of principal units] 
 If $F$ is an $n$-dimensional local field then $\ROI_F$ is Henselian with a finite residue field $\bb F_q$, and so the usual Teichm\"uller lift $[\cdot]:\mult{\bb F_q}\to\mult{\ROI_F}$ exists. It follows that $\mult{\ROI}_F\cong V\times\mult{\bb F_q}$, where \[V=1+\frak p_F^{(n)}=\ker(\mult{\ROI_F}\to\mult{\bb F_q}).\] The group $V$ is called the group of {\em higher principal units} of $F$, and the previous corollary provides a decomposition \[\mult F\cong V\times\mult{\bb F_q}\times t_1^{\bb Z}\times\cdots\times t_n^{\bb Z}.\] $V$ plays a major role in explicit higher dimensional local class field theory; see the references in remarks \cite{remark_Parshin} and \cite{remark_Fesenko}.
\end{remark}

\subsection{Rings of integers under extensions}
We finish this section by offering two resulting concerning the higher rank rings of integers when a change of field is involved. The first is a characterisation of $n$-continuous morphisms, in a similar way as lemma \ref{lemma_extensions_of_dvfields} did in the case $n=1$:

\begin{proposition}
Let $i:F\to L$ be a ring homomorphism between fields of cdvdim $\ge n$; then the following are equivalent:
\begin{enumerate}
\item $i$ is $n$-continuous;
\item $i^{-1}(\ROI_L^{(r)})=\ROI_F^{(r)}$ for $r=0,\dots,n$;
\item $i^{-1}(\frak p_L^{(r)})=\frak p_F^{(r)}$ for $r=0,\dots,n$.
\end{enumerate}
\end{proposition}
\begin{proof}
As with most proofs in this section, this is a simple induction on $n$: $r=0$ is equivalent to $i$ being an embedding, $r=1$ is equivalent to $i$ being a morphism of discrete valuation fields (by lemma \ref{lemma_extensions_of_dvfields}), and assuming the result for $r=0,1$, we claim that the following are then equivalent, from which (i)$\Leftrightarrow$(ii) follows:
\begin{enumerate}
\item[(a)] $\res i:\res F\to\res L$ is $n-1$-continuous.
\item[(b)] $i^{-1}(\ROI_{\res L}^{(r)})=\ROI_{\res F}^{(r)}$ for $r=1,\dots,n-1$.
\item[(c)] $i^{-1}(\ROI_L^{(r)})=\ROI_F^{(r)}$ for $r=2,\dots,n$.
\end{enumerate}
The equivalence of (a) and (b) is the inductive hypothesis; the equivalence of (b) and (c) follows from the facts that $i^{-1}(\frak p_L)=\frak p_F$ (by lemma \ref{lemma_extensions_of_dvfields}) and that $\ROI_F^{(r)}/\frak p_F\cong\ROI_{\res F}^{(r-1)}$ (and similarly for $L$).

The equivalence (i)$\Leftrightarrow$(iii) is very similar, using the identification $\frak p_F^{(r)}/\frak p_F\cong \frak p_{\res F}^{(r-1)}$ for $r=1,\dots,n$.
\end{proof}

Our second result is the higher dimensional generalisation of the well-known result \cite[II.\S2]{Serre1979} that if $L/F$ is a finite extension of complete discrete valuation fields, then $\roi_L$ is the integral closure of $\roi_F$:

\begin{proposition}
Let $L/F$ be a finite extension of fields of cdvdim $\ge n$. Then $\ROI_L^{(n)}$ coincides with the integral closure of $\ROI_F^{(n)}$ inside $L$.
\end{proposition}
\begin{proof}
Standard theory of valuation rings tells us that the valuation subrings of $L$ which dominate $\ROI_F^{(n)}$ are given by localising the integral closure of $\ROI_F^{(n)}$ at any of its maximal ideals (see e.g.~\cite[Exer.~12.2]{Matsumura1989}). But it is also well known that $\ROI_F^{(n)}$ being Henselian impies that $L$ has only one valuation subring dominating it; this completes the proof.
\end{proof}

\comment{By induction on $n$. If $n=0$ then the claim is that $L/F$ is an algebraic extension, which of course is true. Now suppose $n>0$ and take $x\in\ROI_L^{(n)}$; since $\res{x}\in\ROI_{\res{L}}^{(n-1)}$, the inductive hypothesis implies that there is a polynomial $g\in\ROI_{\res{L}}^{(n-1)}[X]$ such that $g(\res{x})=0$. Lifting $g$ obtains a polynomial $f\in\ROI_L^{(n)}[X]$ such that $f(x)\in\frak{p}_L$. Let $\pi\in F$ be a uniformiser for $F$; then $\pi^{-1}f(x)$ belongs to $\roi_L$, which is well-known to be the integral closure of $\roi_F$ inside $L$.

Therefore there is a polynomial $h\in\roi_F[X]$ which satisfies $h(\pi^{-1}f(x))=0$.

Let $f\in F[X]$ be the minimal polynomial of $x$ over $F$. Since $\roi_L$ is the integral closure of $\roi_F$ inside $L$ and $x\in\roi_L$, we conclude that $f\in\roi_F[X]$.}

\section{Topologies on higher dimensional local fields}\label{section_topology}
The usual discrete valuation topology on a higher dimensional local field is unsatisfactory in many ways and is too fine for most applications. For example, in the two-dimensional local field $F=\bb F_p((t_1))((t_2))$, a typical formal series $\sum_{i=0}^{\infty}a_it_1^i$ in the first local parameter is not a genuinely convergent sum in the topology (since $\bb F_p((t_1))$, as a subspace of $\bb{F}_p((t_1))((t_2))$, carries the discrete topology). Moreover, the reduction map $\roi_F\to\res F$ is not a topological quotient map if both $\roi_F$ and $\res F$ are given their discrete valuation topologies. These problems led to a body of work defining topologies on higher dimensional local fields $F$ which take into account the fact that the residue fields $F^{(i)}$ themselves have a topology. The aim of this section is to summarise aspects of the theory, due originally to A.~Parshin \cite{Parshin1984} and I.~Fesenko \cite{Fesenko_thesis, Fesenko1989}, accurately but without proofs; many more details can be found in \cite{Madunts1995}, while \cite{Zhukov2000a} also summarises the constructions. 

More precisely, we will define the so-called `higher topology' on any higher dimensional local field $F$ for which $\res F$ has finite characteristic (see remark \ref{remark_residue_char_restriction} for the reason behind this mild restriction), and in section \ref{subsection_main_properties} we will list its main properties. The remainder of this section then discusses applications of and other other approaches to the topology.

We will be interested in topologies on fields $K$ satisfying at least the following two conditions:
\begin{enumerate}
\item[(T1)] The topology is Hausdorff.
\item[(T2)] Addition $K\times K\stackrel{+}{\to} K$, and multiplication $K\stackrel{\times\al}{\to} K$ by any fixed element $\al\in K$, are all continuous maps (with respect to the product topology on $K\times K$).
\end{enumerate}

\begin{remark}\label{remark_vector_space_topology}
The conditions imply that $K$ is a Hausdorff topological group under addition. Moreover, if product space $K^n$ is given the product topology, then any linear map $K^n\to K^m$ is continuous; in particular, the product topology on any finite dimensional vector space $V$ resulting from a choice of a basis does not depend on the basis. We call this topology on $V$ the {\em natural vector space topology}.
\end{remark}

\subsection{Equal characteristic case}
Let $K$ be a field equipped with a topology satisfying conditions (T1) and (T2). Let $F=K((t))$ be the field of formal Laurent series over $K$; using the topology on $K$, we define a topology on $F$, following Parshin \cite{Parshin1984}. The basic open neighbourhoods of $0$ are those of the form \[\sum_i U_it^i:=\left\{\sum_ia_it^i\in F:\,a_i\in U_i\mbox{ for all }i\right\},\] where $(U_i)_{i=-\infty}^{\infty}$ are open neighbourhoods of $0$ in $K$ with the property that $U_i=K$ for $i\gg0$. The basic open neighbourhoods of any other point $f\in F$ are $f+U$ where $U$ is a basic open neighbourhood of $0$.

\begin{lemma}
This specification of `open neighbourhoods' defines a topology on $F$ satisfying conditions (T1) and (T2).
\end{lemma}
\begin{proof}
The proof is a good exercise; or see one of \cite{Madunts1995, Parshin1984}.
\end{proof}

\begin{remark}
If the topology on $K$ is the discrete topology (which certainly satisfies (T1) and (T2)), then the topology just described on $K((t))$ will be the usual discrete valuation topology.
\end{remark}

Using the lemma, the construction may be iterated to define a topology on $K((t_1))\cdots ((t_n))$ satisfying (T1) and (T2). The following is perhaps the central result regarding topologies on higher dimensional fields of finite characteristic:

\begin{theorem}[Parshin \cite{Parshin1984}]
Let $F$ be a field of cdvdim $n$; suppose $F$ has finite characteristic and that the final residue field $F^{(n)}$ is perfect. Choose any isomorphism $F\cong F^{(n)}((t_1))\cdots((t_n))$, give $F^{(n)}$ the discrete topology, and iterate the process just described to topologise $F$.

Then the resulting topology on $F$ does not depend on the chosen isomorphism.
\end{theorem}

\begin{definition}
Let $F$ be a higher dimensional local field of finite characteristic. Its {\em higher topology} is the topology defined in the previous theorem.
\end{definition}

\begin{remark}\label{remark_residue_char_restriction}
Unfortunately, the theorem drastically fails in characteristic $0$. If $F$ is a two-dimensional local field with residue field $\bb Q_p$ (which has its usual topology), so that it is non-canonically isomorphic to $\bb Q_p((t))$, then different isomorphisms $F\cong \bb Q_p((t))$ will yield different topologies on $F$. This was noticed by A.~Yekutieli \cite[2.1.22]{Yekutieli1992}, and we will now explain his proof. We start by identifying $F$ with $\bb Q_p((t))$ in a fixed way. Let $\{u_\al\}$ be a transcendence basis of $\bb Q_p$ over $\bb Q$, and fix arbitrary elements $v_\al\in\bb Q_p[[t]]$, at least once of which is non-zero. The field embedding \[i:\bb Q(\{u_a\})\to F,\quad u_\al\mapsto u_\al+v_\al t\] partially defines a coefficient field, and it extends to a full coefficient field $i:\bb Q_p\to F$ by repeated application of Hensel's lemma since $\bb Q_p/\bb Q(\{u_\al\})$ is a separable algebraic extension. This new coefficient field for $F$ gives a new isomorphism $F\cong \bb Q_p((t))$, finally resulting in a field automorphism $\sigma\in\Aut(\bb Q_p((t)))$ with the property that $\sigma(u_\al)=u_\al+v_\al t$ for all $\al$. If the topology on $F$ did not depend on the identification of $F$ with $\bb Q_p((t))$, then $\sigma$ would necessarily be continuous, but it is not. For example, $V:=\{\sum_ia_it^i\in\bb Q_p((t)):a_1\neq 0\}$ is open in $\bb Q_p((t))$, so if $\sigma$ were continuous then $\Omega:=\bb Q_p\cap\sigma^{-1}(V)$ would be open in $\bb Q_p$; but $\Omega$ is disjoint from $\bb Q$ and yet contains at least one of the $u_\al$.

Notice that this proof would fail if we replaced $\bb Q_p$ by a characteristic $p$ local field: the use of Hensel's lemma to extend the coefficient field was essential.

Therefore we will often restrict to higher dimensional local fields with finite residue characteristic when discussing topologies. This is not a serious drawback of the theory: if $F$ is a higher dimensional local field for which $\op{char}\res F=0$, then the arithmetic of $F$ reduces very easily to that of $\res F$, which has lower dimension and so we may proceed inductively.
\end{remark}

\subsection{Mixed characteristic case}
Let $K$ be a complete discrete valuation field equipped with a topology satisfying conditions (T1) and (T2). Let $F=K\{\{t\}\}$; using the topology on $K$, we define a topology on $F$, using a construction first given by Fesenko \cite{Fesenko_thesis}. The basic open neighbourhoods of $0$ are those of the form \[\sum_i U_it^i:=\left\{\sum_ia_it^i\in F:\,a_i\in U_i\mbox{ for all }i\right\},\] where $(U_i)_{i=-\infty}^{\infty}$ are open neighbourhoods of $0$ in $K$ with the following properties:
\begin{enumerate}
\item For any integer $r$, we have $\frak p_K^r\subseteq U_i$ for $i\gg 0$.
\item $\bigcap_{i\in\bb Z}U_i$ contains a power of $\frak p_K$.
\end{enumerate}
The basic open neighbourhoods of any other point $f\in F$ are $f+U$ where $U$ is a basic open neighbourhood of $0$.

\begin{lemma}
This specification of `open neighbourhoods' defines a topology on $F$ satisfying conditions (T1) and (T2).
\end{lemma}
\begin{proof}
Another good exercise for the reader; or see \cite{Madunts1995}.
\end{proof}

Using the lemma, the construction may be iterated to define a topology on $K\{\{t_1\}\}\cdots\{\{t_n\}\}$, and we have the following analogue of the earlier equal-characteristic theorem:

\begin{theorem}[Fesenko?]
Let $F$ be an $n$-dimensional local field of mixed characteristic (i.e., $\op{char}F=0\neq\op{char}\res F$); according to the Classification theorem \ref{theorem_classification}(iii), $F$ is isomorphic to a finite field extension of $F':=\bb Q_q\{\{t_1\}\}\cdots\{\{t_{n-1}\}\}$. Starting with the usual topology on $\bb Q_q$, iterate the process just described to topologise $F'$, and then give $F$ the natural vector space topology over $F'$.

Then the resulting topology on $F$ does not depend on any of the choices.
\end{theorem}

In fact, the proposition is proved by giving an alternative, direct construction of the topology on $F$ without appealing to the Classification theorem; this alternative construction uses a variety of subtle lifting maps $H:\res F\to\roi_F$ (set theoretic sections of the reduction map with certain properties). The interested reader can consult the standard reference \cite{Madunts1995} for complete details, or \cite{Zhukov2000a} for a summary of the approach.

\begin{definition}
Let $F$ be a higher dimensional local field of mixed characteristic. Its {\em higher topology} is the topology defined in the previous proposition.
\end{definition}

\subsection{Main properties}\label{subsection_main_properties}
In the previous two sections we defined the higher topology on any higher dimensional local field $F$ for which $\res F$ has finite characteristic. Now we will state the main properties of the topology; in the following theorem, higher dimensional local fields are always equipped with their higher topology:

\begin{theorem}\label{theorem_main_properties_of_topology}
Let $F$ be a higher dimensional local field for which $\res F$ has finite characteristic.
\begin{enumerate}
\item The higher topology on $F$ satisfies conditions (T1) and (T2). Moreover, it is complete in the sense that every Cauchy sequence converges.
\item If $L$ is a finite extension of $F$, then the higher topology on $L$ is the same as its natural vector space topology as a vector space over $F$.
\item The reduction map $\roi_F\to \res F$ is continuous and open (where $\roi_F$ is given the subspace topology from $F$, and $\res F$ is given its higher topology).
\item The higher topology on $F$ is linear, in the sense that any open neighbourhood of $0$ contains an open subgroup.
\item Let $\cal R\subseteq F$ be a set of representatives for the final residue field $F^{(n)}$, and let $t_1,\dots,t_n$ be a sequence of local parameters; then each element of $F$ can be written uniquely as an absolutely convergent series \[\sum_{(i_1,\dots,i_n)\in\Omega}\theta_{\ul i}\,t_1^{i_1}\cdots t_n^{i_n},\] where $\Omega\subseteq\bb Z^n$ is an admissible set and $\theta_{\ul i}\in\cal R$. (See below for further explanation.)
\end{enumerate}
\end{theorem}
\begin{proof}
Proofs in the equal characteristic case may be found in \cite{Parshin1984}; for the general case see \cite{Madunts1995}.
\end{proof}

The final part of the theorem requires explanation. We first recall the classical situation for a complete discrete valuation field $F$. Let $\cal R\subseteq \roi_F$ be a set of representatives for $\res F$; i.e., the reduction map carries $\cal R$ bijectively to $\res F$, and we always assume $0\in\cal R$ (e.g., if $F$ is a local field we could take $\cal R$ to be the Teichm\"uller representatives, together with $0$). Let $t\in F$ be a uniformiser. Then each element of $F$ may be uniquely written as a series \[\sum_i\theta_it^i,\tag{\dag}\] where $\theta_i\in\cal R$ and where $\theta_i=0$ for $i\ll 0$ (e.g., \cite[II.\S4]{Serre1979} or \cite[I.5.2]{Fesenko2002}). Moreover, all these series are absolutely convergent in the discrete valuation topology on $F$.

Now suppose further that $\op{cdvdim}F\ge 2$, so that $\res F$ is also a complete discrete field; then, interpreting $\theta_i$ as an element of $\res F$, we may apply the same result and expand $\theta_i$ in terms of representatives $\cal S$ and a uniformiser $\pi$ for $\res F$. This formally converts (\dag) into a double series \[\sum_i\sum_j\theta_{j,i}\pi^jt^i.\] The coefficients of this formal expansion satisfy the following two properties: firstly, for $i\ll0$, one has $\theta_{j,i}=0$ for all $j$; secondly, for any given $i$, one has $\theta_{j,i}=0$ for $j\ll0$. However, as remarked at the start of this section, such a double series is merely formal, and is {\em not} convergent in the discrete valuation topology on $F$. This is the problem fixed by the higher topology.

Now we precisely define the terms in part (v) of the theorem. The {\em set of representatives} $\cal R$ for $F^{(n)}$ is a (necessarily finite) subset $\cal R$ of the rank $n$ ring of integers $\ROI^{(n)}_F$ containing $0$ such that the reduction map $\ROI^{(n)}_F\to F^{(n)}$ (see proposition \ref{proposition_main_properties_of_rois}(i)) induces a bijection between $\cal R$ and $F^{(n)}$. An {\em absolutely convergent} series, indexed over a countable set, is a series which converges to the same result regardless of reordering; in fact, part (iv) of the theorem implies that all convergent series are automatically absolutely convergent. Finally, an {\em admissible} subset $\Omega\subseteq\bb Z^n$ is one satisfying the following condition:
\begin{quote}
For any $1\le r\le n$ and any fixed $i_{r+1},\dots,i_n\in\bb Z$, there exists $I\in\bb Z$ such that $\Omega$ contains no tuple of integers of the form $(i_1,\dots,i_r,i_{r+1}\dots,i_n)$ for which $i_r<I$.
\end{quote}

\begin{example}
Let $F=\bb F_p((t_1))((t_2))$. A set of representatives for $F^{(2)}=\bb F_p$ is $\cal R=\bb F_p\subseteq F$. Any element $f\in F$ can be formally expanded as \[f=\sum_i\sum_j\theta_{j,i}t_1^jt_2^i=\sum_{(j,i)\in\Omega}\theta_{j,i}t_1^jt_2^i,\] where $\Omega\subseteq\bb Z$ satisfies the following two conditions: firstly, for $i\ll 0$, the set $\Omega$ contains no pair of the form $(j,i)$; secondly, for any fixed $i\in\bb Z$, the set $\Omega$ does not contain $(j,i)$ for $j\ll 0$. These are precisely the conditions for $\Omega$ to be admissible.

Part (v) of the theorem states that this formal expansion of $f$ is absolutely convergent in the higher topology on $F$.
\end{example}

\subsection{Additional topics concerning topology}
In this section we summarise various additional topics concerning topologies on higher dimensional local fields.

\subsubsection{The topology on the multiplicative group}
If $R$ is a (commutative, unital) ring equipped with a topology then, even if addition and multiplication are continuous, inversion $\mult R\to\mult R$ on the group of units need not be continuous. The `correct' topology on $\mult R$ which makes inversion continuous is the weak topology induced by the map \[\mult R\to R\times R,\quad u\mapsto (u,u^{-1}).\]

\begin{example}
Let $\bb{Q}$ be the field of rational numbers and $\bb{A}$ its ring of ad\`eles with the usual topology as a restricted product. Let $a_n$ be the element of $\bb{A}$ defined by \[a_n=(1,1,\dots,1,p_n,1,\dots)\] where the only non-unital term is the $n^\sub{th}$ prime $p_n$, which belongs to the $\bb{Q}_{p_n}$-coordinate. Then $a_n\in\mult{\bb A}$ for all $n\ge 1$ and $a_n\to 1$ as $n\to\infty$; but $a_n^{-1}\not\to 1$. This shows that inversion is not continuous on the ideles $\mult{\bb A}$ if we topologise them as a subspace of $\bb A$.

The `correct' topology on the ideles $\mult{\bb A}$ is the one defined above, which yields the same topology if we were to follow the more familiar method of topologising them as the restricted product $\rprod_p\mult{\bb Q_p}$.
\end{example}

So, if $F$ is an $n$-dimensional local field equipped with its higher topology, then one topologises $\mult F$ using this `correct' topology. However, in the literature one will only find a different approach: according to corollary \ref{corollary_decomposition_of_units}, a choice of a sequence of local parameters induces an isomorphism $\mult F\cong\mult{\ROI_F}\times\bb Z^n$, where $\ROI_F$ is the rank $n$ ring of integers, and so we may give $\mult F$ the product topology after giving $\bb Z$ the discrete topology and topologising $\mult{\ROI_F}$ as a subspace of $F$. One can show that these two approaches result in the same topology on $\mult F$ (at least up to sequential saturation; see below).

Having correctly topologised $\mult F$, there is the following mulitplicative analogue of theorem \ref{theorem_main_properties_of_topology}(v), which generalises the multiplicative expansions for complete discrete valuation fields \cite[I.5.10]{Fesenko2002}:

\begin{proposition}\label{proposition_multiplicative_expansion}
Let $F$ be an $n$-dimensional local field for which $\res F$ has finite characteristic; let $\cal R\subseteq F$ be a set of representatives for the final residue field $F^{(n)}$, and let $t_1,\dots,t_n$ be a sequence of local parameters. Then each element of $\mult F$ can be written uniquely as an absolutely convergent product \[t_1^{r_1}\cdots t_n^{r_n}\theta\prod_{\ul i\in\Omega}(1+\theta_{\ul i}t_1^{i_1}\cdots t_n^{i_n}),\] where $\Omega\subseteq\bb Z^n$ is an admissible set, $r_1,\dots,r_n\in\bb Z$, $\theta\in\cal R\setminus\{0\}$, and $\theta_{\ul i}\in\cal R$.
\end{proposition}
\begin{proof}
\cite{Madunts1995}.
\end{proof}

\subsubsection{Topological $K$-groups}\label{subsection_top_Milnor_K_groups}
Section \ref{subsection_Milnor_K_theory} is a prerequisite for this material.

Given a field $K$, and a topology on $\mult K$, it is natural to study continuous Steinberg maps $(\mult K)^m\to A$ into Hausdorff (typically even discrete) topological abelian groups. To do this, we give the Milnor $K$-group $K_m(K)$ the strongest topology for which the symbol map \[\mult K\times\cdots\times\mult K\To K_m(K),\quad (x_1,\dots,x_m)\mapsto\{x_1,\dots,x_m\}\] is continuous and for which $K_m(K)$ is a topological group. The {\em topological Milnor $K$-group} $K_m^\sub{top}(K)$ is then defined to be the maximal Hausdorff quotient of $K_m(K)$. In other words, $K_m^\sub{top}(K)=K_m(K)/\Lambda_m(K)$, where $\Lambda_m(K)$ is the subgroup of $K_m(K)$ obtained by intersecting all neighbourhoods of $0\in K_m(K)$. This clearly satisfies the desired universal property: a continuous Steinberg map $(\mult K)^m\to A$ into a Hausdorff topological abelian group induces a unique continuous homomorphism $K_m^\sub{top}(K)\to A$.

For $m=2$, these topological $K$-groups were introduced by Milnor \cite[Appendix]{Milnor1970}, who studied them for $\bb R$, $\bb C$, and for local fields. J.~Graham \cite{Graham1973} later considered Laurent series fields.

\begin{example}
Let $K$ be a local field, equipped with its usual topology. Example \ref{example_K_groups_of_local_field} shows that there is an exact sequence \[0\to\mbox{a uniquely-divisible group}\to K_2(K)\to\mu\to 0.\] One can show \cite[Appendix]{Milnor1970} that this uniquely-divisible group is precisely $\Lambda_2(K)$, and so $K_2^\sub{top}(K)\cong\mu$.
\end{example}

\begin{example}
Example \ref{example_K_groups_of_reals} shows that there is a surjective homomorphism $\op{sign}:K_m(\bb R)\to\{\pm 1\}$, whose kernel is divisible. One can show \cite[Appendix]{Milnor1970} that this kernel is precisely $\Lambda_m(\bb R)$, having given $\bb R$ its usual topology, and so $K_m^\sub{top}(\bb R)\cong\{\pm 1\}$.
\end{example}

The topological Milnor $K$-groups of higher dimensional local fields $F$ have turned out to be tremendously important in explicit higher dimensional local class field theory. Elements of $K_m^\sub{top}(F)$ admit expansions in an analogous fashion to theorem \ref{theorem_main_properties_of_topology}(v) and proposition \ref{proposition_multiplicative_expansion}, and the continuous parings discussed in remarks \ref{remark_Parshin} and \ref{remark_Fesenko} can be analysed using these expansions. See \cite{Fesenko2001} and \cite{Zhukov1997} for many more details.

We finish with the higher dimensional analogue of the previous two examples, stating that $\Lambda_m(F)$ is the subgroup of divisible elements of $K_m(F)$; therefore $K_m^\sub{top}(F)$, as a group without any topology, can be defined purely algebraically:

\begin{theorem}[\cite{Fesenko2001}]
Let $F$ be a higher dimensional local field for which $\res F$ has finite characteristic; give $\mult F$ the topology defined above. Then, for any $m\ge 1$, \[\Lambda_m(F)=\bigcap_{\ell\ge 0}\ell K_m(F),\] the subgroup of divisible elements of $K_m(F)$.
\end{theorem}

\subsubsection{Sequential saturation}
The higher topology on a higher dimensional local field exhibits some surprising properties, the worst of which is likely the following: if $F$ is an $n$-dimensional local field, with $n\ge 2$, then multiplication \[F\times F\stackrel{\times}{\To}F\] is not continuous; in fact, if $U,V$ are arbitrarily small open neighbourhoods of $0$, then $UV=F$. Fortunately, it is sequentially continuous in the sense that if $x_n\to x$ and $y_n\to y$ then $x_ny_n\to xy$. These results were already noticed by Parshin \cite{Parshin1984}.

Since most topological calculations in higher dimensional local fields are very `sequential' in nature (e.g.,~the series expansions of theorem \ref{theorem_main_properties_of_topology}(v) and proposition \ref{proposition_multiplicative_expansion}), the most efficient way to fix this problem with multiplication is to replace the topologies on $F$, $F\times F$, $\mult F$, $K_m(F)$, etc., by their so-called sequential saturation: multiplication then becomes genuinely continuous. The reader will quickly encounter such sequential saturations when consulting the literature on higher topologies to which we have referred, so here we briefly summarise the main ideas:

Let $X$ be a topological space. A subset $A$ of $X$ is said to be {\em sequentially open} if and only if whenever $x_n\to x$ is a convergent sequence in $X$ such that $x\in A$, then $x_n\in A$ for $n\gg 0$. Open subspaces are sequentially open, but the converse is not always true: the space is said to be {\em sequentially-saturated} if and only if every sequentially open subset is in fact open. The {\em sequential saturation} $X_\sub{seq}$ of the space $X$ is the topological space whose underlying set is the same as $X$, but such that a subset of $X_\sub{seq}$ is open if and only if it is sequentially open in $X$. It is not hard to check that $X_\sub{seq}$ is a well-defined, sequentially-saturated topological space with the following universal property: a map from a sequentially saturated space $Y$ to $X$ is continuous if and only if the induced map $Y\to X_\sub{seq}$ is continuous. (In other words, the inclusion of $\op{Top}_\sub{ss}$, the category of sequentially saturated topological spaces, into $\op{Top}$, the category of all topological spaces, has a right adjoint $X\mapsto X_\sub{seq}$.) Moreover, a sequentially continuous map between topological spaces will induces a continuous map on their sequential satuations.

\subsubsection{Ind-Pro approach to the higher topology}
Section \ref{section_regular} is a prerequisite for this material.

Suppose that $A$ is a commutative ring equipped with a topology. Given an ideal $I$ of $A$, we may topologise $\hat A=\projlim_rA/I^r$ by giving it the weakest topology such that the projection map $\hat A\to A/I^r$ is continuous for every $r\ge 1$, where we give $A/I^r$ the quotient topology from $A$. Similarly, if $S\subseteq A$ is a multiplicative system, we may topologise $S^{-1}A$ by giving it the strongest topology for which $A\to S^{-1}A,\,a\mapsto a/s$ is continuous for all $s\in S$.

Define a {\em topological regular $n$-chain} to be a regular $n$-chain $\ul A=(A,\frak p_n,\dots,\frak p_0)$ together with a topology on the ring $A$. Given a topological regular $n$-chain $\ul A$, we have just explained in the previous paragraph how to naturally construct a topology on the $\frak p_n$-adic completion $\hat A$ and on the localisation $A_{\frak p_{n-1}}$. Therefore $\op{comp}\ul A$ and $\op{loc}\ul A$ are again topological regular chains. Proceeding recursively, this defines a topology on $HL(\ul A)$, the higher dimensional field associated to the regular $n$-chain $\ul A$.

\begin{theorem}
Let $F$ be an $n$-dimensional local field for which $\res F$ has finite characteristic, and suppose that $F=HL(\ul A)$ for some regular $n$-chain $\ul A$. Give $A$ the adic topology associated to its maximal ideal, so that $\ul A$ become a topological regular $n$-chain.

Then the topology just defined on $F=HL(\ul A)$ coincides with the higher topology (at least up to sequential saturation).
\end{theorem}
\begin{proof}
This is an unpublished result of the author.
\end{proof}

\section{A summary of higher dimensional local class field theory}\label{section_CFT}
Classical local class field theory states that if $K$ is a (one-dimensional) local field, then there is a natural reciprocity homomorphism \[\rec_K:\mult K\To\Gal(K^\sub{ab}/K),\] where $K^\sub{ab}$ is the maximal abelian extension of $F$, which satisfies various functorialities and which induces an isomorphism \[\mult K/N_{L/K}(\mult L)\isoto\Gal(L/K)\] whenever $L$ is a finite abelian extension of $K$.

Moreover, the theory of finite fields (zero-dimensional local fields) tells us that if $\bb F$ is a finite field then there is a natural homomorphism \[\rec_{\bb F}:\bb Z\to\Gal(\bb F^\sub{ab}/\bb F),\quad 1\mapsto\op{Frob}_{\bb F}\] with similar functorial properties and which induces an isomorphism \[\bb Z/|\bb L/\bb F|\bb Z\isoto\Gal(\bb L/\bb F)\] whenever $\bb L$ is a finite (automatically abelian) extension of $\bb F$.

Summarising the situation so far, we may draw the following table:
\begin{quote}
\begin{tabular}{l|c}
Type of field & Domain of reciprocity map for class field theory\\ \hline
finite field & $\bb Z$\\
(one-dimensional) local field & its multiplicative group\\ 
two-dimensional local field & ??\\
three-dimensional local field & ??\\
$\vdots$ & $\vdots$
\end{tabular}
\end{quote}
A crowning achievement in the theory of higher dimensional local fields was the discovery of the objects ?? and the development of a class field theory for higher dimensional local fields entirely comparable to that for usual local fields. The theory was discovered independently by A.~Parshin and K.~Kato; historical comments and references may be found in the remarks in section \ref{subsection_hlcft} below.

The objects ?? which take over the role of the multiplicative group in higher dimensions are Milnor $K$-groups, and so before stating the main results of higher dimensional local class field theory we must summarise that theory; other introductions are \cite[IV]{Fesenko2002} and \cite[Chap.~7]{Gille2006}.

\subsection{Milnor $K$-theory}\label{subsection_Milnor_K_theory}
Algebraic $K$-theory has a long and interesting history, but only several key steps are important to us. By 1967 there were definitions of $K_0$, $K_1$, and $K_2$ of a ring, due to A.~Grothendieck, H.~Bass et al., and J.~Milnor respectively; we refer the interested reader to C.~Weibel's historical survey \cite{Weibel1999} for more details. If $F$ is a field then it follows rapidly from the definitions that there are isomorphisms \[K_0(F)\cong\bb Z,\quad\quad K_1(F)\cong\mult F.\] Moreover, given $x,y\in F$, R.~Steinberg's earlier work \cite{Steinberg1962} showed how to define an element $\{x,y\}\in K_2(F)$ and it implied that these elements are bilinear in $x,y$ and satisfy the {\em Steinberg relation} \[\{x,1-x\}=0. \tag{$x\in F\setminus\{0,1\}$}\] H.~Matsumoto \cite{Matsumoto1969} subsequently showed that these elements and relations give a presentation for $K_2(F)$.

Inspired by Matsumoto's result, Milnor made the following definition \cite{Milnor1970}:

\begin{definition}
Let $F$ be an arbitrary field, and $m\ge 0$ an integer. The {\em $m^\sub{th}$ Milnor $K$-group of $F$}, denoted (in these notes) $K_m(F)$, is \[K_m(F)=\underbrace{\mult{F}\otimes_{\bb{Z}}\dots\otimes_{\bb{Z}}\mult{F}}_\sub{$m$ times}/I\] where $I$ is the subgroup of $(\mult{F})^{\otimes m}$ generated by the Steinberg relation \[\{a_1\otimes\dots\otimes a_m: a_1,\dots,a_m\in\mult{F}\mbox{ such that }a_i+a_j=1\mbox{ for some }i\neq j\}.\] In the slightly degenerate cases of $m=0,1$, we have $K_0(F)=\bb{Z}$ and $K_1(F)=\mult{F}$, in accordance with the notation above, while Matsumoto's result implies that our new definition of $K_2(F)$ agrees with the old definition.

The image of a pure tensor $a_1\otimes\dots\otimes a_m$ in $K_m(F)$ is denoted $\{a_1,\dots,a_m\}$ and is called a {\em symbol}. The natural map \[\mult{F}\times\cdots\times\mult F\to K_m(F)\,\quad(a_1,\dots,a_m)\mapsto\{a_1,\dots,a_m\}\] is called {\em the symbol map}.

If $A$ is any abelian group, then a function $f:(\mult{F})^m\to A$ which is multilinear and satisfies $f(a_1,\dots,a_m)=0$, whenever $a_i+a_j=1$ for some $i\neq j$, is called a {\em Steinberg map}.
\end{definition}

In other words, $K_m(F)$ is the abelian group generated by symbols $\{a_1,\dots,a_m\}$, where $a_i\in\mult{F}$, subject to the relations 
\begin{description}
\item Multilinearity: $\{a_1,\dots,a_ia_i',\dots,a_m\}=\{a_1,\dots,a_i,\dots,a_m\}+\{a_1,\dots,a_i',\dots,a_m\}$
\item Steinberg relation: $\{a_1,\dots,a_m\}=0$ if $a_i+a_j=1$ for some $i\neq j$.
\end{description}

If $f:(\mult{F})^m\to A$ is a Steinberg map, then the defining relations for $K_m(F)$ imply that there is a unique group homomorphism $f^M:K_m(F)\to A$ making the following diagram commute
\[\xymatrix{
(\mult{F})^m \ar[r]^f\ar[d]_{\{\;\}} & A\\
K_m(F)\ar[ur]_{f^M} &
}\] where the vertical arrow is the symbol map. That is, the symbol map $(\mult{F})^m\to K_m(F)$ is the universal Steinberg map. We will often abuse notation and write $f$ instead of $f^M$.

\begin{example}\label{example_K_groups_of_local_field}
Let $K$ be a local field and let $\mu\subset K$ be the group of roots of unity inside it; put $m=|\mu|$. Then the Hilbert symbol $H:\mult K\times\mult K\To\mu$ is surjective, bilinear, and satisfies the Steinberg identity (e.g.~\cite[XIV.\S2]{Serre1979} or \cite[IV.5.1]{Fesenko2002}), thus inducing a surjective homomorphism $H:K_2(K)\to\mu$. A theorem of C.~Moore \cite{Moore1968} states that $\ker H=mK_2(K)$ and that this kernel is an uncountable, divisible group (even uniquely-divisible, by \cite{Merkurjev1983}); moreover, $K_2(K)\to\mu$ splits. In other words, \[K_2(K)\cong\mu\oplus\mbox{an uncountable, uniquely-divisible group.}\]

If $m>2$ then $K_m(K)$ is simply an uncountable, uniquely-divisible group \cite{Sivitskii1985}.
\end{example}

\begin{example}\label{example_K_groups_of_reals}
The map
\begin{align*}
\op{sign}:(\mult{\bb{R}})^m&\to\{\pm1\}\\
(x_1,\dots,x_n)&\mapsto\begin{cases}-1& \mbox{if }x_1,\dots,x_n<0 \\ 1&\mbox{otherwise}\end{cases}
\end{align*}
may be checked to be multilinear and Steinberg, thus inducing a homomorphism $\op{sign}:K_m(\bb R)\to\{\pm 1\}$. The kernel of this homomorphism is divisible if $m\ge 2$, though the author does not know if it is uniquely so.
\end{example}

The usual norm map $N_{L/F}:\mult{L}\to\mult{F}$ for a finite extension $L/F$ of fields generalises to Minor $K$-groups. Historically, it was first defined on Milnor $K_2$ (as well as $K_0$ and $K_1$, using the original definitions of the $K$-groups) by J.~Milnor \cite[\S14]{Milnor1970}, and subsequently extended to $K_n$, all $n\ge 0$, by J.~Tate and H.~Bass \cite[Chap.~1, \S5]{Bass1973}, though they could not show in general that their definition was independent of a choice of generators for $L/F$. The well-definedness was established by K.~Kato \cite[\S1.7]{Kato1980}; Kato distributed his result at the 1980 Oberwolfach Algebraic $K$-theory conference, and A.~Suslin \cite{Suslin1982} immediately used it to prove that the Milnor $K$-theory of a field embeds into the Quillen $K$-theory, at least modulo torsion.

Kato's proof of the existence and uniqueness (once characterising properties are specified) of the norm map on Milnor $K$-theory is rather technical; complete proofs may be found in \cite[IV]{Fesenko2002} or \cite[7.3]{Gille2006}. We content ourselves with summarising the main properties:

\begin{proposition}
For each finite extension $L/F$ of fields there is a group homomorphism \[N_{L/F}:K_m(L)\to K_m(F),\] satisfying the following properties (among others):
\begin{enumerate}
\item When $m=1$, the norm map on Milnor $K$-groups $K_1(L)\to K_1(F)$ coincides with the usual norm map of fields $\mult{L}\to\mult{F}$.
\item If $E$ is an intermediate extension, then $N_{L/F}=N_{E/F}\circ N_{L/E}:K_M(L)\to K_m(E)\to K_m(F)$.
\item The composition $K_m(F)\to K_m(L)\stackrel{N_{L/F}}{\To} K_m(F)$ coincides with multiplication by $|L/F|$.
\item if $\{b_1\dots,b_m\}$ is a symbol in $K_m(L)$, with $b_1,\dots,b_{i}\in\mult{L}$ and $b_{i+1},\dots,b_m\in\mult{F}$, then \[N_{L/F}(\{b_1,\dots,b_m\})=N_{L/F}(\{b_1,\dots,b_i\})\{b_{i+1},\dots,b_m\}\in K_m(F),\] the product of a norm from $K_i(L)$ and a symbol in $K_{m-i}(F)$.
\end{enumerate}
\end{proposition}

When $F$ is a discrete valuation field, its Milnor $K$-groups inherit additional structure, the most important of which is the border (or boundary) map \[\bor:K_n(F)\to K_n(\res F),\] where $\res F$ is the residue fields of $F$ with respect to its valuation:

\begin{proposition}[Milnor \cite{Milnor1970}]\label{lemma_the_second_border_map}
Let $F$ be a discrete valuation field. There is a unique group homomorphism \[\bor:K_m(F)\to K_{m-1}(\res{F})\] which satisfies \[\bor(\{u_1,\dots,u_{m-1},\pi\})=\{\res{u}_1,\dots,\res{u}_{m-1}\}\] for all uniformisers $\pi\in\mult{F}$ and all $u_i\in\mult\roi_F$.
\end{proposition}
\comment{
Still supposing that $F$ is a discrete valuation field, one writes $UK_n(F)$ for the subgroup of $K_n(F)$ generated by symbols of the form $\{u_1,\dots,u_n\}$, where $u_i\in U_F$. The border map of the previous lemma induces a short exact sequence \[0\To UK_n(F)\To K_n(F)\stackrel{\bor}{\To}K_{n-1}(\res F)\To 0\]
}

\begin{remark}
Suppose that $F$ is a discrete valuation field. Then the classical {\em tame symbol} is the map \[\mult F\times\mult F\to\mult{\res F},\quad (u\pi^r,vt\pi^s)\mapsto (-1)^{rs}\res{uv^{-1}},\] where $u,v\in\mult{\roi_F}$, $r,s\in\bb Z$, and we have fixed a uniformiser $\pi$. This is a bilinear Steinberg map, independent of $\pi$, and so induces a natural homomorphism \[K_2(F)\to\mult{\res F}=K_1(\res F).\] This is exactly the border map of the proposition, in the case $m=2$.
\end{remark}

\begin{example}
For each prime number $p$, let $\bor_p:K_2(\bb Q)\to\mult{\bb F_p}$ be the border map for the p-adic valuation provided by the proposition or previous example. Let $\bor_\infty$ denote the composition \[\bor_\infty:K_2(\bb Q)\to K_2(\bb R)\stackrel{\op{sign}}{\To}\{\pm 1\},\] where the $\op{sign}$ homomorphism was defined in example \ref{example_K_groups_of_reals}. Then \[\bor_\infty\oplus\bigoplus_p\bor_p:K_2(\bb Q)\To\{\pm1\}\oplus\bigoplus_p\mult{\bb F_p}\] is an isomorphism (the proof may be found as an exercise in \cite[IX.2]{Fesenko2002}).

A theorem of H.~Garland \cite{Garland1971} states that the analogous homomorphism for an arbitrary number field in place of $\bb Q$ has finite kernel.
\end{example}

\comment{
\begin{remark}[Compatibility between norm and border maps]
Suppose that $L/F$ is a finite extension of complete discrete valuation fields; then the following diagram commutes:
\[\xymatrix{
K_{m+1}(L) \ar[d]_{N_{L/F}} \ar[r]^{\bor_L} & K_m(\res{L}) \ar[d]^{N_{\res{L}/\res{F}}}\\
K_{m+1}(F) \ar[r]_{\bor_F} & K_m(\res{F})
}\]
where $\bor_F$, $\bor_L$ are the border maps on $F$, $L$ associated to their discrete valuations.
\end{remark}
}

\subsection{Higher dimensional local class field theory: statement and sketches}\label{subsection_hlcft}
With the definitions of Milnor $K$-theory completed, we may state the main theorem of higher dimensional local class field theory, which the reader may compare with the classical statement for (one-dimensional) local fields (e.g.~\cite[XIII.\S4]{Serre1979} or \cite[IV.4]{Fesenko2002}).

\begin{theorem}[A.~Parshin, K.~Kato]
Given any $n$-dimensional local field $F$ there is a natural `reciprocity' homomorphism \[\rec_F:K_n(F)\To\Gal(F^\sub{ab}/F).\] These have the following properties:
\begin{enumerate}
\item (Functoriality) Given a finite extension $L/F$ of $n$-dimensional local fields, the following two diagrams commute:
\[\xymatrix{
K_n(L) \ar[r]^{\rec_L}\ar[d]_{N_{L/F}} & \Gal(L^\sub{ab}/L)\ar[d]^{\sub{restriction}}\\
K_n(F) \ar[r]_{\rec_F} & \Gal(F^\sub{ab}/F)
}\qquad
\xymatrix{
K_n(L) \ar[r]^{\rec_L} & \Gal(L^\sub{ab}/L)\\
K_n(F) \ar[r]_{\rec_F}\ar[u] & \Gal(F^\sub{ab}/F)\ar[u]_{\sub{transfer}}
}
\]
If $L/F$ is moreover an abelian extension then the induced homomorphism \[K_n(F)/N_{L/F}(K_n(L))\stackrel{\rec_F}{\To} \Gal(L/F)\] is an isomorphism.
\item (Compatibility with residue field) The following diagram commutes:
\[\xymatrix{
K_n(F)\ar[r]^{\rec_F}\ar[d]_\bor & \Gal(F^\sub{ab}/F)\ar[d]^{\sigma\mapsto\res\sigma}\\
K_{n-1}(\res F)\ar[r]_{\rec_{\res F}} & \Gal(\res F^\sub{ab}/\res F)
}\]
where the bottom arrow is the reciprocity map for the $n-1$ -dimensional local field $\res F$, and the right arrow comes from restricting an automorphism $\sigma$ to the maximal unramified abelian extension of $F$.
\item (Normalisation) If $F$ is a finite field, i.e.~a zero-dimensional local field, then the reciprocity map is given by \[\rec_F:K_0(F)=\bb Z\To\Gal(F^\sub{ab}/F),\quad 1\mapsto\op{Frob}_F.\]
\item (Existence theorem) See remarks \ref{remark_Kato} and \ref{remark_Fesenko} below.
\end{enumerate} 
\end{theorem}

The first glimpse of higher dimensional local class field theory was in work by Y.~Ihara \cite{Ihara1968}, who studied degree $p$ cyclic extensions of the field $\bb Q_p\{\{t\}\}$. Around ten years later, Parshin and Kato independently discovered higher dimensional local class field theory using Milnor $K$-theory, and in the following remarks we will describe their, and others', approaches.

A more detailed survey of higher dimensional local class field theory is \cite{Fesenko1996}.

\begin{remark}[Parshin's approach]\label{remark_Parshin}
A.~Parshin \cite{Parshin1975, Parshin1984, Parshin1990} established higher dimensional local class field theory in finite characteristic by developing a generalisation of the one-dimensional Artin-Schreier-Witt theory due to Y.~Kawada and I.~Satake \cite{Kawada1956}. First recall that if $F$ is any field of characteristic $p\neq 0$ then there is the non-degenerate Witt pairing for any $r\ge 1$ \[\Gal(F_r/F)\times W_r(F)/(\op{Fr}-1)W_r(F)\To \bb Z/p^r\bb Z,\] where $W_r(F)$ is the ring of $r$-truncated Witt vectors of $F$, $\op{Fr}$ is the usual Frobenius operator on it, and $F_r$ is the maximal abelian $p^r$-exponent extension of $F$. When $F$ is an $n$-dimensional local field of characteristic $p$, Parshin explicitly defined a non-degenerate, continuous pairing \[K_n^\sub{top}(F)/p^rK_n^\sub{top}(F)\times W_r(F)/(\op{Fr}-1)W_r(F)\to \bb Z/p^r\bb Z,\] where $K_n^\sub{top}(F)$ denotes the topological Milnor $K$-group of section \ref{subsection_top_Milnor_K_groups}. From this one obtains a $p^r$-approximation of the reciprocity map \[K_n^\sub{top}(F)/p^rK_n^\sub{top}(F)\To \Gal(F_r/F).\]

The tame and unramified parts of the reciprocity map can be treated much more easily, and, letting $r\to\infty$, they all stitch together to produce the complete reciprocity map $\rec_F:K_n^\sub{top}(F)\to\Gal(F^\sub{ab}/F)$.
\end{remark}

\begin{remark}[Kato's approach]\label{remark_Kato}
K.~Kato \cite{Kato1979, Kato1980, Kato1982, Kato1996} established higher dimensional local class field theory in general by using the same type of cohomological argument as J.~Tate \cite{ArtinTate1968} used in dimension one, as follows. Firstly, for any field $F$ and any integer $\ell$ not divisible by its characteristic, the Kummer isomorphism \[\mult F/F^{\times\ell}\isoto H^1(F,\mu_{\ell})\] is known to be a Steinberg symbol and so it gives rise to the so-called cohomological symbol \[K_n(F)/\ell K_n(F)\To H^n(F,\mu_{\ell}^{\otimes n}).\] Taking a cup product then yields a natural pairing \[K_n(F)/\ell K_n(F)\times H^1(F,\bb Z/\ell\bb Z)\To H^{n+1}(F,\mu_{\ell}^{\otimes n}).\]

Kato's key step was to show that if $F$ is an $n$-dimensional local field then it has cohomological dimensional $n+1$ and there is a natural isomorphism \[H^{n+1}(F,\mu_{\ell}^{\otimes n})\cong\frac{1}{\ell}\bb Z/\bb Z.\] Assuming that $F$ has characteristic $0$, one takes the limit over all $\ell$ and thus arrives at a pairing \[K_n(F)\times H^1(F,\bb Q/\bb Z)\To\bb Q/\bb Z;\] Pontryagin duality immediately produces the reciprocity map \[K_n(F)\to\Gal(F^\sub{ab}/F).\] When $F$ has finite characteristic, one should replace certain of these Galois cohomology groups by de-Rham-Witt cohomology groups.

Kato later proved the Existence theorem in \cite{Kato2000}: There is a bijective correspondence between finite abelian extensions of $F$ and `finite index, open subgroups' of $K_n(F)$. Preferring not to use topological Milnor $K$-groups, Kato gave a definition of the `finite index, open subgroups' of $K_n(F)$ using an Ind-Pro formalism which he developed.
\end{remark}

\begin{remark}[Fesenko's approach]\label{remark_Fesenko}
I.~Fesenko \cite{Fesenko1991b, Fesenko1991, Fesenko1991a} generalised J.~Neukirch's \cite[II]{Neukirch1986} explicit approach to local class field theory to the higher dimensional setting as follows. Let $F$ be an $n$-dimensional local field, and let $p$ denote the characteristic of the final residue field of $F$. Given any finite extension $L$ of $F$, let $\tilde L$ be the extension of $L$ obtained by adjoining all roots of unity of order coprime to $p$; this is the {\em maximal purely unramified extension} of $L$, and $\Gal(\tilde L/L)$ is procyclic, generated by the Frobenius.

Let $L$ be a finite abelian extension of $F$. Given $\sigma\in\Gal(L/F)$, there exists $\tilde\sigma\in\Gal(\tilde L/F)$ such that $\tilde\sigma$ extends $\sigma$ and such that $\tilde\sigma|_{\tilde F}$ is a strictly positive integer power of the Frobenius. Let $\Sigma$ be the fixed field of $\tilde\sigma$ (it is a finite extension of $F$), let $t_1,\dots,t_n$ be a system of local parameters for $\Sigma$, and define \[{\bf r}_{L/F}(\sigma):=N_{\Sigma/F}(\{t_1,\dots,t_n\})\mbox{ mod }N_{L/F}(K_n^\sub{top}(L))\,\in K_n^\sub{top}(F)/N_{L/F}(K_n^\sub{top}(L)).\] Fesenko showed that this does not depend on the choice of $\tilde\sigma$ or of the local parameters, and that the resulting map \[{\bf r}_{L/F}:\Gal(L/F)\To K_n^\sub{top}(F)/N_{L/F}(K_n^\sub{top}(L))\] is an isomorphism. Its inverse is the reciprocity map for the extension $L/F$.

To verify the properties of ${\bf r}_{L/F}$, Fesenko analysed the structure of $K_n^\sub{top}(F)$ using explicitly-defined pairings and topological expansions of elements in $K_n^\sub{top}(F)$. In finite characteristic these are the same pairings as appear in Parshin's approach; in mixed characteristic he used S.~Vostokov's higher dimensional generalisation \cite{Vostokov1985} of his explicit formula \cite{Vostokov1978} for the classical Hilbert symbol.

Fesenko also established the Existence theorem using his framework of topological Milnor $K$-groups: There is a bijective correspondence between the finite abelian extensions of $F$ and the finite index open subgroups of $K_n^\sub{top}(F)$, given by $L\mapsto N_{L/F}(K_n^\sub{top}(L))$.
\end{remark}

\begin{remark}[Koya and Spiess' approach]
Y.~Koya \cite{Koya1990, Koya1993} proved higher dimensional local class field theory using the language of class formations and complexes of Galois modules. His methods were further developed by M.~Speiss \cite{Spiess1997, Spiess2000}, who gave the first proof that if $L/F$ is a finite Galois extension (not necessarily abelian) of $n$-dimensional local fields, then \[K_n(F)/N_{L/F}(K_n(L))\cong\Gal(L/F)^\sub{ab}.\]
\end{remark}

\section{Constructing higher dimensional fields: the regular case}\label{section_regular}
The aim of this section is to explain, in the easiest case, the standard semi-global construction of higher dimensional fields; in a sense, this is the natural way in which higher dimensional fields arise. The one-dimensional case is straightforward and well-known: if $\roi$ is a one-dimensional, Noetherian, regular local ring, then it is a discrete valuation ring and so \[F=\Frac\hat\roi\] is a complete discrete valuation field. Observe that we formed $F$ by first {\em completing} $\roi$ and then {\em localising} by passing to the field of fractions. To generalise this construction to higher dimensions we will iterate these processes.

For example, suppose that $A$ is a two-dimensional, regular, local ring, essentially of finite type over $\bb Z$, with maximal ideal $\frak m$, and that we choose a non-maximal, non-zero prime ideal $\frak p$ for which $A/\frak p$ is regular; consider the following sequence of localisations and completions:
\[\begin{array}{ccccccccccc}
A&\!\leadsto\!&\hat{A}&\!\leadsto\!&\left(\hat{A}\right)_{\hat{\frak p}}&\!\leadsto\!&\hat{\left(\hat{A}\right)_{\hat{\frak p}}}&\!\leadsto\!&\mbox{Frac}\left(\hat{\left(\hat{A}\right)_{\hat{\frak p}}}\right)\\
&&&&&&\parallel&&\parallel\\
&&&&&&A_{\frak{m},\frak{p}}&&F_{\frak{m},\frak{p}}
\end{array}\]
which we now explain in more detail. First we $\frak m$-adically complete $A$, and then we localise away from $\hat{\frak p}=\frak p\hat A$ (the regularity of $A/\frak p$ ensures that $\hat{\frak p}$ is a prime ideal of $\hat A$). The result is a regular one-dimensional local ring, which we complete to obtain $A_{\frak{m},\frak{p}}$, a complete discrete valuation ring. Its field of fractions $F_{\frak{m},\frak{p}}$ is a two-dimensional local field.

In this section we will carefully show that this continued localisation-completion process works in arbitrary dimensions, in suitably regular situations: the input is a `regular $n$-chain' $\ul A$, namely a local ring equipped with a suitably regular chain of prime ideals, and the output is $HL(\ul A)$, a field of cdvdim $\ge n$ and with $n^\sub{th}$ residue field equal to $k(A)$. Geometrically, this corresponds to picking a flag of irreducible subschemes $\xi$ on a scheme $X$, in which case the resulting field is often denoted $F_\xi$; see remark \ref{remark_geometric_interpretation}.

In the next section we will discard the regularity hypotheses and describe the construction in general.

\begin{remark}
We require several definitions and results which may be found in any standard text on commutative algebra (e.g.~\cite{Matsumura1989}). A degree of familiarity with these ideas is expected. Let $A$ be a Noetherian local ring with maximal ideal $\frak{m}$ and residue field $k$.

{\em Regularity.} $A$ is said to be regular if and only if $\dim{A}=\dim_k{\frak{m}/\frak{m}^2}$. Geometrically, letting $X=\Spec{A}$ and $x=\frak{m}\in X$, this says that the dimension of the cotangent space at $x$ equals the dimension of $X$. If $A$ is regular then it is a domain. If $A$ is regular then a localisation of $A$ away from a prime ideal is also regular.

{\em Completions.} Let $\hat{A}$ be the $\frak{m}$-adic completion of $A$. Then the natural map $A\to\hat{A}$ is faithfully flat (hence injective); in particular, this implies that for any ideal $I\subset A$, the natural map $A/I\to\hat{A}/I\hat{A}$ induces an isomorphism $\widehat{A/I}\cong\hat{A}/I\hat{A}$. Moreover, $\dim A=\dim\widehat{A}$, and $A$ is regular if and only if $\hat{A}$ is regular.
\end{remark}

\begin{lemma}\label{lemma_regularity_completion_preserves_primality}
Let $A$ be a Noetherian local ring. If $\frak{p}\subset A$ is a prime ideal of $A$ such that the local ring $A/\frak{p}$ is regular then $\frak{p}\hat{A}$ is a prime ideal of $\hat{A}$ such that $\hat{A}/\frak{p}\hat{A}$ is regular.
\end{lemma}
\begin{proof}
Using the remarks above we see that $\hat{A}/\frak{p}\hat{A}\cong\widehat{A/\frak{p}}$ is the completion of a regular local ring, hence is regular, and hence is a domain.
\end{proof}

Let $A$ be a Noetherian local ring and $\frak{p}\subset A$ a prime ideal; it is convenient to say that $\frak{p}$ is regular if and only if $A/\frak{p}$ is regular. Geometrically, this means $V(\frak p)$ is a regular subscheme of $\Spec A$.

The data which are required for our first construction of higher dimensional fields are regular Noetherian local rings equipped with a suitable chain of prime ideals: To be precise, we consider tuples \[\underline{A}=(A,\frak{p}_n,\dots,\frak{p}_0)\] where $A$ is a Noetherian, $n$-dimensional, regular local ring and the $(\frak{p}_i)_i$ are regular prime ideals which form a complete flag of primes; i.e., \[A\supset\frak{p}_n\supset\dots\supset\frak{p}_0.\] The fact that $A$ is local implies that $\frak{p}_n$ is the maximal ideal of $A$, and the fact that $A$ is a domain implies that $\frak{p}_0=0$.

For brevity we will refer to a piece of data $\ul{A}$ satisfying these conditions as {\em a regular chain of length $n$} (or {\em regular $n$-chain}). The residue field $A/\frak{p}_n$ of $A$ will be of particular interest, so we will denote it by $k(\ul{A})=k(A)$.

\begin{example}\label{example_low_dimensions}
Consider the low dimensional examples:
\begin{enumerate}
\item $n=0$. A regular chain of length $0$ is obviously is just a field and its maximal ideal $0$: \[\ul F=(F,0).\] It is convenient to identify a regular chain of length $0$ with the field itself.
\item $n=1$. A regular chain of length $1$ is the data of a discrete valuation ring $\roi$, its maximal ideal $\frak p$, and the zero ideal: \[\ul{\roi}=(\roi,\frak{p},0).\]
\item $n=2$. In the previous two cases, the data was completely determined by the choice of the ring; that is no longer the case when $n\ge 2$. A regular chain of length $2$ consists of a Noetherian, two-dimensional, regular local ring $A$, its maximal ideal $\frak m$, a {\em choice} of a regular, non-zero, non-maximal prime ideal $\frak p$, and the zero ideal: \[\ul{A}=(A,\frak m,\frak p,0).\]

For example, let $A=\bb{Z}[t]_{\langle p,t\rangle}$ be the localization of $\bb{Z}[t]$ away from the maximal ideal generated by $t$ and a prime number $p\in\bb{Z}$. Then $(A,\frak{m},tA,0)$ and $(A,\frak{m},pA,0)$ are regular chains of length $2$: all that needs to be noticed is that $A/tA=\bb{Z}_{\pid p}$ and $A/pA=\bb{F}_p[t]_{\pid t}$ are regular.

As a second, more geometric, example, take $A=k[t_1,t_2]_{\pid{t_1,t_2}}$ (with $k$ a field), and let $\frak p$ be the ideal generated by either $t_1$ or $t_2$.
\end{enumerate}
\end{example}

\begin{remark}[Local parameters]\label{remark_local_parameters}
If $\ul A$ is a regular $n$-chain, then there exists a sequence of elements $t_1,\dots,t_n\in A$ such that \[\frak p_n=\pid{t_1,\dots,t_n},\,\frak p_{n-1}=\pid{t_2,\dots,t_n},\dots,\,\frak p_1=\pid{t_n},\,\frak p_0=0.\] Conversely, if $A$ is a regular local ring and $t_1,\dots,t_n$ are a regular sequence (i.e.~they generate the maximal ideal), then defining $\frak p_i$ in this way produces a regular $n$-chain. The relationship with sequences of local parameters in higher dimensional fields will be given in remark \ref{remark_local_parameters_2}.
\end{remark}

The two main ways of manipulating the data $\underline{A}$ are completing and localising. Firstly completing: \[\op{comp}\underline{A}:=(\hat{A},\frak{p}_n\hat{A},\dots,\frak{p}_0\hat{A}).\] That is, one completes $A$ with respect to its maximal ideal $\frak{p}_n$ and then pushes forward each of the primes. Secondly, localising: \[\op{loc}\underline{A}:=(A_{\frak{p}_{n-1}},\frak{p}_{n-1}A_{\frak{p}_{n-1}},\dots,\frak{p}_0A_{\frak{p}_{n-1}}).\] That is, one localises away from $\frak{p}_{n-1}$ and then pushes forward each of the subsequent primes.

The fundamental result is that these two processes preserve the desired properties of the data:

\begin{lemma}\label{lemma_comp_and_loc_preserve_reg_chains}
If $\ul{A}$ is a regular chain of length $n$ then $\op{comp}\ul A$ is a regular chain of length $n$ and $\op{loc}\ul A$ is a regular chain of length $n-1$. Moreover, regarding residue fields, \[k(\op{comp}\ul A)=k(\ul A),\qquad k(\op{loc}\ul A)=\Frac(A/\frak p_{n-1}).\]
\end{lemma}
\begin{proof}
We first consider the completion process. If $\ul{A}$ is a regular chain of length $n$ then we saw in lemma \ref{lemma_regularity_completion_preserves_primality} that $\frak{p}_i\hat{A}$ is a regular prime ideal of $\hat{A}$ for $i=0,\dots,n$. Secondly our remarks above tell us that $\hat A$ is regular and of dimension $n$. Finally, injectivity of $A/\frak p_i\to\hat{A/\frak p_i}=\hat A/\frak p_i\hat A$ implies that $A\cap\frak{p}_i\hat{A}=\frak{p}_i$ for each $i$, so that $\frak p_0\hat A\subseteq\cdots\subseteq \frak p_n\hat A$ is a strictly increasing chain. Therefore $\op{comp}\underline{A}$ is another regular $n$-chain. Regarding residue fields, \[k(\op{comp}\ul A)=\hat{A}/\frak{p}_n\hat A=A/\frak p_n=k(A).\]

The localisation process is even easier. The isomorphism \[A_{\frak{p}_{n-1}}/\frak{p}_iA_{\frak{p}_{n-1}}\cong(A/\frak{p}_i)_{\frak{p}_{n-1}/\frak{p}_i}\] and the fact that localisation preserves regularity imply that $\frak{p}_iA_{\frak{p}_{n-1}}$ is a regular prime of the regular local ring $A_{\frak{p}_{n-1}}$, for $i=1,\dots,n-1$. Moreover, $\dim A_{p_{n-1}}=n-1$. Finally, $A\cap \frak{p}_iA_{\frak{p}_{n-1}}=\frak{p}_i$ for each $i<n$, so that $\frak p_0A_{\frak{p}_{n-1}}\subseteq\cdots\subseteq\frak p_{n-1}A_{\frak{p}_{n-1}}$ is a strictly increasing chain. Therefore $\op{loc}\ul{A}$ is a regular chain of length $n-1$, with residue field \[k(\op{loc}\ul A)=A_{\frak{p_{n-1}}}/\frak p_{n-1}A_{\frak{p_{n-1}}}=\Frac(A/\frak p_{n-1}).\qedhere\]
\end{proof}

Suppose that $\ul A$ is a regular $n$-chain. Then, according to the lemma, $\op{loc}\op{comp}\ul A$ is a regular chain of length $n-1$. So, iterating these procedures, we see that \[\op{HL}(\ul A):=\underbrace{\op{loc}\op{comp}\cdots \op{loc}\op{comp}}_{\sub{`loc\,comp' $n$ times}}\ul A\] is a regular chain of length $0$, which, in accordance with example \ref{example_low_dimensions}(i), is a field; we will soon see that this is the higher dimensional field associated to $\ul A$.

\begin{example}\label{example_low_dimensions2}
Again we examine what is going on in low dimensional examples:
\begin{enumerate}
\item $n=0$. Here nothing is happening: $\op{HL}(F,0)=F$.
\item $n=1$. Here we recover the process explained at the start of the section: \[\op{HL}(\roi,\frak p,0)=\op{loc}\op{comp}(\roi,\frak p,0)=\op{loc}(\hat{\roi},\frak{p}\hat{\roi},0)=\hat{\roi}_{\{0\}}=\Frac\hat{\roi}\]
\item $n=2$. Recall from example \ref{example_low_dimensions}(iii) that a typical regular chain \[\ul{A}=(A,\frak{m},\frak{p},0)\] of length $2$ consists of a two-dimensional regular local ring $A$, its maximal ideal $\frak m$, and any regular, non-zero, non-maximal, prime $\frak{p}$. Then \[\op{loc}\op{comp}\ul A=\op{loc}(\hat{A},\frak m\hat{A},\frak p\hat A,0)=((\hat{A})_{\frak p\hat A},\frak p(\hat{A})_{\frak p\hat A},0),\] which is the regular chain of length $1$ determined by the discrete valuation ring $(\hat{A})_{\frak p\hat A}$. Repeating the loc\,comp process, \[\op{HL}(\ul A)=\Frac\left(\hat{ (\hat{A})_{\frak{p}\hat{A}}}\right),\] which the reader should digest, perhaps returning to the discussion at the start of the section.
\end{enumerate}
\end{example}

Now for the main result of the section, which states that the process really does work:
\begin{theorem}
Let $\ul A$ be a regular chain of length $n$. The $HL(\ul A)$ is a field of $\cdvdim\ge n$. Moreover, its $n^\sub{th}$ residue field is $k(A)$.
\end{theorem}
\begin{proof}
The proof is of course by induction on $n$, with nothing to prove if $n=0$. So suppose $n\ge 1$ and that $\ul A$ is a regular $n$-chain. Then we have seen that $\op{loc}\op{comp}\ul A$ is a regular chain of length $n-1$, so the inductive hypothesis implies that $\op{HL}(\op{loc}\op{comp}\ul A)=\op{HL}(\ul A)$ has $\cdvdim\ge n-1$ with $n-1^\sub{st}$ residue field equal to \[k(\op{loc}\op{comp}\ul A)=\Frac(\hat{A}/\frak p_{n-1}\hat{A}).\] (The formula for the residue field comes from lemma \ref{lemma_comp_and_loc_preserve_reg_chains})

But $\hat{A}/\frak p_{n-1}\hat{A}$ is a complete discrete valuation ring with residue field $k(\ul A)$; thus the $n-1^\sub{st}$ residue field of $HL(\ul A)$ is a complete discrete valuation field. Therefore $\cdvdim HL(\ul A)\ge n$, and the $n^\sub{th}$ residue field of $HL(\ul A)$ is $k(\ul A)$.
\end{proof}

\begin{corollary}\label{corollory_essentially_of_finite_type_case}
Let $\ul{A}$ be a regular $n$-chain such that $A$ is essentially of finite type over $\bb{Z}$. Then $HL(\ul A)$ is an $n$-dimensional local field.
\end{corollary}
\begin{proof}
In this case $k(A)$ is a finite field.
\end{proof}

\begin{remark}\label{remark_local_parameters_2}
Suppose $\ul A$ is a regular $n$-chain and put $F=HL(\ul A)$. Let $t_1,\dots,t_n$ be a regular sequence describing it as in remark \ref{remark_local_parameters}. Then $t_1,\dots,t_n$ is a sequence of local parameters for $F$. We leave it to the reader to check this, by induction on $n$.
\end{remark}

\begin{example}
We consider our examples in dimension $2$. 
\begin{enumerate}
\item Let $p$ be a prime and put $A=\bb{Z}[t]_{\langle p,t\rangle}$, with maximal ideal $\frak{m}$ generated by $p,t$. We will explicitly describe $\op{HL}(\ul A)$, where $\ul{A}=(A,\frak m,\frak p,0)$, for two choices of prime $\frak p$, namely $\frak p=tA$ and $\frak p=pA$ (we saw in example \ref{example_low_dimensions2}(iii) that these are regular $2$-chains). To start, notice that $\hat A=\bb Z_p[[t]]$.
\begin{enumerate}
\item $\frak{p}=tA$. Then $\op{loc}\op{comp}\ul A$ is the regular chain of length 1 determined by the discrete valuation ring $\bb Z_p[[t]]_{\pid t}$, whose completion is \[\projlim_r(\bb Z_p[[t]]_{\pid t}/t^r\bb Z_p[[t]]_{\pid t}).\] We claim that $\bb Z_p[[t]]_{\pid t}/t^r\bb Z_p[[t]]_{\pid t}=\bb Q_p[[t]]/t^r\bb Q_p[[t]]$ for each $r$. Indeed, if $f\in\bb Z_p[[t]]\setminus\pid t$ then $f$ mod $\pid{t^r}$ is a non-zero divisor in the one-dimensional local ring $\bb Z_p[[t]]/\pid{t^r}$, and so $\bb Z_p[[t]]/\pid{t^r,f}$ is zero-dimensional; this implies that $p^k\in\pid{t^r,f}$ for $k\gg 0$. So $f\equiv p^k$ mod $\pid{t^r}$, which obviously suffices to show that $f$ is invertible in $(\bb Z_p[[t]]/t^r\bb Z_p[[t]])[\frac{1}{p}]=\bb Q_p[[t]]/t^r\bb Q_p[[t]]$. This completes the proof of the claim.

Therefore \[HL(A,tA,0)=\Frac\projlim_r\bb Q_p[[t]]/t^r\bb Q_p[[t]]=\Frac\bb Q_p[[t]]=\bb Q_p((t)).\]
\item $\frak{p}=pA$. A very similar argument to the previous case, but swapping the order of the parameters, shows that 
\begin{align*}
HL(A,pA,0)&=\Frac\projlim_r\bb Z_p((t))/p^r\bb Z_p((t)\\
	&=\Frac\bigg(\mbox{$p$-adic completion of }\bb Z_p((t))\bigg)=\bb Q_p\{\{t\}\}.
\end{align*}
\end{enumerate}
\item Let $k$ be a field, put $A=k[t_1,t_2]_{\pid{t_1,t_2}}$, and let $\frak p$ be the prime ideal of $A$ generated by $t_2$. Then another repetition of the same argument shows that $HL(\ul A)=k((t_1))((t_2))$.
\end{enumerate}
\end{example}

\begin{remark}[Flat $A$-algebra structure]\label{remark_augmentation}
It is easy to see that if $\ul A$ is a regular chain of length $n$, then there is a natural injective, ring homomorphism \[A\To\op{HL}(\ul A).\] Moreover, since completions and localisations are flat, this morphism is flat. \comment{We will see in proposition \ref{proposition_universality_via_ROIs} that this coaugmentation universally characterises the functor $HL$.}
\end{remark}

\begin{remark}[Geometric interpretation]\label{remark_geometric_interpretation}
Let $X$ be an $n$-dimensional, irreducible Noetherian scheme, and \[\xi=(y_n\subset y_{n-1}\subset \cdots\subset y_0)\] a complete flag of irreducible closed subschemes; so $\codim y_i=i$. Write $y_n=\{z\}$, and assume that $y_{n-1},\dots,y_0$ are regular at $z$; note that this includes the assumption that $z$ is a regular point of $X$, since $y_0=X$.

Put $A=\roi_{X,z}$ and let $\frak{p}_i\subset A$ be the local equation for $y_i$ at $z$. Then \[\ul{A}=(A,\frak p_n,\dots,\frak p_0)\] is a regular chain of length $n$, and thus $HL(\ul A)$ is a field of $\cdvdim\ge n$ with $n^\sub{th}$ residue field equal to $k(z)$. If the function field of $X$ is denoted $F$ then the higher dimensional field $HL(\ul A)$ is often denoted $F_\xi$.

According to corollory \ref{corollory_essentially_of_finite_type_case}, if $X$ is of finite type over $\bb Z$ then $F_\xi$ will be an $n$-dimensional local field.
\end{remark}

\begin{remark}[Weakening regularity to normality]
Suppose $A$ is an excellent (see remark \ref{remark_excellence} for a review of excellence) $n$-dimensional, normal local ring. Then the construction of this section continues to hold if we work with complete flags $A\supset\frak p_n\supset\cdots\supset\frak p_0$ for which $A/\frak p_i$ is normal for each $i$. This follows from the fact that the completion of an excellent, normal, local ring is again an excellent, normal, local ring.
\end{remark}

We will next describe all the residue fields of $HL(\ul A)$. For this it is useful to introduce a truncation operation:  if $\ul A=(A,\frak p_n,\dots,\frak p_0)$ is a regular $n$-chain and $0\le i\le n$, then $\tau_i\ul A:=(A/\frak p_i,\frak p_n/\frak p_i,\dots,\frak p_i/\frak p_i)$ is obviously a regular $i$-chain. In particular, we may identify $\tau_n\ul A$ with the residue field $k(A)$.

The reader should note that truncation commutes with localisation and completion: if $\ul A$ is a regular $n$-chain, then \[\tau_i\op{comp}\ul A=\op{comp}\tau_i\ul A\qquad(i=0,\dots,n)\] and \[\tau_i\op{loc}\ul A=\op{loc}\tau_i\ul A\qquad(i=0,\dots,n-1).\]

The following describes all the residue fields of $HL(\ul A)$:

\begin{proposition}\label{proposition_the_residue_fields}
Let $A$ be a regular $n$-chain. Then the $i^\sub{th}$ residue field of $HL(\ul A)$ equals $HL(\tau_i\ul A)$, for $i=0,\dots,n$.
\end{proposition}
\begin{proof}
When $n=0$ there is nothing to prove; so we may assume $n\ge 1$ and proceed by induction. Suppose for a moment that we knew that the first residue field of $HL(\ul A)$ were $HL(\tau_1\ul A)$. Then we could apply the inductive hypothesis to deduce that, for $i=1,\dots,n$, the $i-1$ residue field of $HL(\tau_1\ul A)$ is $HL(\tau_{i-1}\tau_1A)\,(=HL(\tau_i\ul A))$; i.e. the $i^\sub{th}$ residue field of $HL(\ul A)$ is $HL(\tau_i\ul A)$, which would complete the proof.

Therefore it remains only to show that the first residue field of $HL(\ul A)$ is $HL(\tau_1\ul A)$. We will do this by induction on the length of the regular chain $\ul A$. If $n=1$, then examples \ref{example_low_dimensions}(ii) and \ref{example_low_dimensions2}(ii) imply that $\ul A=(A,\frak m,0)$ for some discrete valuation ring $A$ with maximal ideal $\frak m$, and that $HL(\ul A)=\Frac\hat A$; this has residue field $A/\frak m=\tau_1\ul A$, as required.

Now suppose $n>1$. Then $HL(\ul A)=HL(\op{loc}\op{comp}\ul A)$, where $\op{loc}\op{comp}\ul A$ is a regular $(n-1)$-chain. So the inductive hypothesis tells us that $HL(\ul A)$ has first residue field \[HL(\tau_1\op{loc}\op{comp}\ul A)\stackrel{(1)}{=}HL(\op{loc}\op{comp}\tau_1\ul A)\stackrel{(2)}{=}HL(\tau_1\ul A),\] where (1) follows from the commutativity of truncation with localisation and completion, and (2) follows from the iterative definition of $HL$. This completes the proof.
\end{proof}

We finish this section by showing that the construction $\ul A\mapsto HL(\ul A)$ is functorial. Define a morphism $f:\ul A=(A,\frak p_n,\dots,\frak p_0)\To\ul B=(B,\frak q_n,\dots,\frak q_0)$ of regular $n$-chains to be a ring homomorphism $f:A\to B$ with the property that $f^{-1}(\frak p_i)=\frak q_i$ for $i=0,\dots,n$. In particular, $f$ is a local, injective homomorphism. It is easy to see that this makes the collection of regular $n$-chains into a category; our theorem is:

\begin{theorem}\label{theorem_n_continuity_of_functor}
$HL$ defines a functor from the category of regular $n$-chains to the category of fields of cdvdim $\ge n$ with $n$-continuous embeddings as morphisms.
\end{theorem}
\begin{proof}
The proof is straightforward but requires some notation for the sake of clarity: let $\cal C_n$ denote the category of regular $n$-chains, and let $\cal C_n^q$ denote its subcategory consisting of objects $\ul A$ for which $\op{cdvdim}k(\ul A)\ge q$ and of morphisms $f:\ul A\to\ul B$ for which the induced homomorphism $k(\ul A)\to k(\ul B)$ is a $q$-continuous embedding.

The category $\cal C_0^n$ may be identified with the category of fields of cdvdim $\ge n$ with $n$-continuous embeddings as morphisms, while the category $\cal C_n^0$ is the category of regular $n$-chains. So, by induction, it is enough to show that the process `loc\,comp' defines a functor $\cal C_n^q\to\cal C_{n-1}^{q+1}$ for any $n,q$.

Well, if $\ul A\in\cal C_n^q$, then lemma \ref{lemma_comp_and_loc_preserve_reg_chains} tell us that $k(\op{loc}\op{comp}\ul A)=\Frac(\hat A/\frak p_{n-1}\hat A)$, which is a complete discrete valuation field whose residue field is $k(\ul A)$, which has $\cdvdim\ge q$ by assumption; therefore $\cdvdim k(\op{loc}\op{comp}\ul A)\ge q+1$ and so $\op{loc}\op{comp}\ul A\in\cal C_{n-1}^{q+1}$. Secondly, if $f:\ul A\to \ul B$ is a morphism in $\cal C_n^q$, then the induced homomorphism $\Frac(\hat A/\frak p_{n-1}\hat A)\to\Frac(\hat B/\frak q_{n-1}\hat B)$ is continuous, since $f^{-1}(\frak q_n)=\frak p_n$, and moreover the induced embedding $k(\ul A)\to k(\ul B)$ is $q$-continuous by assumption; therefore $\Frac(\hat A/\frak p_{n-1}\hat A)\to\Frac(\hat B/\frak q_{n-1}\hat B)$ is $q+1$-continuous and so $\op{loc}\op{comp}f$ is a morphism in $\cal C_{n-1}^{q+1}$.
\end{proof}

\section{Constructing higher dimensional fields: the general case}\label{section_singular}
Unfortunately, the construction given in the previous section is geometrically rather restrictive because it required all the irreducible closed subschemes occurring in the flag to be regular at a fixed point (or, equivalently, all the chosen prime ideals had to be regular in the local ring). In practice one cannot expect this to be true, so in this section we will drop the regularity assumptions, and check in general that the same localisation-completion process continues to produce higher dimensional fields.

The first part of the section summarises the main results, while the second part contains proofs. The reader who is encountering this material for the first time (who hopefully has read the previous section), should briefly read the first part, particularly note \ref{remark_geometric_interpretation_2}, and then proceed directly to section \ref{section_adeles}.

This material requires a stronger background in commutative algebra than the previous sections.

\subsection{Statement of the results}\label{subsection_main_theorems}
In the regular case the essential use of regularity was in lemma \ref{lemma_regularity_completion_preserves_primality}, where in particular we saw that pushing forward a regular prime to the completion gave us another prime ideal.  It is precisely this which fails in the singular case. However, with some assumptions on the ring we can find a natural balance of properties to recover a similar construction which is strong enough for geometric applications.

Here we explain what replaces our old notion of a regular chain and state the main results; terms which may not be familiar to the reader will be discussed section \ref{subsection_definitions_and_proofs}, where the main results are proved.

\begin{definition}
Let us describe an ideal $I$ of a Noetherian ring $A$ as being {\em equiheighted} if and only if all minimal primes over $I$ have the same height in $A$ (the term `unmixed' is preferred in \cite{Matsumura1989}). Equivalently, $I$ is equiheighted if and only if all maximal ideals of $A_I$ ($:=S^{-1}A$, where $S=\{a\in A\,:\,a\mod I\mbox{ is not a zero divisor in }A/I\}$) have the same height.

The best way to imagine an equiheighted ideal is geometrically; let $V(I)$ be the Zariski closed set defined by $I$ in $X=\Spec A$. Then $I$ is equiheighted if and only if all the irreducible components of $V(I)$ have the same codimension in $X$.
\end{definition}

The data which replaces the regular chains of length $n$ of the previous section are {\em reduced chains of length $n$} (or {\em reduced $n$-chain}): \[\underline{A}=(A,I_n,\dots,I_0),\] where $A$ is an excellent, $n$-dimensional, reduced, semi-local ring and $(I_i)$ is a chain of radical equiheighted ideals \[I_n\supset\dots\supset I_0\] with $\height I_i=i$ and with $I_n$ equal to the Jacobson radical of $A$.

\begin{remark}\label{remark_standard_reduced_chain}
The easiest and most important way of constructing a reduced chain of length $n$ is to start with an excellent, $n$-dimensional, local domain $A$ and a complete flag of primes \[A\supset\frak p_n\supset\frak p_{n-1}\dots\supset\frak p_0.\] Then $(A,\frak p_n,\dots,\frak p_0)$ is a reduced chain of length $n$.
\end{remark}

Exactly as in the regular case, we can manipulate the data by completing or localizing: \[\op{comp}\ul A=(\hat{A},I_n\hat A,\dots,I_0\hat A),\] \[\op{loc}\ul A=(A_{I_{n-1}},I_{n-1}A_{I_{n-1}},\dots,I_0A_{I_{n-1}}).\] The analogue of lemma \ref{lemma_comp_and_loc_preserve_reg_chains} is that these processes preserve our new conditions:

\begin{lemma}\label{lemma_comp_and_loc_preserve_red_chains}
If $\ul{A}$ is a reduced chain of length $n$ then $\op{comp}\ul A$ is a reduced chain of length $n$ and $\op{loc}\ul A$ is a reduced chain of length $n-1$.
\end{lemma}
\begin{proof}
The proof is on page \pageref{proof_of_main_lemma}.
\end{proof}

\begin{remark}
The completion of the type of reduced chain described in the previous remark will typically not again be of that form; this is what forces us to work with general reduced chains.
\end{remark}

Once the lemma has been established we may, exactly as in the regular case, iterate the localization and completion processes: \[\op{HL}(\ul A):=\underbrace{\op{loc}\op{comp}\cdots \op{loc}\op{comp}}_{\sub{`loc\,comp' $n$ times}}\ul A\] This is a reduced chain of length $0$, i.e.~a reduced Artinian ring, i.e.~a finite product of fields.

\begin{theorem}\label{theorem_main_theorem_in_reduced_case}
If $\ul{A}$ is a reduced chain of length $n$ then $\op{HL}(\ul A)$ is a finite product of fields, each of $\cdvdim\ge n$ (and we may describe the residue fields).
\end{theorem}
\begin{proof}
See theorem \ref{theorem_proof_of_main_theorem} below.
\end{proof}

\begin{corollary}\label{corollory_essentially_of_finite_type_case_2}
If $\ul{A}$ is a reduced chain of length $n$, and $A$ is essentially of finite type over $\bb{Z}$, then $\op{HL}(\ul A)$ is a finite product of $n$-dimensional local fields.
\end{corollary}
\begin{proof}
Immediate from theorem \ref{theorem_proof_of_main_theorem}.
\end{proof}

\begin{remark}[Geometric interpretation]\label{remark_geometric_interpretation_2}
Let $X$ be an excellent, reduced scheme, and \[\xi=(y_n\subset y_{n-1}\subset \cdots\subset y_0)\] a complete flag of irreducible closed subschemes; so $\codim y_i=i$ and $y_0$ is an irreducible component of $X$. Write $y_n=\{z\}$.

Put $A=\roi_{X,z}$ and let $\frak{p}_i\subset A$ be the local equation for $y_i$ at $z$. Then \[\ul{A}=(A,\frak p_n,\dots,\frak p_0)\] is a reduced $n$-chain of the type which appeared in remark \ref{remark_standard_reduced_chain}. Thus $HL(\ul A)$ is a finite product, $\prod_i F_i$, of fields of $\cdvdim\ge n$, such that the $n^\sub{th}$ residue field of each $F_i$ is a finite field extension of $k(z)$. If the function field of $X$ is denoted $F$ then $HL(\ul A)$ is often denoted $F_\xi$.

The top degree ring of ad\`eles associated to $X$ is a certain restricted product \[\bb A_X=\rprod_\xi F_\xi,\] where $\xi$ varies over all complete flags of irreducible closed subschemes; see the end of section \ref{subsection_restricted_products} for details.
\end{remark}

\subsection{Definitions and proofs}\label{subsection_definitions_and_proofs}
In this section we set up the required machinery and prove lemma \ref{lemma_comp_and_loc_preserve_red_chains} and theorem \ref{theorem_main_theorem_in_reduced_case} from above. We begin with several remarks on commutative algebra, all of which may be found in standard textbooks.

\begin{remark}
{\em Minimal primes, radical ideals, and semi-local rings.} Let $A$ be a Noetherian ring. Then $A$ has only finitely many minimal primes $\frak{p}$ and the common intersection $\bigcap_{\frak{p}}\frak{p}$ is the nilradical of $A$. More generally, if $I$ is an ideal of $A$ then there are only finitely many primes of $A$ which are minimal over $I$ and their intersection is the radical of $I$.

If $I$ is a radical ideal of $A$ then the localisation of $A$ away from $I$ is defined to be $A_I=S^{-1}A$ where $S=\{a\in A : a\mbox{ is not a zero divisor in }A/I\}$. If $\frak{p}_1,\dots,\frak{p}_n$ are the primes minimal over $I$, so that $I=\bigcap_{i=1}^n\frak{p}_i$, then $S=A\setminus\bigcup_i\frak p_i$.

If $A$ is reduced, so that the zero ideal is radical, then $\Frac A=A_{\{0\}}$ is the total quotient ring of $A$, and we say that $A$ is normal if and only if it is integrally closed in $\Frac A$; in any case, $\tilde A$ denotes the integral closure of $A$ in $\Frac A$. Still with $A$ reduced, if $\frak p_1,\dots,\frak p_s$ are its (distinct) minimal prime ideals, then $\Frac A=\prod_i\Frac(A/\frak p_i)$ and $\tilde A=\prod_i\tilde{A/\frak p_i}$.

If $A$ is a Noetherian, semi-local ring then the intersection of its finitely many maximal ideals is called the Jacobson radical. We denote by $\hat{A}$ the completion of $A$ at its Jacobson radical. If $\frak{m}_1,\dots,\frak{m}_n$ are the distinct maximal ideals of $A$ then the diagonal map $A\to\prod_i A_{\frak{m}_i}$ induces an isomorphism $\hat{A}\cong\prod_i\hat{A_{\frak{m}_i}}$ where $\hat{A_{\frak{m}_i}}$ is the completion of the local ring $A_{\frak{m}_i}$ at its maximal ideal.
\end{remark}

\begin{remark}
{\em Zero-dimensional rings.} Suppose that $L$ is a Noetherian, zero-dimensional, reduced ring. Being Noetherian, $L$ has only finitely many minimal prime ideals $\frak{p}$; but since $L$ is zero-dimensional, these primes are also maximal ideals and the Chinese remainder theorem implies $L/\bigcap_{\frak{p}}\frak{p}\cong\prod_{\frak{p}}L/\frak{p}$; finally, $L$ reduced implies $\bigcap_{\frak{p}}\frak{p}=0$. Therefore $L$ is isomorphic to a finite product of fields and these fields are uniquely determined by $L$. Conversely, any finite product of fields is a Noetherian, zero-dimensional, reduced ring.
\end{remark}

\begin{remark}
{\em Heights of prime ideals.} If $A$ is a Noetherian ring then the height of a prime ideal $\frak{p}$ is the largest integer $n$ such that there is chain of prime ideals \[\frak{p}=\frak{p}_n\supset\dots\supset\frak{p}_0.\] We write $\height{\frak{p}}=n$; in other words, $\height{\frak{p}}=\dim A_{\frak{p}}$. 

More generally, if $\frak{q}$ is a prime contained inside $\frak{p}$, then $\height(\frak{p}/\frak{q})$ is the largest integer $n$ for which there exists a chain of primes \[\frak{p}=\frak{p}_n\supset\dots\supset\frak{p}_0=\frak{q}.\] In other words, it is the height of $\frak p/\frak q$ in the ring $A/\frak q$. If $I$ is an aribitrary ideal, then $\height{I}$ is defined to be the infimum of the heights of the prime ideals containing it.

{\em Universally catenary rings.} A Noetherian ring is said to be catenary if and only if for each triple of prime ideals $\frak{m}\supseteq\frak{p}\supseteq\frak{q}$ there is an equality of heights \[\height(\frak{m}/\frak{q})= \height(\frak{m}/\frak{p})+\height(\frak{p}/\frak{q}).\] The ring is said to be universally catenary if and only if every finitely generated algebra over the ring is catenary.

If $A$ is a Noetherian, universally catenary domain and $B$ is a ring extension of $A$, finitely generated as an $A$-module and also a domain, then \[\mbox{ht}_A(\frak{p}\cap A)=\mbox{ht}_B(\frak{p})\] for any prime $\frak{p}\subset B$ \cite[Cor.~8.2.6]{Liu2002}.
\end{remark}

\begin{remark}\label{remark_excellence}
{\em Excellent rings.} A Noetherian ring is said to be excellent if it satisfies some rather technical conditions \cite[\S32]{Matsumura1989}. What is more important for us is which operations preserve excellence and various properties which result from excellence. Firstly, any excellent ring is universally catenary. Secondly, if $A$ is an excellent ring, then so are
\begin{enumerate}
\item any finitely-generated $A$-algebra,
\item any quotient of $A$,
\item and any localisation of $A$ with respect to a multiplicative subset of $A$.
\end{enumerate}
Moreover, Dedekind domains of characteristic zero, Noetherian complete local rings, and fields are all excellent. Combining these properties we see that any ring which is finitely generated over a field or finitely generated over $\bb{Z}$ is excellent. Excellence is a local condition, which is to say that if $\Spec A$ has an open affine cover by the spectra of excellent rings then $A$ is also excellent; in particular, the product of finitely many excellent rings is again excellent. An excellent, local ring is reduced if and only if its completion is.

Let $A$ be an excellent, reduced local ring with maximal ideal $\frak{m}$ and let $\widetilde{A}$ be the integral closure of $A$ inside $\Frac{A}$. Then $\tilde{A}$ is finitely generated as an $A$-module and \[\hat{\tilde A}=\tilde{\hat A};\] in particular, $A$ is normal if and only if $\widehat{A}$ is normal.
\end{remark}

The most important property of excellence concerns the relationship between normalisation and completion. One can study `local branches' of a scheme either through preimages via the normalisation map or through infintesimal behaviour at the point. Excellence ensures that the two notions are the same. The following lemma explain the first part of the correspondence; this, and the subsequent results, are required solely to prove lemmas \ref{lemma_compltion_works_for_reduced_chains} and \ref{lemma_compltion_works_for_reduced_chains}.

\begin{lemma}
Let $A$ be an excellent, reduced local ring. Let $M$ be a maximal ideal of $\tilde{A}$ and let $f$ be the natural map \[\widehat{A}\stackrel{f}{\longrightarrow}\widehat{(\widetilde{A})_M}.\] Then $f$ is a finite morphism and its kernel is a minimal prime ideal $\frak{p}$ of $\hat{A}$. Moreover,  $\frak{p}\cap A$ is a minimal prime of $A$ and $\dim \hat{A}/\frak{p}=\dim A/\frak{p}\cap A$.
\end{lemma}
\begin{proof}
We first prove that the kernel of $f$ is prime. Let $P_1,\dots,P_s$ be the distinct minimal primes of $A$; the natural map \[\widetilde{A}\to\prod_i\widetilde{A/P_i}\] is an isomorphism. So, under this identification, we may write \[M=M_0\times\prod_{i\neq i_0}\widetilde{A/P_i}\] for some $i_0$ in the range $1,\dots,s$ and maximal ideal $M_0$ of $\widetilde{A/P_{i_0}}$. Localising obtains an isomorphism \[\tilde{A}_M\cong(\widetilde{A/P_{i_0}})_{M_0},\] which is an excellent, normal local domain; therefore its completion \[\widehat{\tilde{A}_M}\cong\widehat{(\widetilde{A/P_{i_0}})_M}\] is also a normal domain, by excellence. Therefore the kernel of the natural map \[f:\widehat{A}\stackrel{f}{\longrightarrow}\widehat{(\widetilde{A})_M}\] is a prime $\frak{p}$. Note that $P_{i_0}=A\cap\frak{p}$.

We now prove that $f$ is a finite morphism. The excellence of $A$ implies that $\widetilde{A}$ is finitely generated as an $A$-module and therefore is a semi-local ring; let $M_1,\dots,M_r$ be its distinct maximal ideals, with $M=M_1$ say. Then \[\widetilde{A}\otimes_A\hat{A}\cong\prod_i\widehat{(\widetilde{A})_{M_i}},\] which is a finite $\hat{A}$-module; but the homomorphism \[\hat{A}\to\widetilde{A}\otimes_A\hat{A}\cong\prod_i\widehat{(\widetilde{A})_{M_i}}\to\widehat{(\widetilde{A})_{M_1}}\] is exactly $f$ and therefore $f$ is a finite morphism.

Finally we consider dimensions; write $P=P_{i_0}$ to ease notation. Since $f$ induces a finite injection between Noetherian domains \[\hat{A}/\frak{p}\to\widehat{(\widetilde{A})_M},\] these rings have equal dimension; therefore \[\dim \hat{A}/\frak{p}=\dim \widetilde{A}_M=\dim(\widetilde{A/P})_{M_0}.\] Apply the remark on universally catenary domains to the morphism $A/P\to\tilde{A/P}$ (which is finite by excellence) to conclude that $\height_{\tilde{A/P}}(M_0)=\height_{A/P}(\frak{m}/P)$, where $\frak{m}$ is the maximal ideal of $A$; but this is exactly $\dim A/P$. This is enough to prove that $\frak{p}$ is minimal, for if $\frak{q}\subset\frak{p}$ were a smaller prime in $\hat{A}$ then, because $P$ is minimal, one still has $P=A\cap\frak{q}$; but then $P\hat{A}\subseteq\frak{q}$ and so \[\dim A/P=\dim \hat{A}/P\hat{A}\ge\dim\hat{A}/\frak{q}>\dim\hat{A}/\frak{p},\] a contradiction.
\end{proof}

\begin{remark}\label{remark_excellence_correspondence}
Let $A$ be an excellent, reduced local ring. The previous result defines an association 
\begin{align*}
\{\mbox{maximal primes of $\widetilde{A}$}\}&\longrightarrow\{\mbox{minimal primes of $\hat{A}$}\}\\
M&\mapsto\frak{p}=\mbox{ker}\langle{\hat{A}\to\widehat{(\tilde{A})_M})}\rangle
\end{align*}
This is a bijective correspondence, which is what we meant by saying that normalisation branches and infintesimal branches were the same; a proof may be found in \cite[Thm.~6.5]{Dieudonne1967}.

A consequence of this and the previous lemma which we will use repeatedly is that if $\frak{p}$ is a minimal prime of $\hat{A}$ then $\frak{p}\cap A$ is a minimal prime of $A$ and $\dim A/\frak{p}\cap A=\dim\hat{A}/\frak{p}$.
\end{remark}

\begin{lemma}[Going down lemma]
Let $A$ be an excellent, local ring. Let $\frak{q}\subset\frak{p}$ be two primes in $A$ and let $P\subset\hat{A}$ be a prime sitting over $\frak{p}$. Then there is a prime $Q\subset\hat{A}$ sitting over $\frak{q}$ satisfying $Q\subset P$.
\end{lemma}
\begin{proof}
We are free to replace the prime ideal $P$ by one which is minimal over $\frak{p}\hat{A}$ and we do so. Then let $Q$ be a prime inside $P$ which is minimal for containing $\frak{q}\hat{A}$. The non-trivial part is proving that $Q\neq P$ and for this we use excellence.

Since $A/\frak{p}$ and $A/\frak{q}$ are excellent the previous remark implies that $\dim{\hat A/P}=\dim{A/P\cap A}$ and $\dim{\hat A/Q}=\dim{A/Q\cap A}$; but it also implies that $P\cap A$ is a minimal prime over, hence equals, $\frak{p}$. Similarly $Q\cap A=\frak{q}$; since $\dim A/\frak{p}<\dim A/\frak{q}$ the proof is complete.
\end{proof}

Recall from subsection \ref{subsection_main_theorems} the notion of an equiheighted ideal $I$ of a Noetherian ring $A$: it means that all minimal primes over $I$ have the same height in $A$.

\begin{lemma}\label{lemma_compltion_works_for_reduced_chains}
Let $A$ be an excellent, local ring and suppose $I$ is a radical ideal of $A$. Then $I\hat{A}$ is a radical ideal, $\height I\hat{A}=\height I$, and if $I$ is equiheighted so is $I\hat{A}$.
\end{lemma}
\begin{proof}
Since $A/I$ is an excellent, reduced local ring, its completion is also reduced, i.e.~$I\hat{A}$ is a radical ideal.

Let $P$ be a prime containing $I\hat{A}$; then $\frak{p}=P\cap A$ is a prime containing $I$ and so there is a chain $\frak{p}=\frak{p}_s\supset\dots\supset\frak{p}_0$ witnessing $\height\frak{p}\ge s$. By the previous going down lemma we obtain a chain of primes of $\hat{A}$ \[P=P_s\supset\dots\supset P_0\] where $P_i\cap A=\frak{p}_i$ for each $i$. Therefore $\height{P}\ge s$. Since this is true for every prime $P$ containing $I\hat{A}$ we obtain $\height I\ge s$.

Furthermore, by definition of the height of $I$, there is a prime $\frak{q}$, minimal for containing $I$, which satisfies $\height\frak{q}=\height I$. Let $Q$ be a prime of $\hat{A}$, minimal for containing $\frak{q}\hat{A}$; then the correspondence of remark \ref{remark_excellence_correspondence} implies that $Q\cap A$ is minimal over, hence equal to, $\frak{q}$. There is a minimal prime $\frak{l}$ of $\hat{A}$ which satisfies $\height Q=\height Q/\frak{l}$. Let $\frak{m}$ be the maximal ideal of $A$. We manipulate heights and dimensions as follows
\begin{align*}
\height_{\hat{A}} Q&=\height_{\hat{A}} Q/\frak{l}\\
	&=\height_{\hat{A}} \frak{m}\hat{A}/\frak{l}-\height_{\hat{A}}\frak{m}\hat{A}/Q\tag{1}\\
	&=\dim \hat{A}/\frak{l}-\dim\hat{A}/Q\\
	&=\dim A/\frak{l}\cap A-\dim A/Q\cap A\tag{2}\\
	&=\height_A \frak{m}/\frak{q}\cap A-\height_A\frak{m}/Q\cap A\\
	&=\height_A \frak{q}/\frak{l}\cap A\tag{3}\\
	&\le\height_A \frak{q}\\
	&=\height_A I,
\end{align*}
where (1) (resp. (3)) follows from $\hat{A}$ (resp. A) being excellent and hence universally catenary and (2) follows from applying the correspondence of remark \ref{remark_excellence_correspondence} to $A$ and $A/I$. The definition of the height of $I\hat{A}$ implies now that $\height I\hat{A}\le\height{I}$; in conjunction with the earlier inequality we have therefore proved equality.

More generally, if $Q$ is any prime of $\hat{A}$ which is minimal over $I\hat{A}$ then $Q\cap A$ is minimal over $I$ and the calculation above obtains $\height Q\le\height Q\cap A$. So if $I$ is equiheighted then $\height Q\le\height I$. But $Q\supseteq I\hat{A}$ implies $\height Q\ge\height I\hat{A}$ and so $\height Q=\height I\hat{A}$. So $I$ equiheighted implies $I\hat{A}$ equiheighted.
\end{proof}

We offer the following improvement on the previous result, which proves most of lemma \ref{lemma_comp_and_loc_preserve_red_chains}:

\begin{lemma}\label{lemma_compltion_works_for_reduced_chains_2}
Let $A$ be an excellent, reduced semi-local ring and suppose that $I$ is a radical ideal of $A$ contained in the Jacobson radical. Then $I\hat{A}$ is a radical ideal, $\height I\hat{A}=\height I$ and if $I$ is equiheighted so is $I\hat{A}$.
\end{lemma}
\begin{proof}
Let $\frak{m}_1,\dots,\frak{m}_s$ be the distinct maximal ideals of $A$ and write $A_i=A_{\frak{m}_i}$ for each $i$. The isomorphism $\hat{A}\cong\prod_i \hat{A}_i$ restricts to $I\hat{A}$ as $I\hat{A}\cong\prod_i I\hat{A}_i$. The desired results all follow from this identification and the previous result; let us explain how in greater detail. 

Firstly, since $IA_i$ is a radical ideal of $A$ and $A_i$ is an excellent, reduced local ring the previous result implies that $I\hat{A}_i$ is radical for each $i$. Clearly this implies that $I\hat{A}$ is radical.

Secondly, the primes of $\hat{A}$ which are minimal over $I\hat{A}$ have the form $P=\frak{p}A_{i_0}\times\prod_{i\neq i_0} A_i$ where $\frak{p}$ is a prime of $A$ which is minimal over $I$ and contained inside $\frak{m}_{i_0}$. Moreover, \[\height_{\hat{A}}P\stackrel{(1)}{=}\height_{A_{i_0}} \frak{p}A_{i_0}\stackrel{(2)}{=}\height_A\frak{p},\] where (1) follows from the obvious structure of primes in $\prod_iA_i$ and (2) follows from basic properties of localisation. This is enough to complete the proof.
\end{proof}

Now we may finish the proof of lemma \ref{lemma_comp_and_loc_preserve_red_chains}:

\begin{proof}[Proof of lemma \ref{lemma_comp_and_loc_preserve_red_chains}]\label{proof_of_main_lemma}
After using our remarks on excellence the only remaining difficulty with the completion process is knowing that $I\hat{A}$ is equiheighted of the correct height; but this is exactly covered by our previous lemma.

For localisation, let $\frak{p}_1,\dots,\frak{p}_s$ be the distinct primes of $A$ which are minimal over $I_{n-1}$. Then the prime ideals of $A_{I_{n-1}}$ correspond to the prime ideals $\frak q$ of $A$ which are disjoint from $A\setminus\bigcup_i\frak p_i$; by prime avoidance, this means $\frak q\subseteq\frak p_i$ for some $i$. Therefore $A_{I_{n-1}}$ is semi-local of dimension $\height I_{n-1}=n-1$.

Let $j$ be in the range $0,\dots,n-1$ and write $I=I_j$ for simplicity. We must show that $IA_{I_{n-1}}$ is a radical equiheighted ideal of height $j$. Firstly, $A_{I_{n-1}}/IA_{I_{n-1}}$ is a localisation of the reduced ring $A/I$, hence is reduced. Secondly, the prime ideals of $A_{I_{n-1}}$ which are minimal over $IA_{I_{n-1}}$ correspond to those prime ideals of $A$ which are minimal over $I$ and which are contained inside one of $\frak p_1,\dots,\frak p_s$. It follows that $IA_{I_{n-1}}$ is equiheighted of the same height as $I$.
\end{proof}

It is convenient to describe the quotient of a semi-local ring $A$ by its Jacobson radical as its {\em residue ring} $k(A)$, which is a finite product of fields. If $\ul A$ is a reduced chain, write $k(\ul A)=k(A)$.

\begin{remark}\label{remark_change_of_residue_field_of_reduced_chain}
As in lemma \ref{lemma_comp_and_loc_preserve_reg_chains} we note how the completion and localisation processes effect the residue ring of a reduced chain: \[k(\op{comp}\ul A)=k(\ul A)\] \[k(\op{loc} A)=\Frac(A/I_{n-1})\]
\end{remark}

\begin{lemma}
Let $B$ be a Noetherian, one-dimensional, complete, reduced semi-local ring. Then $\Frac{B}$ is a finite product of complete discrete valuation fields, $\prod_i F_i$, such that the product of residue fields, namely $\prod_i\res{F_i}$, is a finite extension of $k(B)$.
\end{lemma}
\begin{proof}
First suppose that $B$ is actually local, and let $\frak{p}$ be a minimal prime. Since $B/\frak{p}$ is a complete local domain the correspondence for excellent rings implies that its normalisation $\tilde{B/\frak{p}}$ inside $\Frac{B/\frak{p}}$ has a unique maximal ideal. So $\tilde{B/\frak{p}}$ is a Noetherian, one-dimensional, normal, local ring, i.e.~a discrete valuation ring; being finite over the complete ring $B/\frak{p}$, it is itself complete. We have proved that $\tilde{B/\frak{p}}$ is a complete discrete valuation ring, and therefore its field of fractions $\Frac B/\frak{p}$ is a complete discrete valuation field.

Furthermore, $B_{\frak{p}}$ is a Noetherian, zero-dimensional, reduced local ring, i.e.~a field; the maximal ideal of this field is $0=\frak{p}B_{\frak{p}}$ and therefore the natural map $B\to B_{\frak{p}}$ induces an isomorphism $\Frac B/\frak{p}\cong B_{\frak{p}}$. So $B_{\frak{p}}$ is a complete discrete valuation ring whose residue field is that of $\tilde{B/\frak{p}}$; this is a finite field extension of $k(B)$.

We complete the proof in the local case by recalling that $\Frac{B}\cong\prod_{\frak{p}}B_{\frak{p}}$, where $\frak{p}$ runs over the finitely many minimal primes of $B$.

Now consider the semi-local case. Since $B$ is assumed complete, therefore $B\cong\prod_{\frak{m}}\hat{B}_{\frak{m}}$, where $\frak{m}$ runs over the finitely many maximal ideals of $B$. So $\Frac B\cong\prod_{\frak{m}}\Frac\hat{B}_{\frak{m}}$, easily reducing the claim to the local case.
\end{proof}

We may now prove theorem \ref{theorem_main_theorem_in_reduced_case} and identify the residue fields:

\begin{theorem}\label{theorem_proof_of_main_theorem}
If $\ul{A}$ is a reduced chain of length $n$ then $\op{HL}(\ul A)$ is a finite product of fields, $\prod_i F_i$, each of $\cdvdim\ge n$, and the product of the $n^\sub{th}$ residue fields, namely $\prod_iF_i^{(n)}$, is a finite extension of $k(A)$.
\end{theorem}
\begin{proof}
By induction on $n$, with nothing to prove if $n=0$. Applying a single completion and localisation to the reduced chain $\underline{A}$ obtains \[\op{loc}\op{comp}\underline{A}=(\hat{A}_{I_{n-1}},I_{n-1}\hat{A}_{I_{n-1}},\dots,I_0\hat{A}_{I_{n-1}}).\] By the inductive hypothesis, $\op{HL}(\underline{A})=\op{HL}(\op{loc}\op{comp}\underline{A})$ is isomorphic to a finite product $\prod_i F_i$ where each $F_i$ is a field of $\cdvdim\ge n-1$, and $\prod_iF_i^{(n-1)}$ is a finite extension of $k(\op{loc}\op{comp}\underline{A})=\Frac \hat{A/I_{n-1}}$ (using remark \ref{remark_change_of_residue_field_of_reduced_chain}).

Moreover, $\hat{A/I_{n-1}}$ is a Noetherian, one-dimensional, complete, reduced semi-local ring and so the previous lemma implies that $\Frac \hat{A/I_{n-1}}$ is a finite product of complete discrete valuation fields, $\prod_j K_j$, such that $\prod_j\res{K_j}$ is a finite ring extension of $k(A)$. It is easy to see that any reduced finite ring extension of $\prod_j K_j$ is therefore also a finite product of complete discrete valuation fields whose product of residue fields is a finite extension of $k(A)$, and this competes the proof.
\end{proof}

\section{An introduction to higher dimensional ad\`eles}\label{section_adeles}
The ring of ad\`eles associated to a number field or to a smooth curve over a field are a familiar object in algebraic and arithmetic geometry, comprised of a restricted product of local fields. As well as defining two-dimensional local fields, A.~Parshin defined two-dimensional ad\`eles for smooth algebraic surfaces, this time formed by taking a restricted product of two-dimensional local fields associated to the surface. The constructions were subsequently extended to arbitrary dimensions by A.~Beilinson.

In this section we will first describe Beilinson's general construction. The input is a scheme $X$, a quasi-coherent sheaf $M$, and a set $T$ of flags on $X$; the output, constructed in a recursive fashion via rules (A1) -- (A4) below, is an abelian group $\bb A(T,M)$ which the first lemma of section \ref{subsection_restricted_products} shows can be interpreted as a restricted product $\rprod_{\xi\in T}M_\xi^\comp$. The local factors $M_\xi^\comp$ are obtained from $M$ by successively completing and localising along the flag $\xi$, in a manner analogous to the constructions of sections \ref{section_regular} and \ref{section_singular}.

In the particular case that $M=\roi_X$ and $T$ is the set of all complete flags, we obtain $\rprod_\xi(\roi_X)_\xi^\comp$, a restricted product of higher dimensional fields, as constructed in remarks \ref{remark_geometric_interpretation} and \ref{remark_geometric_interpretation_2} (where we wrote $F_\xi$ instead of $(\roi_X)_\xi^\comp$). This is the higher dimensional analogue of the familiar ring of ad\`eles.

In subsection \ref{subsection_adeles_and_cohomology} we describe Beilinson's result that these ad\`ele groups may be put together to form a semi-cosimplicial flasque resolution $\ul{\bb A}(S_\bullet^\sub{red}(X), M)$ of $M$. Taking global sections thus yields a complex of ad\`eles group $\bb A(S_\bullet^\sub{red}(X), M)$ which computes $H^*(X,M)$.

We then analyse the construction in more detail when $\dim X=1$, showing in proposition \ref{proposition_one_dim_adeles} that we recover the familiar ring of ad\`eles if $M=\roi_X$. Finally we analyse the construction when $\dim X=2$ and explicitly describe all the ad\`ele groups as restricted products, thereby recovering the original definitions given by Parshin.

\subsection{The definition of the higher dimensional ad\`eles}
Let $X$ be a Noetherian scheme. We let $Coh(X)$ (resp.~$QCoh(X)$) denote the category of coherent (resp.~quasi-coherent) $\roi_X$-modules. A {\em chain} on $X$ is a finite sequence of points \[\xi=(x_0,\dots,x_m)\] such that $x_i\in\res{\{x_{i-1}\}}$ for $i=1,\dots,m$; let $S_m(X)$ denote the set of chains of length $m$. The chain is said to be {\em reduced} if and only if it has no repetitions, i.e.~$x_0\neq x_1\neq\cdots\neq x_m$; let $S_m^\sub{red}(X)$ denote the reduced chains of length $m$. Obviously $S_m^\sub{red}(X)=\emptyset$ if $m>\dim X$.

\begin{remark}
Up until now, we have preferred to work with the flag of irreducible closed subschemes \[\Cl{x_m}\subseteq\Cl{x_{m-1}}\subseteq\cdots\subseteq\Cl{x_0}\] rather than the so-called chain of its generic points. Passing between the two approaches is trivial, but the notation of chains of points is better suited to adelic constructions.
\end{remark}

For $x\in X$ and $T\subseteq S_m(X)$, we write \[_xT=\{\xi\in S_{m-1}(X):(x,\xi)\in T\}\subseteq S_{m-1}(X).\]

For any quasi-coherent sheaf $M$ on $X$, any $m\ge 0$, and any $T\subseteq S_m(X)$, the associated {\em ad\`ele group} $\bb{A}_X(T,M)=\bb{A}(T,M)$ is the abelian group defined according to the following recursive rules:
\begin{enumerate}
\item[(A1)] $\bb{A}(T,\cdot):QCoh(X)\to Ab$ commutes with direct limits.
\item[(A2)] If $M$ is coherent and $m=0$, then \[\bb A(T,M)=\prod_{x\in T}\hat{M_x},\] where $\hat{M_x}$ is the $\frak m_{X,x}$-adic completion of the finitely generated $\roi_{X,x}$-module $M_x$.
\item[(A3)] If $M$ is coherent and $m\ge1$, then \[\bb{A}(T,M)=\prod_{x\in X}\projlim_r \bb{A}(_xT,\tilde j_{rx}M),\] where the notation $j_{rx}M$ will be explained in the next remark.
\item[(A4)] $\bb{A}(\emptyset,M)=0$.
\end{enumerate}
The proof that these rules uniquely give a well-defined family of exact functors $\bb A(T,\cdot):QCoh(X)\to Ab$ is very nicely explained in \cite{Huber1991}, so we content ourselves with an example to demonstrate:

\begin{example}
Suppose that $T\subseteq S_1(X)$ and that $M\in Coh(X)$. Then rule (A3) implies that \[\bb A(T,M)=\prod_{x\in X}\projlim_r \bb{A}(_xT,\tilde j_{rx}M).\] Next, since $\tilde j_{rx}M$ is a quasi-coherent sheaf, it is equal to the inductive limit of its coherent subsheaves $N$, i.e.~$\tilde j_{rx}M=\indlim_{N\subseteq \tilde j_{rx}M}N$, and so rule (A1) implies that \[\bb{A}(_xT,\tilde j_{rx}M)=\indlim_{N\subseteq \tilde j_{rx}M}\bb A(_x T,N)\] But $_xT$ is a set of chains of length $0$, so rule (A2) implies that $\bb A(_x T,N)=\prod_{z\in _xT}\hat{N_z}$, where $\hat{N_z}$ is the $\frak m_{X,z}$-adic completion of the finitely generated $\roi_{X,z}$-module $N_z$.

In conclusion, \[\bb A(T,M)=\prod_{x\in X}\projlim_r \indlim_{N\subseteq \tilde j_{rx}M}\prod_{z\in _xT}\hat{N_z},\] where each inductive limit is taken over the coherent submodules $N$ of $\tilde j_{rx}M$.

In other words, by iterating rules (A1) and (A3), one reduces the calculation of any ad\`ele group to a set of length $0$ chains and a coherent sheaf, at which point rule (A2) is applied.
\end{example}

\begin{remark}\label{remark_notation_j_rx}
We must explain the notation $\tilde j_{rx}M$:

Given a morphism of Noetherian schemes $f:W\to X$, there is the adjoint pair of functors $(f^*,f_*)$: \[\xymatrix@=1cm{QCoh(W)\ar@/^5mm/[r]^{f_*}&QCoh(X)\ar@/^5mm/[l]^{f^*}},\] and we will write \[\tilde f=f_*f^*:QCoh(X)\to QCoh(X).\]

If $x$ is a point of a scheme $X$, and $r\ge 0$, then we denote by $j_{rx}$ the natural morphism \[j_{rx}:\Spec\roi_{X,x}/\frak m_{X,x}^r\to X\] from the `fat point' around $x$ to $X$. (We sometimes write $j_x=j_{1x}:\Spec k(x)\to X$.)

Combining these two pieces of notation defines the functor $\tilde j_{rx}:QCoh(X)\to QCoh(X)$. However, more concretely, the reader can check that if $M$ is a quasi-coherent sheaf on $X$, then $\tilde j_{rx}M$ is the quasi-coherent sheaf on $X$ which is constantly equal to $M_x/\frak m_{X,x}^rM_x$ along $\Cl x$, and is zero elsewhere.
\end{remark}

\subsection{Ad\`ele groups as restricted products}\label{subsection_restricted_products}
The familiar ring of ad\`eles \[\rprod_p\bb Q_p=\big\{(f_p)\in\prod_p\bb Q_p:f_p\in\bb Z_p\mbox{ for all but finitely many }p\big\}\] is often called a `restricted product'. In this section we explain that each higher dimensional ad\`ele group $\bb A(T,M)$ can also be seen as a restricted product, though typically not defined simply by imposing a `for all but finitely many' condition. Indeed, the reader may wish to glance at proposition \ref{proposition_two_dim_adeles}, where the restriction conditions are explicitly described for a two-dimensional scheme.

We will abuse notation slightly by writing $\bb A(\xi,M)=\bb A(\{\xi\},M)$ if $\xi\in S_m(X)$ is a single chain on $X$ (very soon we will write $M_\xi^\comp$ instead); these will be the local factors appearing in the restricted products:

\begin{lemma}
Let $X$ be a Noetherian scheme and $M\in QCoh(X)$. Let $T\subseteq S_m(X)$. Then there is a natural injective homomorphism \[\bb A(T,M)\To\prod_{\xi\in T}\bb A(\xi,M),\] whose image contains $\bigoplus_{\xi\in T}\bb A(\xi,M)$.
\end{lemma}
\begin{proof}
The result is proved by induction on $m$, the length of the chains contained inside $T$. This is the standard method of proof for results about higher dimensional ad\`eles.

For any fixed $m$ it is enough to prove the claim when $M$ is coherent, for then we may pass to the limit using rule (A1). So henceforth $M$ is coherent.

Suppose first that $m=0$. Then rule (A2), applied to both $T$ and to all the singletons $\{x\}$, $x\in T$, implies \[\bb A(T,M)=\prod_{x\in T}\hat{M_x}=\prod_{x\in T}\bb A(x,\cal M).\] In this case we do not merely have a natural injective homomorphism, we have equality.

Now suppose $m>1$. Then rule (A3) implies $\bb{A}(T,M)=\prod_{x\in X}\projlim_n \bb{A}(_xT,\tilde j_{rx}M)$, and the inductive hypothesis applied to $_xT$, for each $x\in X$, implies that there are natural injective homomorphisms \[\bb{A}(_xT,\tilde j_{rx}M)\to\prod_{\xi\in _xT}\bb A(\xi,\tilde j_{rx}M)\] whose images contain $\bigoplus_{\xi\in _xT}\bb A(\xi,\tilde j_{rx}M)$. Taking $\projlim_r$ we obtain a natural injective homomorphism \[\projlim_r\bb{A}(_xT,\tilde j_{rx}M)\to\projlim_r\prod_{\xi\in _xT}\bb A(\xi,\tilde j_{rx}M)=\prod_{\xi\in _xT}\projlim_r\bb A(\xi,\tilde j_{rx}M)\] whose image contains $\projlim_r\bigoplus_{\xi\in _xT}\bb A(\xi,\tilde j_{rx}M)\supseteq \bigoplus_{\xi\in _xT}\projlim_r\bb A(\xi,\tilde j_{rx}M)$.

Noting that $\prod_{x\in X}\prod_{\xi\in _xT}=\prod_{(x,\xi)\in T}$, putting this together gives a natural injective homomorphism \[\bb A(T,M)\to\prod_{(x,\xi)\in T}\projlim_r\bb A(\xi,\tilde j_{rx}M)\] whose image contains $\bigoplus_{(x,\xi)\in T}\projlim_r\bb A(\xi,\tilde j_{rx}M)$. But rule (A3) implies that \[\projlim_r\bb A(\xi,\tilde j_{rx}M)=\bb A((x,\xi),M),\] completing the proof.
\end{proof}

We may now introduce the following pieces of suggestive notation:

\begin{definition}\label{definition_local_factors_in_adeles}
Let $X$ be a Noetherian scheme and $M\in QCoh(X)$. If $\xi\in S_m(X)$ is a length $m$ chain on $X$, define \[M_\xi^\comp:=\bb A(\xi,M).\] If $T\subseteq S_m(X)$, then introduce `restricted product notation' for the ad\`ele group: \[\rprod_{\xi\in T} M_\xi^\comp=\rprod_{\xi\in T}\bb A(\xi,M):=\bb A(T,M).\]
\end{definition}

\begin{remark}
It the literature it is common to write $M_\xi$ instead of $M_\xi^\comp$. This leads to an obvious conflict of notation: if we identify a length $0$ chain $(x)$ with the point $x$ (as we have already done in rule (A2)), then $M_x$ could either mean the stalk of $M$ at $x$ or it could mean $\bb A(x,M)$ ($=\hat{M_x}$  if $M$ is coherent). Therefore we prefer the notation $M_\xi^\comp$, which also reminds us that it is a group typically involving completions.
\end{remark}

According to the previous lemma, $\bb A(T,M)=\rprod_{\xi\in T}M_\xi^\comp$ is a large subgroup of $\prod_{\xi\in T}M_\xi^\comp$. We emphasise again that as soon as $\dim X>1$, there is no simple description of this subgroup, unlike the `all but finitely many' condition appearing in the familiar ad\`eles $\rprod_p\bb Q_p$. Therefore $\rprod_{\xi\in T}M_\xi^\comp$ is simply a piece of suggestive notation for $\bb A(T,M)$.

Next we describe the local factors $M_\xi^\comp$ appearing in the groups of ad\`eles; they are constructed from $M$ by a series of localisations and completions along the chain $\xi$, in a very similar way to the constructions of sections \ref{section_regular} and \ref{section_singular}. We first introduce some functors for clarity of exposition:

If $A$ is a Noetherian local ring, with maximal ideal $\frak m_A$, then $A\dash Mod$ denotes the category of all $A$-modules. Define the functor \[C_A:A\dash Mod\to A\dash Mod,\quad M\mapsto\indlim_{N\subseteq M}\hat N,\] where $N$ varies over all finitely-generated submodules of $M$, and $\hat N$  denotes the $\frak m_A$-adic completion of $N$; in other words, $C_A$ acts as completion on finitely-generated modules and commutes with direct limits (c.f.~the definition of the ad\`ele groups). Secondly, if $\frak p\subseteq A$ is a prime ideal, let \[S_\frak p^{-1}:A\dash Mod\to A_{\frak p}\dash Mod,\quad M\mapsto M\otimes_AA_{\frak p}\] denote localisation.

\begin{lemma}\label{lemma_local_factors}
Let $X$ be a Noetherian scheme, $M\in QCoh(X)$, and $\xi=(x_0,\dots,x_m)\in S_m(X)$. Let $A=\roi_{X,x_m}$, and let $\frak p_m\supseteq\frak p_{m-1}\supseteq\cdots\supseteq \frak p_0$ be the local equations at $x_m$ for $x_m,\dots,x_0$ respectively. Then \[M_\xi^\comp=C_{A_{\frak p_0}}S_{\frak p_0A_{\frak p_1}}^{-1}\cdots S_{\frak p_{m-2}^{-1}A_{\frak p_{m-1}}}C_{A_{\frak p_{m-1}}}S_{\frak p_{m-1}}^{-1}C_A(M_{x_m}).\]
\end{lemma}
\begin{proof}
The proof is a straightforward induction on $m$ using the original definition of the ad\`ele groups. See \cite[Prop.~3.2.1]{Huber1991} for details.
\end{proof}

\begin{example}
Suppose $X$ is integral with generic point $\eta$ and function field $F$; let $x\in X$ and put $\xi=(\eta,x)$. Then, for any $M\in Coh(X)$, \[M_\xi^\comp=\hat{M_x}\otimes_{\roi_{X,x}}F,\] where $\hat{M_x}$ is the $\frak m_{X,x}$-adic completion of $M_x$.

In particular, if $x$ is a codimension one point of $X$ for which $\roi_{X,x}$ is a discrete valuation ring, then $M_\xi^\comp=M_x\otimes_{\roi_{X,x}}F_x$, where $F_x=\Frac\hat\roi_{X,x}$ is the completion of $F$ at the discrete valuation corresponding to $x$.
\end{example}

In the particular case when $\xi$ is a complete chain, the local factor $M_\xi^\comp=\bb A(\xi,M)$ reduces to the construction of higher dimensional fields described in sections \ref{section_regular} and \ref{section_singular}:

\begin{proposition}
Let $X$ be an excellent, reduced, $n$-dimensional scheme, and $\xi=(x_0,\dots,x_n)$ a complete chain on $X$, i.e.~$\dim \roi_{X,x_i}=i$ for $i=0,\dots,n$. Let $A=\roi_{X,x_n}$, and let $\frak p_i\subset A$ be the local equation for $x_i$ at $x_n$, so that $\ul A=(A,\frak p_n,\dots,\frak p_0)$ is a reduced $n$-chain in the sense of section \ref{subsection_main_theorems} (or a regular $n$-chain in the sense of section \ref{section_regular} if each subscheme $\Cl{x_i}$ is regular at $x_n$). Then \[M_\xi^\comp=M_{x_n}\otimes_A HL(\ul A),\] where $HL(\ul A)$ denote the finite product of fields of cdvdim $\ge n$ constructed in sections \ref{section_regular} and \ref{section_singular}.
\end{proposition}
\begin{proof}
This is proved by modifying the previous proof to compare with the localisation-completion processes of sections \ref{section_regular} and \ref{section_singular}.
\end{proof}

In particular, we can take $M=\roi_X$ (where we write $\roi_{X,\xi}^\comp=(\roi_X)_\xi^\comp$):

\begin{corollary}
Let $X$ be a reduced, $n$-dimensional scheme, essentially of finite type over $\bb Z$, and let $\xi=(x_0,\dots,x_n)$ be a complete chain on $X$. Then $\roi_{X,\xi}^\comp=\bb A(\xi,\roi_X)$ is a finite product of $n$-dimensional local fields (which has been denoted $F_\xi$ so far in these notes, especially in remarks \ref{remark_geometric_interpretation} and \ref{remark_geometric_interpretation_2}).
\end{corollary}
\begin{proof}
This is the content of the previous proposition when $M=\roi_X$, since corollaries \ref{corollory_essentially_of_finite_type_case} and \ref{corollory_essentially_of_finite_type_case_2} imply that $HL(\ul A)$ is a finite product of $n$-dimensional local fields.
\end{proof}

\begin{remark}[Product structure]
Another standard result about ad\`ele groups, proved in the usual inductive fashion, is the following: Given $M,N\in QCoh(X)$ and $T\subseteq S_m(X)$, there is a natural homomorphism $\bb A(T,M)\otimes_{\bb Z}\bb A(T,N)\to\bb A(T,M\otimes_{\roi_X}N)$.

In particular, the ad\`ele group $\bb A(T,\roi_X)$ is naturally a commutative, unital ring.
\end{remark}

We can finally make make precise various comments we have made, such as in remark \ref{remark_geometric_interpretation_2}, that there are higher dimensional rings of ad\`eles which are restricted products of higher dimensional local fields.

Indeed, if $X$ is as in the previous corollary, then $S_n^\sub{red}(X)$ is a set of complete chains on $X$ (if $X$ is equi-dimensional, it is the set of all complete chains) and so the ring \[\bb A_X:=\bb A(S_n^\sub{red}(X),\roi_X)=\rprod_{\xi\in S_n^\sub{red}(X)}\roi_{X,\xi}^\comp\] is a restricted product of (finite products of) $n$-dimensional local fields. We will soon see in section \ref{subsection_adeles_and_cohomology} that $\bb A_X$ is the top degree part of a semi-cosimplicial group of ad\`eles. Moreover, proposition \ref{proposition_one_dim_adeles} will show that if $X$ is the spectrum of the ring of integers of a number field, or a curve over a finite field, then $\bb A_X$ is the usual ring of ad\`eles.

\subsection{Ad\`eles compute cohomology}\label{subsection_adeles_and_cohomology}
In this section we describe, without proof, the main theorem of higher dimensional ad\`eles: they provide a flasque resolution of the quasi-coherent sheaf $M$ and thus may be used to compute its cohomology.

We will use the language of semi-cosimplicial sets. Let $\Delta^\sub{red}$ denote the reduced simplex category: its objects are the sets $\ul n=\{0,1,\dots,n\}$, for each $n\ge 0$, and its morphisms are the strictly increasing maps. A {\em semi-cosimplicial group} $A_\bullet$ is a functor from $\Delta^\sub{red}$ to the category of abelian groups. Taking alternating sums of coface maps then gives rise to a complex \[0\to A_0\to A_1\to A_2\to\cdots,\] whose cohomology is called the cohomology of $A_\bullet$. The reader unfamiliar with these constructions may wish to consult \cite[Chap.~8]{Weibel1994} for details.

Now let $X$ be our Noetherian scheme and $M\in QCoh(X)$. Let $\xi=(x_0,\dots,x_m)\in S_m^\sub{red}(X)$ and let $f:\ul n\to\ul m$ be a strictly increasing map; set \[f(\xi):=(x_{f(0)},\dots,x_{f(n)})\in S_n^\sub{red}(X).\] Using either the original definition of the ad\`eles, or using the description of $M_\xi^\comp$ provided by lemma \ref{lemma_local_factors}, one checks that there is a natural homomorphism \[M_{f(\xi)}^\comp\to M_\xi^\comp.\] Taking the product over all $\xi\in S_m^\sub{red}(X)$ shows that \[\ul m\mapsto \prod_{\xi\in S_m^\sub{red}(X)}M_\xi^\comp\] has the structure of a semi-cosimplicial group.

The first part of the main theorem states that this semi-cosimplicial structure restricts to the ad\`eles, while the second part states that the cohomology of the resulting semi-cosimplicial group is exactly $H^*(X,M)$:

\begin{theorem}\label{theorem_adeles}
Let $X$ be a Noetherian scheme and $M\in QCoh(X)$. Then the just-defined semi-cosimplicial structure on $\prod_{\xi\in S_\bullet^\sub{red}(X)}M_\xi^\comp$ restricts to a semi-cosimplicial structure on \[\bb A(S_\bullet^\sub{red}(X),M)=\rprod_{\xi\in S_\bullet^\sub{red}(X)}M_\xi^\comp\subseteq\prod_{\xi\in S_\bullet^\sub{red}(X)}M_\xi^\comp.\]

The cohomology of the resulting semi-cosimplicial ad\`ele group $\bb A(S_\bullet^\sub{red}(X),M)$ is canonically isomorphic to $H^*(X,M)$.
\end{theorem}
\begin{proof}
We sketch the proof; see \cite{Beilinson1980} for the original statement, and \cite{Huber1991} for the details.

To check that the coface maps on $\prod_{\xi\in S_\bullet^\sub{red}(X)}M_\xi^\comp$ restrict to the subgroups $\rprod_{\xi\in S_\bullet^\sub{red}(X)}M_\xi^\comp$, one works directly with the original definition of the ad\`eles to verify that certain finiteness conditions hold.

One then sheafifies the ad\`ele groups by defining \[\ul{\bb A}(T, M)(U):=\bb A_U(T|_U, M|_U),\] where $T|_U=\{(x_0,\dots,x_m)\in T:x_i\in U\mbox{ for }i=0,\dots,m\}$, and it is not hard to check, inductively using the recursive definition, that $\ul{\bb A}(T, M)$ is a flasque sheaf of abelian groups. 

Thus one obtains a coaugmented, semi-cosimplicial flasque sheaf \[M\to \ul{\bb A}(S_\bullet^\sub{red}(X),M),\] and one proves that this is a resolution of $M$. Taking global sections completes the proof.
\end{proof}

\begin{corollary}\label{corollary_adeles}
Let $X$ be a Noetherian scheme of dimension $n$, and $ M\in QCoh(X)$. Then there is a natural isomorphism \[H^*(X, M)\cong H^*(0\to\bb A(S_0^\sub{red}(X), M)\to\cdots\to \bb A(S_n^\sub{red}(X), M)\to 0),\] where the complex is the usual complex associated to the semi-cosimplicial abelian group $\bb A(S_{\bullet}^\sub{red}(X), M)$
\end{corollary}
\begin{proof}
This is just a restatement of the second part of the theorem, noting that if $m>n$ then $S_m^\sub{red}(X)=\emptyset$ and so $\bb A(S_m^\sub{red}(X),M)=0$.
\end{proof}

Before making the theorem and corollary much more explicit in low dimensions, we show that each ad\`ele group $\bb A(S_m^\sub{red}(X), M)$ breaks into a direct sum of smaller pieces:

\begin{lemma}
Let $X$ be a Noetherian scheme, let $T\subseteq S_m(X)$ be a set of chains, and suppose that $T$ admits a partition $T=T_1\sqcup T_2$. Then, for any $ M\in QCoh(X)$, the natural map \[\bb A(T, M)\to\bb A(T_1, M)\oplus\bb A(T_2, M)\] is an isomorphism.
\end{lemma}
\begin{proof}
This is easily checked by induction on $m$ using the original definition of $\bb A(T, M)$.
\end{proof}

Let $0\le I_0<\dots<I_m\le\dim X$ be a strictly increasing sequence of integers bounded above by the dimension of $X$. Then we will be interested in the set of chains given by \[\mathbf{I_0\dots I_m}:=\{(x_0,\dots,x_m)\in S_m(X):x\in X^{I_i}\mbox{ for }i=0,\dots,m\}\subseteq S_m^\sub{red}(X).\] In other words, we impose restrictions on the codimensions of the points appearing in the chains ($X^I$ denotes the codimension $I$ points of $X$).

\begin{corollary}
Let $X$ be a Noetherian scheme and let $ M\in QCoh(X)$. Then there is a natural isomorphism of abelian groups \[\bb A(S_m^\sub{red}(X), M)\cong\bigoplus_{0\le I_1<\cdots<I_m\le\dim X}\bb A(\mathbf{I_0\dots I_m}, M),\]
where the sum is taken over all possible increasing sequences $0\le I_1<\cdots<I_m\le\dim X$.
\end{corollary}
\begin{proof}
Immediate from the previous lemma since $S_m^\sub{red}(X)=\bigsqcup_{0\le I_1<\cdots<I_m\le\dim X}\mathbf{I_0\dots I_m}$.
\end{proof}

Combining the two previous corollaries, we see that the cohomology of $ M$ on $X$ can be calculated using a complex of ad\`ele groups of the form \[\bb A(\mathbf{I_0\dots I_m}, M)=\rprod_\xi M_\xi^\comp,\] where $\xi$ varies over all chains $(x_0,\dots,x_m)$ on $X$ satisfying $x\in X^{I_i}$ for $ i=0,\dots,m$.

In the next two sections we will explicitly describe all these groups for schemes of dimension one and two.

\subsection{Ad\`eles in dimension one}\label{section_one_dim_adeles}
Let $X$ be a one-dimensional, irreducible, Noetherian scheme, and let $ M\in Coh(X)$. We will now explicitly calculate the ad\`ele groups $\bb A(\mathbf{0}, M)$, $\bb A(\mathbf{1}, M)$, and $\bb A(\mathbf{01}, M)$:

\begin{description}
\item[(0)] First consider $\mathbf{0}=\{x\in S_0(X):x\in X^0\}=\{\eta\}$, where $\eta$ is the generic point of $X$. Then rule (A2) implies that \[\bb{A}(\mathbf{0},M)=\prod_{x\in \mathbf{0}}\hat{M_x}=M_\eta.\]
\item[(1)] Next consider $\mathbf{1}=\{x\in S_0(X):x\in X^1\}=X_0$, where $X_0$ is the set of closed points of $X$. Then rule (A2) implies that \[\bb{A}(\mathbf{1},M)=\prod_{x\in \mathbf{1}}\hat{M_x}=\prod_{x\in X_0}\hat{M_x}.\]
\item[(01)] Finally, consider \[\mathbf{01}=\{(x_0,x_1)\in S_1(X):x_0\in X^0, \,x_1\in X^1\}=\{(\eta,x):x\in X_0\}.\] To apply rule (A3), we calculate, for any $z\in X$, \[_z\mathbf{01}=\{x\in X:(z,x)\in\mathbf{01}\}=\begin{cases}\emptyset& z\in X_0,\\X_0& z=\eta,\end{cases}\] so that rule (A3) implies
\begin{align*}
\bb{A}(\mathbf{01},M)
	&=\projlim_r \bb{A}(\mathbf{1},\tilde j_{r\eta} M)\\
	&=\bb{A}(\mathbf{1},\tilde j_\eta M)
\end{align*}
where $\tilde j_\eta M$ is the constant sheaf on $X$ everywhere equal to $\cal{F}_\eta$.

Next, $\tilde j_\eta M=\indlim_DM(D)$ where $D$ varies over the effective divisors on $X$ and $M(D)=M\otimes_{\roi_X}\roi_X(D)$. Hence rule (A1) implies $\bb{A}(\mathbf{01},M)=\indlim_{D\ge 0} \bb{A}(\mathbf{1},M(D))$. But now we may apply the {\bf(1)} calculation to the sheaf $ M(D)$ to write $\bb{A}(\mathbf{1},M(D))=\prod_{x\in X_0}\hat{M(D)_x}$; this demonstrates the recursive way in which ad\`ele groups are constructed. The final conclusion is that \[\bb{A}(\mathbf{01},M)=\indlim_D\prod_{x\in X_0}\hat{M(D)_x}.\]
\end{description}

This finishes our calculation of the ad\`ele groups on $X$ for an arbitrary coherent sheaf. Their cosimplicial structure is reflected by the natural homomorphisms \[\bb A(\mathbf{0}, M)\stackrel{\bor^{\mathbf 0}}{\To}\bb A(\mathbf{01}, M) \stackrel{\bor^{\mathbf 1}}{\longleftarrow}\bb A(\mathbf{1}, M),\] and corollary \ref{corollary_adeles} states that \[H^*(X, M)\cong H^*(0\to \bb A(\mathbf{0}, M)\oplus \bb A(\mathbf{1}, M)\stackrel{\bor^{\mathbf 0}-\bor^{\mathbf 1}}{\To}\bb A(\mathbf{01}, M)\to 0),\] which the reader may wish to directly check without appealing to the general theory.

Now we suppose further that $X$ is regular (e.g.~the spectrum of the ring of integers of a number field, or a smooth curve over a field) and take $ M=\roi_X$; the following proposition shows that the ad\`ele groups reduce to the very familiar ring of ad\`eles, its global subring, and its integral subring:

\begin{proposition}\label{proposition_one_dim_adeles}
Suppose that $X$ is a one-dimensional, regular, irreducible scheme, and set $ M=\roi_X$. For $x\in X_0$, let $K(X)_x=\Frac\hat{\roi}_{X,x}$ be the complete discrete valuation field associated to the closed point $x$. Then \[\bb{A}(\mathbf{0},\roi_X)=K(X),\quad\bb{A}(\mathbf{1},\roi_X)=\prod_{x\in X_0}\hat{\roi_{X,x}},\] and \[\bb A(\mathbf{01},\roi_X)=\rprod_{x\in X_0}K(X)_x:=\{(f_x)\in\prod_xK(X)_x:f_x\in\hat{\roi_{X,x}}\mbox{ for all but finitely many }x\}\] is the usual ring of (finite) ad\`eles.

The homomorphisms $\bor^{\mathbf 0}$ and $\bor^{\mathbf 1}$ are respectively given by the diagonal embedding and natural inclusion \[K(X)\stackrel{\bor^{\mathbf 0}}{\To}\rprod_{x\in X_0}K(X)_x\stackrel{\bor^{\mathbf 0}}{\longleftarrow}\prod_{x\in X_0}\hat{\roi_{X,x}}.\] 
\end{proposition}
\begin{proof}
The only part of the proposition which does not follow at once from the results immediately above for an arbitrary one-dimensional scheme is the identification of $\bb A(\mathbf{01},\roi_X)$ with the usual ad\`ele ring $\rprod_{x\in X_0}K(X)_x$.

For each effective divisor $D$, the inclusion $\roi_X(D)_x\into K(X)$ induces an inclusion on the completions $\hat{\roi_X(D)_x}\into K(X)_x$. Passing to the limit, we may identify $\bb A(\mathbf{01},\roi_X)$ with a subgroup of $\prod_{x\in X_0}K(X)_x$; to see that this subgroup is the restricted product, identify divisors with Weil divisors and take the limit. We leave the details to the reader.
\end{proof}

\begin{remark}
In fact, when $X$ is one-dimensional, the ad\`ele group $\bb A({\mathbf 01},M)$ may always be described by an `all but finitely many' restriction. This will be used in the next section when some ad\`ele groups for two-dimensional schemes are built recursively from one-dimensional ad\`ele groups.
\end{remark}

\begin{remark}
Of course, if $X$ is the spectrum of the ring of integers of a number field, then $\rprod_{x\in X_0}K(X)_x$ is not the full ring of ad\`eles, but only the finite part. The problem of incorporating archimedean data into higher dimensional ad\`eles has not been completely satisfactorily solved, and we say nothing about it here.
\end{remark}

\subsection{Ad\`eles in dimension two}
Let $X$ be a two-dimensional, irreducible, Noetherian scheme. We will now explicitly describe all the ad\`ele groups of $X$, thereby recovering the original, direct definitions given by Parshin of two-dimensional ad\`eles.

For simplicity we assume that $X$ is normal in codimension one. This has two useful (but unnecessary) consequences: firstly, each local ring $\roi_{X,y}$, for $y\in X^1$, is a discrete valuation ring; secondly, the function field of $X$, as a constant sheaf, is equal to $\indlim_{D\ge 0}\roi_X(D)$, where $D$ varies over all effective divisors.

Let $\eta$ denote the generic point of $X$ and $F$ its function field; codimension $2$ points will always be denoted $x\in X^2$, and codimension $1$ points will be denoted $y\in X^1$. To simplify our descriptions of the ad\`ele groups, we introduce the following standard pieces of notation for the local factors arising in the two-dimensional ad\`eles (the swapping of the order of $x$ and $y$ in some of the notation is an unfortunate historical convention arising from):

Firstly, `$x\in y$' always means that $x\in X^2$, $y\in X^1$, and $x\in \Cl y$ (i.e.~$(y,x)\in\mathbf{12}$). Then we have:
\begin{align*}
\roi_x&:=\roi_{X,x}^\comp=\bb A(x,\roi_X)=\hat{\roi_{X,x}}\quad\mbox{(a two-dimensional, complete local ring)}\\
F_x&:=\roi_{X,(\eta,x)}^\comp=\bb A((\eta,x),\roi_X)=\roi_x\otimes_{\roi_{X,x}}F\quad\mbox{(not a field!)}\\
\roi_y&:=\roi_{X,y}^\comp=\bb A(y,\roi_X)=\hat{\roi_{X,y}}\quad\mbox{(a complete discrete valuation ring)}\\
F_y&:=\roi_{X,(\eta,y)}^\comp=\bb A((\eta,y),\roi_X)=\roi_y\otimes_{\roi_{X,y}}F=\Frac\roi_y\quad\mbox{(a complete discrete valuation field)}\\
\roi_{x,y}&:=\roi_{X,(y,x)}^\comp=\bb A((y,x),\roi_X)\quad\mbox{(a finite product of one-dimensional, complete local rings; see below)}\\
F_{x,y}&:=\roi_{X,(\eta,y,x)}^\comp=\bb A((\eta,y,x),\roi_X)=\roi_{x,y}\otimes_{\roi_{X,y}}F=\Frac \roi_{x,y}\quad\mbox{(a finite product of fields of cdvdim $\ge 2$)}
\end{align*}
The ring $\roi_{x,y}$ requires some further attention: if $\frak p\subset\roi_{X,x}$ denotes the local equation of $y$ at $x$, then \[\roi_{x,y}=\hat{(\roi_x)_{\frak p\roi_x}}.\] i.e.~Localise $\roi_x$ away from the pushforward of $\frak p$, to get a one-dimensional, semi-local ring, and then complete at the Jacobson radical; compare with example \ref{example_low_dimensions2}(iii). We will denote the Jacobson radical of $\roi_{x,y}$ by $\frak p_{x,y}$; since we assumed $X$ to be normal in codimension one, $\frak p_{x,y}$ is principal, generated by a local parameter at $y$.

Now we are equipped to describe the ad\`ele groups of $X$:
\begin{proposition}\label{proposition_two_dim_adeles}
Let $X$ be a two-dimensional, irreducible, Noetherian scheme, normal in codimension one. Then the ad\`ele groups of $\roi_X$ are as follows:
\[
\bb{A}(\mathbf{0},\roi_X)=F,\quad
\bb{A}(\mathbf{1},\roi_X)=\prod_{y\in X^1}\roi_y,\quad
\bb{A}(\mathbf{2},\roi)=\prod_{x\in X^2}\roi_x.
\]
\[
\bb A(\mathbf{01},\roi_X)=\rprod_{y\in X^1}F_y=\bigg\{(f_y)\in\prod_{y\in X^1}F_y:f_y\in\roi_y\mbox{ for all but finitely many }y\bigg\}
\]
\begin{align*}
\bb{A}(\mathbf{02},\roi_X)=\rprod_{x\in X^2}F_x=\bigg\{(f_x)\in\prod_{x\in X^2}F_x:&\mbox{ there exists an effective divisor }D
	\\&\mbox{such that }f_x\in\roi_x\otimes_{\roi_{X,x}}\roi_X(D)_x\mbox{ for all }x\bigg\}
\end{align*}
\begin{align*}
\bb A(\mathbf{12},\roi_X)=\rprod_{(y,x)\in\mathbf{12}}\roi_{x,y}=\bigg\{(f_{x,y})&\in\prod_{(y,x)\in\mathbf{12}}\roi_{x,y}: \mbox{ for each $y\in X^1$ and each }r\ge 0,\\ &f_{x,y} \in\roi_x+\frak p_{x,y}^r \mbox{ for all but finitely many }x\in y\bigg\}
\end{align*}
\begin{align*}
\bb A(\mathbf{012},\roi_X)=\rprod_{(y,x)\in\mathbf{12}}F_{x,y}=\bigg\{(f_{x,y})&\in\prod_{(y,x)\in\mathbf{12}}F_{x,y}:
\mbox{(1): there exists an effective divisor $D$ }\\ 
& \mbox{such that }f_{x,y}\in\roi_{x,y}\otimes_{\roi_{X,y}}\roi_X(D)_y\mbox{ for all }(y,x),\\& \mbox{and (2): for each $y\in X^1$ and each }r\in\bb Z,\\ &f_{x,y} \in\roi_x+\frak p_{x,y}^r \mbox{ for all but finitely many }x\in y\bigg\}
\end{align*}
\end{proposition}
\begin{proof}
As we did in dimension one, we will calculate the ad\`ele groups for an arbitrary coherent sheaf $M$ and then the results follow by specialising to $M=\roi_X$.

{\bf (0,1,2):} Firstly, rule (A2) implies that \[\bb{A}(\mathbf{0},M)=M_\eta,\quad \bb{A}(\mathbf{1},M)=\prod_{y\in X^1}\hat{M_y},\quad \bb{A}(\mathbf{2},M)=\prod_{x\in X^2}\hat{M_x}.\] We remark also that $\hat{M_y}=M_y\otimes_{\roi_{X,y}}\roi_y$ and $\hat{M_x}=M_x\otimes_{\roi_{X,x}}\roi_x$.

{\bf(01):} The calculation of the {\bf01} ad\`ele group for a one-dimensional scheme generalises to show that \[\bb{A}(\mathbf{01},\cal M)=\rprod_{y\in X^1}M_y\otimes_{\roi_{X,y}}F_y\] where the product is restricted by only taking those $(f_y)\in\prod_{y\in X^1}M_y\otimes_{\roi_{X,y}}F_y$ for which $f_y$ is in the image of $\hat{M_y}$ for all but finitely many $y\in X^1$.

{\bf(02)} Similarly to the {\bf(01)} group,
\begin{align*}
\bb{A}(\mathbf{02},M)
	&=\indlim_{D\ge 0} {\prod_{x\in X^2}}\hat{M(D)_x}\\
	&=\rprod_{x\in X^2}M_x\otimes_{\roi_{X,x}}F_x
\end{align*}
where the product is restricted by only taking those $(f_x)\in\prod_{x\in X^2}M_x\otimes_{\roi_{X,x}}F_x$ for which there exists an effective divisor $D$ such that $f_x$ is in the image of $\hat{M(D)_x}$ for all $x\in X^2$.

{\bf(12):} Now consider $\mathbf{12}=\{(y,x):y\in X^1, x\in X^2, x\in\Cl{y}\}$. To calculate $\bb A(\mathbf{12},M)$ we will appeal to a general result on ad\`ele groups which can be proved relatively easily in the usual inductive fashion:
\begin{quote}
If $X$ is a Noetherian scheme, $M\in Coh(X)$, and $T\subseteq S_m(X)$, then \[\bb A(T,M)=\prod_{z\in X}\projlim_r\bb A_{r\res z}(\{z\}\times{_zT},j_{r\res z}^*M),\] where the ad\`ele groups on the right side are for the scheme $r\res z:=\Spec\roi_X/\cal I(z)^r$ ($\cal I(z)$ denotes the ideal subsheaf of $\roi_X$ defining $\Cl z$) and $j_{r\res z}:r\res x\to X$ denotes the natural inclusion.
\end{quote}

Applying this in our setting yields the identity \[\bb A(\mathbf{12},M)=\prod_{y\in X^1}\projlim_r\bb A_{r\res z}(\mathbf{01},j_{r\res z}^*M).\] But $r\res z$ is a one-dimensional scheme so we may appeal to the calculations of section \ref{section_one_dim_adeles} to deduce that
\begin{align*}
\bb A_{r\res z}(\mathbf{01},j_{r\res z}^*M)
	&=\rprod_{x\in y}M_x\otimes_{\roi_{X,x}}\roi_{x,y}/\frak p_{x,y}^r\\
	&=\bigg\{(f_x)\in\prod_{x\in y}M_x\otimes_{\roi_{X,x}}\roi_{x,y}/\frak p_{x,y}^r: f_x\mbox{ is in the image of} \\ 
	&\phantom{=\bigg\{(f_x)\in\prod_{x\in y}M_x\otimes_{\roi_{X,x}}\roi_{x,y}/\frak p_{x,y}^r}M_x\otimes_{\roi_{X,x}}\roi_x\mbox{ for all but finitely many }x\bigg\}
\end{align*}
where `$x\in y$' really means `$x\in X^2\cap\Cl y$'. Taking $\prod_{y\in X^1}\projlim_r$ of these restricted products gives us $\bb A(\mathbf{12},M)$.

{\bf(012)} We finally calculate $\bb{A}(\mathbf{012}, M)$. Since $_z\mathbf{012}=\mathbf{12}$ if $z=\eta$ and $=\emptyset$ if $z$ is any other point, rule (A3) implies \[\bb{A}(\mathbf{012},M)=\bb{A}(\mathbf{12},\tilde j_\eta M),\] where $\tilde j_\eta M$ is the constant $M_\eta$ sheaf on $X$. Next, $\tilde j_\eta M=\indlim_{D\ge 0} M(D)$ and so rule (A1) implies \[\bb{A}(\mathbf{012},M)=\indlim_{D\ge 0}\bb{A}(\mathbf{12},M(D)),\] where each of the terms on the right has already just been analysed (replacing $M$ by $M(D)$). This completes our explicit analysis of the ad\`ele groups.
\comment{
write the constant sheaf $\tilde{\cal{F}_{\eta}}$ as a direct limit $\displaystyle\tilde{\cal{F}_{\eta}}=\lim_{\To} \cal{F}(D)$ over effective divisors $D$; since the ad\`elic construction commutes with direct limits, we may apply (iv) to see that \[\bb{A}(\mathbf{012},\cal{F})=\indlim_{D\ge 0}\left(\prod_{y\in X^1}{\prod_{x\in X_0\cap\Cl{y}}}^{\!\!\prime}\cal{F}(D)_x\otimes_{\roi_{X,x}}\roi_{x,y}\right),\] which may be interpreted as a restricted product \[\prod_{y\in X^1}{\prod_{x\in X_0\cap\Cl{y}}}^{\!\!\prime}\cal{F}_x\otimes_{\roi_{X,x}}F_{x,y}.\]}
\end{proof}

We continue this example of a two-dimensional scheme by explaining theorem \ref{theorem_adeles} and its corollary, working with $M=\roi_X$ for concreteness, though the discussion holds in general.

There is a large commutative diagram of products of our local factors
\[\xymatrixcolsep{1cm}\xymatrixrowsep{1cm}\xymatrix{
& 	& 	F\ar[dl]\ar[dr]&	&	\\
&	\prod_{y\in X^1}F_y\ar[r] & \prod_{(y,x)\in\mathbf{12}}F_{x,y} & \prod_{x\in X^2}F_x\ar[l] & \\
\prod_{y\in X^1}\roi_y\ar[rr]\ar[ur] & & \prod_{(y,x)\in\mathbf{12}}\roi_{x,y}\ar[u] & & \prod_{x\in X^2}\roi_x\ar[ll]\ar[ul]
}\]
where the three ascending arrows are the obvious inclusions and the remaining arrows are diagonal embeddings. Theorem \ref{theorem_adeles} states that these homomorphisms restrict to the ad\`ele groups to give a commutative diagram of ring homomorphisms:
\[\xymatrixcolsep{1.5cm}\xymatrixrowsep{1.5cm}\xymatrix{
& 	& 	\bb{A}(\mathbf{0},\roi_X)\ar[d]\ar[dl]_{\bor^{\mathbf{0}}_{\mathbf{01}}}\ar[dr]^{\bor^{\mathbf{0}}_{\mathbf{02}}}&	&	\\
&	\bb{A}(\mathbf{01},\roi_X)\ar[r]^{\bor^{\mathbf{01}}_{\mathbf{012}}} & \bb{A}(\mathbf{012},\roi_X) & \bb{A}(\mathbf{02},\roi_X)\ar[l]_{\bor^{\mathbf{02}}_{\mathbf{012}}} & \\
\bb{A}(\mathbf{1},\roi_X)\ar[urr]\ar[rr]_{\bor^{\mathbf{1}}_{\mathbf{12}}}\ar[ur]^{\bor^{\mathbf{1}}_{\mathbf{01}}} & & \bb{A}(\mathbf{12},\roi_X)\ar[u]_{\bor^{\mathbf{12}}_{\mathbf{012}}} & & \bb{A}(\mathbf{2},\roi_X)\ar[ull]\ar[ll]^{\bor^{\mathbf{2}}_{\mathbf{12}}}\ar[ul]_{\bor^{\mathbf{2}}_{\mathbf{02}}}
}\] For example, to define the diagonal map \[\bor^{\mathbf{1}}_{\mathbf{12}}:\bb{A}(\mathbf{1},\roi_X)=\prod_{y\in X^1}\roi_y\To \bb{A}(\mathbf{12},\roi_y)=\rprod_{(x,y)\in \mathbf{12}}\roi_{x,y},\] one must check that if $y\in X^1$, $f\in\roi_y$, and $r\ge 0$, then $f\in\roi_x+\frak p_{x,y}^r$ for all but finitely many $x\in y$. These are the types of `finiteness' conditions to which we alluded when sketching the proof of theorem \ref{theorem_adeles}.

The second part of theorem \ref{theorem_adeles} states that the resulting complex
\[\xymatrixcolsep{0.1cm}\xymatrixrowsep{0.4cm}\xymatrix{
\bb{A}(\mathbf{0},\roi_X)\oplus\bb{A}(\mathbf{1},\roi_X)\oplus\bb{A}(\mathbf{2},\roi_X)\ar[r]&
\bb{A}(\mathbf{01},\roi_X)\oplus\bb{A}(\mathbf{02},\roi_X)\oplus\bb{A}(\mathbf{12},\roi_X)\ar[r]&
\bb{A}(\mathbf{012},\roi_X)\\
(f_0,f_1,f_2)\ar[r]& (\bor^{\mathbf{0}}_{\mathbf{01}}f_0+\bor^{\mathbf{1}}_{\mathbf{01}}f_1, -\bor^{\mathbf{0}}_{\mathbf{02}}f_0+\bor^{\mathbf{2}}_{\mathbf{02}}f_2,-\bor^{\mathbf{1}}_{\mathbf{12}}f_1-\bor^{\mathbf{2}}_{\mathbf{12}}f_2) &\\
&(g_{01},g_{02},g_{12})\ar[r]&\bor^{\mathbf{01}}_{\mathbf{012}}g_{01}+\bor^{\mathbf{02}}_{\mathbf{012}}g_{02}+\bor^{\mathbf{12}}_{\mathbf{012}}g_{12} \\}\]\marginpar{Add zeros and restructure}
has cohomology naturally isomorphic to $H^*(X,\roi_X)$ (and we stress again that we could have replaced $\roi_X$ by any other quasi-coherent sheaf in this discussion).

This completes our dissection of the general adelic theory in dimension two.

\bibliographystyle{acm}
\bibliography{../Bibliography}

\noindent Matthew Morrow,\\
University of Chicago,\\
5734 S. University Ave.,\\
Chicago,\\
IL, 60637,\\
USA\\
{\tt mmorrow@math.uchicago.edu}\\
\url{http://math.uchicago.edu/~mmorrow/}\\
\end{document}